\documentclass[11pt]{article}
\usepackage[inner=3.4 cm,outer=3.4 cm,top=3.0 cm,bottom=2.0 cm]{geometry}

\usepackage{amsmath,amssymb,amsfonts,amsthm}
\usepackage{setspace}
\newcounter{item}[section]
\newcounter{kirshr}
\newcounter{kirsha}
\newcounter{kirshb}
\newenvironment{enumroman}{\setcounter{kirshr}{1}
\begin{list}{(\roman{kirshr})}{\usecounter{kirshr}} }{\end{list}}
\newenvironment{enumarab}{\setcounter{kirshb}{1}
\begin{list}{(\arabic{kirshb})}{\usecounter{kirshb}} }{\end{list}}
\newenvironment{athm}[1]{\vskip3mm\par\noindent%\stepcounter{item}
{\bf #1 }. \slshape }
{\upshape\par\vskip10pt minus3pt}
\newtheorem{theorem}{Theorem}[section]

\newtheorem{lemma}[theorem]{Lemma}
\newtheorem{corollary}[theorem]{Corollary}
\newenvironment{demo}[1]{\noindent{\bf #1.}\upshape\mdseries}
{\nopagebreak{\hfill\rule{2mm}{2mm}\nopagebreak}\par\normalfont}
\theoremstyle{definition}
\newtheorem{remark}[theorem]{Remark}

\newtheorem{example}[theorem]{Example}
\newtheorem{definition}[theorem]{Definition}
\newtheorem{proposition}[theorem]{Proposition}
\def\R{\mathbb{R}}
\def\Q{\mathbb{Q}}
\def\A{{\mathfrak{A}}}
\def\At{{\sf At}}
\def\B{{\mathfrak{B}}}
\def\C{{\mathfrak{C}}}
\def\Ca{{\mathfrak Ca}}
\def\CA{{\sf CA}}
\def\Cm{{\sf Cm}}
\def\CRCA{{\sf CRCA}}
\def\de{Dedekind-MacNeille}
\def\ef{Ehren\-feucht--Fra\"\i ss\'e}
\def\F{{\mathfrak{F}}}
\def\Fm{{\mathfrak{Fm}}}
\def\Fr{{\mathfrak{Fr}}}
\def\g{{\sf g}}
\def\Id{{\sf Id}}
\def\K{{\sf K}}
\def\Lf{{\sf Lf}}
\def\Mo{{\sf M}}
\def\N{{\mathbb{N}}}
\def\nodes{{\sf nodes}}
\def\Nr{{\mathfrak{Nr}}}
\def\Nrr{{\mathfrak{Nr}}}
\def\pa{$\forall$}

\def\pe{$\exists$}

\def\QEA{{\sf QEA}}

\def\Nr{{\sf{Nr}}}

\def\A{{\mathfrak{A}}}
\def\At{{\sf At}}
\def\B{{\mathfrak{B}}}
\def\Bb{{\mathfrak{Bb}}}
\def\C{{\mathfrak{C}}}
\def\Ca{{\mathfrak Ca}}
\def\CA{{\sf CA}}
\def\Cm{{\sf Cm}}
\def\CRCA{{\sf CRCA}}
\def\de{Dedekind-MacNeille}
\def\ef{Ehren\-feucht--Fra\"\i ss\'e}
\def\F{{\mathfrak{F}}}
\def\Fm{{\mathfrak{Fm}}}
\def\Fr{{\mathfrak{Fr}}}
\def\g{{\sf g}}
\def\Id{{\sf Id}}
\def\K{{\sf K}}
\def\Lf{{\sf Lf}}
\def\Mo{{\sf M}}
\def\N{{\mathbb{N}}}
\def\nodes{{\sf nodes}}
\def\Nr{{\mathfrak{Nr}}}
\def\Nrr{{\mathfrak{Nr}}}
\def\pa{$\forall$}

\def\pe{$\exists$}

\def\QEA{{\sf QEA}}
\def\r{{\sf r}}
\def\R{\mathfrak{R}}
\def\Ra{{\sf Ra}}
\def\RCA{{\sf RCA}}
\def\Rd{{\mathfrak{Rd}}}
\def\Rl{{\mathfrak Rl}}
\def\rng{{\sf rng}}

\def\s{{\sf s}}
\def\Sc{{\sf Sc}}
\def\Sg{{\mathfrak Sg}}
\def\Tm{{\mathfrak{Tm}}}
\def\tp{{\sf tp}}
\def\VT{{\sf VT}}
\def\w{{\sf w}}
\def\ws{winning strategy}
\def\y{{\sf y}}
\def\Z{{\mathbb{Z}}}
\def\D{{\mathfrak{D}}}

\def\Ra{{\sf Ra}}
\def\set#1{ \{#1\}}
\def\Mo{{\sf Mo}}
\def\restr #1{{\restriction_{#1}}}
\def\ws{winning strategy}

\def\d{{\sf d}}
 \def\CA{{\sf CA}}
\def\B{{\sf B}}
\def\G{{\sf G}}
\def\w{{\sf w}}
\def\y{{\sf y}}
\def\g{{\sf g}}

\def\r{{\sf r}}
\def\K{{\sf K}}

\def\pa{$\forall$}
\def\pe{$\exists$}
\def\Sn{{\mathfrak{Sn}}}
\def\ef{Ehren\-feucht--Fra\"\i ss\'e}

\def\tp{{\sf tp}}

\def\A{{\mathfrak{A}}}
\def\At{{\sf At}}
\def\B{{\mathfrak{B}}}
\def\Bb{{\mathfrak{Bb}}}
\def\C{{\mathfrak{C}}}
\def\Ca{{\mathfrak Ca}}
\def\CA{{\sf CA}}
\def\Cm{{\sf Cm}}
\def\CRCA{{\sf CRCA}}
\def\de{Dedekind-MacNeille}
\def\ef{Ehren\-feucht--Fra\"\i ss\'e}
\def\F{{\mathfrak{F}}}
\def\Fm{{\mathfrak{Fm}}}
\def\Fr{{\mathfrak{Fr}}}
\def\g{{\sf g}}
\def\Id{{\sf Id}}
\def\K{{\sf K}}
\def\Lf{{\sf Lf}}
\def\Mo{{\sf M}}
\def\N{{\mathbb{N}}}
\def\nodes{{\sf nodes}}
\def\Nr{{\mathfrak{Nr}}}
\def\Nrr{{\mathfrak{Nr}}}
\def\pa{$\forall$}

\def\pe{$\exists$}

\def\QEA{{\sf QEA}}
\def\r{{\sf r}}
\def\R{\mathfrak{R}}
\def\Ra{{\sf Ra}}
\def\RCA{{\sf RCA}}
\def\Rd{{\mathfrak{Rd}}}
\def\Rl{{\mathfrak Rl}}
\def\rng{{\sf rng}}

\def\s{{\sf s}}
\def\Sc{{\sf Sc}}
\def\Sg{{\mathfrak Sg}}
\def\Tm{{\mathfrak{Tm}}}
\def\tp{{\sf tp}}
\def\VT{{\sf VT}}
\def\w{{\sf w}}
\def\ws{winning strategy}
\def\y{{\sf y}}
\def\Z{{\mathbb{Z}}}
\def\D{{\mathfrak{D}}}

\def\Ra{{\sf Ra}}
\def\set#1{ \{#1\}}
\def\Mo{{\sf Mo}}
\def\restr #1{{\restriction_{#1}}}
\def\ws{winning strategy}

\def\d{{\sf d}}
 \def\CA{{\sf CA}}
\def\B{{\sf B}}
\def\G{{\sf G}}
\def\w{{\sf w}}
\def\y{{\sf y}}
\def\g{{\sf g}}

\def\r{{\sf r}}
\def\K{{\sf K}}

\def\pa{$\forall$}
\def\pe{$\exists$}
\def\Sn{{\mathfrak{Sn}}}
\def\ef{Ehren\-feucht--Fra\"\i ss\'e}

\def\tp{{\sf tp}}

\def\M{{\mathfrak{M}}}

\def\Ca{{\mathfrak{Ca}}}

\def\M{{\mathfrak{M}}}

\def\A{{\mathfrak{A}}}
\def\B{{\mathfrak{B}}}
\def\C{{\mathfrak{C}}}
\def\D{{\mathfrak{D}}}
\def\Ig{{\mathfrak{Ig}}}
\def\M{{\mathfrak{M}}}
\def\dom{{\sf dom}}
\def\rng{{\sf rng}}
\def\Rd{{\sf{Rd}}}
\def\At{{\sf At}}
\def\Ra{{\mathfrak{Ra}}}
\def\Tm{{\mathfrak{Tm}}}
\def\Cm{{\mathfrak{Cm}}}
\def\E{{\mathfrak{E}}}

\def\VT{{\sf VT}}

\def\ef{Ehren\-feucht--Fra\"\i ss\'e}

\def\Id{{\sf Id}}
\def\Rl{{\mathfrak{Rl}}}
\def\F{{\mathfrak{F}}}
\def\Sc{{\mathfrak{Sc}}}
\def\QEA{{\sf QEA}}

\def\RCA{{\sf RCA}}
\def\QEA{{\bf PEA}}

\def\R{{\sf R}}
\def\L{{\sf L}}
\def\t{{\sf t}}
\def\t{{\sf t}}
\def\si{{i_0, i_1, \cdots, i_k}}
\def\sj{{j_0, j_1, \cdots, j_k}}

\def\ttr{{(\t^{j_0}_{i_0}\cdots \t^{j_k}_{i_k})^{(\mathfrak{Rd}_{pt}\A)}}}

\def\U{\mathfrak{U}}
\def\map{{\sf map}}
\def\Df{{\sf Df}}

\def\Rd{{\mathfrak{Rd}}}

\def\s{{\sf s}}
\def\Sc{{\sf Sc}}
\def\nodes{{\sf nodes}}
\def\d{{\sf d}}

\def\si{{i_0, i_1, \cdots, i_k}}
\def\sj{{j_0, j_1, \cdots, j_k}}

\def\ttr{{(\t^{j_0}_{i_0}\cdots \t^{j_k}_{i_k})}}

\def\QEA{{\sf PEA}}

\def\cyl#1{{\sf c}_{#1}}

\def\cyl#1{{\sf c}_{#1}}

\def\diag#1#2{{\sf d}_{#1#2}}

\def\V{{\sf V}}
\def\de{Dedekind-MacNeille}

\def\Z{{\mathbb{Z}}}
\def\Lf{{\sf Lf}}
\def\Bb{{\sf Bb}}
\def\Cs{{\sf Cs}}
\def\Ra{{\sf Ra}}
\def\Mo{{\sf M}}
\def\QEA{{\sf QEA}}

\def\RA{{\sf RA}}
\def\CRCA{{\sf CRCA}}
\def\G{{\sf G}}

\def\Bb{\mathfrak{Bb}}

\def\Nrr{\mathfrak{{Nr}}}
\def\w{{\sf w}}
\def\g{{\sf g}}
\def\y{{\sf y}}
\def\r{{\sf r}}
\def\L{\mathfrak{L}}
\def\G{{\cal G}}

\def\R{\mathfrak{R}}
\def\Nr{{\sf Nr}}

\def\Q{\mathbb{Q}}
\def\C{{\mathfrak{C}}}
\def\Fm{{\mathfrak{Fm}}}

\def\At{{\sf At}}
\def\N{{\cal N}}

\def\Nr{{\mathfrak{Nr}}}
\def\Fr{{\mathfrak{Fr}}}
\def\Sg{{\mathfrak{Sg}}}
\def\Fm{{\mathfrak{Fm}}}
\def\A{{\mathfrak{A}}}
\def\B{{\mathfrak{B}}}
\def\C{{\mathfrak{C}}}
\def\D{{\mathfrak{D}}}
\def\M{{\mathfrak{M}}}
\def\Sn{{\mathfrak{Sn}}}
\def\CA{{\bf CA}}

\def\QEA{{\bf QEA}}

\def\Df{{\bf Df}}

\def\Lf{{\bf Lf}}

\def\K{{\bf K}}
\def\K{{\bf K}}

\def\RCA{{\bf RCA}}

\def\Rd{{\mathfrak{Rd}}}
\def\(R)RA{{\bf (R)RA}}
\def\RA{{\bf RA}}

\def\R{\mathbb{R}}
\def\Q{\mathbb{Q}}

\def\Sc{{\bf Sc}}

\def\Cs{{\bf Cs}}

 \def\CA{{\sf CA}}
\def\B{{\sf B}}
\def\G{{\sf G}}
\def\w{{\sf w}}
\def\y{{\sf y}}
\def\g{{\sf g}}

\def\r{{\sf r}}
\def\K{{\sf K}}
 \def\Cm{{\mathfrak{Cm}}}
\def\Nr{{\mathfrak{Nr}}}

\def\restr #1{{\restriction_{#1}}}
\def\cyl#1{{\sf c}_{#1}}
\def\diag#1#2{{\sf d}_{#1#2}}

\def\si{{i_0, i_1, \cdots, i_k}}
\def\sj{{j_0, j_1, \cdots, j_k}}

\def\ttr{{(\t^{j_0}_{i_0}\cdots \t^{j_k}_{i_k})}}

\def\t{{\sf t}}
\def\map{{\sf map}}

\def\Rl{\mathfrak{Rl}}

\def\Ra{{\mathfrak{Ra}}}
\def\Ca{{\mathfrak{Ca}}}
\def\set#1{\{#1\} }
\def\Ra{{\mathfrak{Ra}}}
\def\Nr{{\mathfrak{Nr}}}
\def\Tm{{\mathfrak{Tm}}}
\def\A{{\mathfrak{A}}}
\def\B{{\mathfrak{B}}}
\def\C{{\mathfrak{C}}}
\def\D{{\mathfrak{D}}}
\def\E{{\mathfrak{E}}}
\def\Bb{{\mathfrak{Bb}}}
\def\Zd{{\mathfrak{Zd}}}
\def\CA{{\bf CA}}
\def\RA{{\bf RA}}

\def\RCA{{\bf RCA}}
\def\G{{\bf G}}

\def\L{{\mathfrak{L}}}
\def\R{\cal{R}}

\def\TeCA{{\sf TeCA}}
\def\ws{winning strategy}

\def\d{{\sf d}}

\def \set#1{\{#1\} }

\def\Nr{{\mathfrak{Nr}}}
\def\Fr{{\mathfrak{Fr}}}
\def\Sg{{\mathfrak{Sg}}}
\def\Fm{{\mathfrak{Fm}}}
\def\Rd{{\mathfrak{Rd}}}
\def\Ig{{\mathfrak{Ig}}}
\def\CA{{\bf CA}}
\def\RCA{{\bf RCA}}
\def\K{{\bf K}}
\def\L{{\bf L}}

\def\QEA{{\bf QEA}}
\def\(R)RA{{\bf (R)RA}}
\def\RA{{\bf RA}}

\def\R{\mathbb{R}}

\def\Nr{{\mathfrak{Nr}}}
\def\Fr{{\mathfrak{Fr}}}
\def\Sg{{\mathfrak{Sg}}}
\def\Fm{{\mathfrak{Fm}}}
\def\Rd{{\mathfrak{Rd}}}
\def\Ig{{\mathfrak{Ig}}}
\def\CA{{\bf CA}}
\def\RCA{{\bf RCA}}
\def\K{{\bf K}}
\def\L{{\bf L}}

\def\QEA{{\bf QEA}}

\def\(R)RA{{\bf (R)RA}}
\def\RA{{\bf RA}}

\def\R{\mathbb{R}}

\def\QEA{{\bf QEA}}

\def\R{\mathbb{R}}
\def\N{\mathbb{N}}
\def\Q{\mathbb{Q}}

\def\R{\mathbb{R}}
\def\Q{\mathbb{Q}}

\def\Sc{{\bf Sc}}

\def\Cs{{\bf Cs}}

 \def\CA{{\sf CA}}

\def\M{{\mathfrak{M}}}

\def\RA{{\bf RA}}

\def\At{{\mathfrak{At}}}

\def\G{{\mathfrak{G}}}

\def\K{{\bf K}}

\def\QEA{{\bf QEA}}

\def\tp{{\sf tp}}

\def\cyl#1{{\sf c}_{#1}}
\def\diag#1#2{{\sf d}_{#1#2}}

\def\s{{\sf s}}

\def\pa{$\forall$}
\def\pe{$\exists$}
\def\LCA{{\sf LCA}}
\def\ef{Ehren\-feucht--Fra\"\i ss\'e}

\def\TeCA{{\sf TeCA}}

\def\nodes{{\sf nodes}}

\def\restr #1{{\restriction_{#1}}}

\def\A{{\mathfrak{A}}}
\def\B{{\mathfrak{B}}}
\def\C{{\mathfrak{C}}}
\def\D{{\mathfrak{D}}}
\def\P{{\mathfrak{P}}}
\def\Fm{{\mathfrak{Fm}}}
\def\Ra{{\mathfrak{Ra}}}
\def\Nr{{\mathfrak{Nr}}}
\def\F{{\mathfrak{F}}}
\def\CA{{\bf CA}}
\def\RCA{{\bf RCA}}

\def\set#1{ \{#1\}}

\def\Ca{{\mathfrak Ca}}

\def\pe{$\exists$}
\def\pa{$\forall$}
\def\Cm{{\mathfrak Cm}}
\def\Sg{{\mathfrak Sg}}

\def\At{{\sf At}}
\def\Id{{\sf Id}}

\def\rng{{\sf rng}}
\def\dom{{\sf dom}}
\def\Fl{{\mathfrak{Fl}}}

\def\w{{\sf w}}
\def\g{{\sf g}}
\def\y{{\sf y}}
\def\r{{\sf r}}
\def\Co{{\sf Co}}
\def\cyl#1{{\sf c}_{#1}}

\def\diag#1#2{{\sf d}_{#1#2}}

\def\ws{winning strategy}
\def\ef{Ehren\-feucht--Fra\"\i ss\'e}

\def\Rl{\mathfrak{Rl}}

\def\y{{\sf y}}
\def\g{{\sf g}}
\def\r{{\sf r}}
\def\w{{\sf w}}

\def\Sc{{\sf Sc}}

\def\TCA{{\sf TCA}}
\def\CA{{\sf CA}}
\def\TDc{{\sf TDc}}
\def\R{{\sf R}}

\def\L{{\mathfrak{L}}}

\def\Z{{\mathbb{Z}}}
\def\K{{\sf K}}
\def\de{Dedekind-MacNeille}
\def\LCA{{\sf LCA}}

\def\Nr{{\sf Nr}}
\def\Nrr{{\mathfrak{Nr}}}
\def\M{{\sf M}}
\def\RCA{{\sf RCA}}

\def\QEA{{\sf QEA}}

\def\RA{{\sf RA}}
\def\Sg{\mathfrak{Sg}}

\makeindex             
\title{Reflecting space and time via Topological and Temporal cylindric algebras}
\author{Tarek Sayed Ahmed}
\date{}
\begin{document}
\maketitle
\begin{abstract}
Let $\alpha$ be an arbritary ordinal, and $2<n<\omega$. In \cite{3} accepted for publication in Quaestiones Mathematicae,  we studied using algebraic logic, 
interpolation, amalgamation  using $\alpha$ many variables for topological logic with $\alpha$ many variables briefly  $\sf TopL_{\alpha}$. 
This is a sequel to \cite{3}; the second  part on modal cylindric algebras, where we study algebraically other properties of $\sf TopL_{\alpha}$.
Modal cylindric algebras are cylindric algebras of infinite dimension expanded with unary modalities inheriting their semantics from a unimodal logic $\sf L$ such as $\sf K5$ or $\sf S4$.
Using the methodology of algebraic logic, we study topological (when $\sf L=S4$),  in symbols 
$\sf TCA_{\alpha}$.  
We study completeness and omitting types $\sf OTT$s for 
$\sf TopL_{\omega}$ and $\sf TenL_{\omega}$, by proving several representability results for locally finite such algebras. Furthermore, we study the notion of atom-canonicity  
for both  ${\sf TCA}_{n}$ and ${\sf TenL}_n$, a well known persistence property in modal logic, 
in connection to $\sf OTT$ for  ${\sf TopL}_n$ and ${\sf TeLCA}_n$, respectively.   We study representability,
omitting types, interpolation and complexity isssues (such as undecidability) for topological cylindric algebras.
In Part 2, we introduce temporal cyindric algebras 
and point out the way how to amalgamate algebras of space (topological algebars) and algebras of time (temporal algebras)  
forming topological-temporal cylindric algebras that lend themselves to  encompassing spacetime gemetries, 
in a purely algebraic fashon.
Having a geometric dimension literally, namely, that  of their cylindric reducts, the geometry of such algbras is conceptually distinct 
from the `standard tensor and manifolds' mathemtical model for general relativity tajen to be a part of algebraic geometry, raher than algebraic logic, which is the path 
introduced in the 2nd part of this paper. 
\end{abstract}

\section{Introduction and overview}
\subsection{Universal and algebraic logic}

One aim of universal logic is to determine the domain of validity of such and such metatheorem
(e.g. the completeness theorem, the Craig interpolation
theorem, or the Orey-Henkin omitting types theorem of first order logic) and to
give general formulations of metatheorems in broader, or even entirely other contexts.
This is also done in algebraic logic, by dealing with modifications and variants of first order logic resulting in a natural way
during the process of {\it algebraisation}, witness for example the omitting types theorem proved in  \cite{Sayed}.
This kind of investigation  is extremely potent  for applications and helps to make the distinction
between what is really essential to a particular logic and what is not.
During the 20th
century, numerous logics have been created, to mention only a few: intuitionistic logic, modal logic, topological logic, topological dynamic logic,
spatial logic,  dynamic logic,  tense logic,
temporal logic, many-valued logic, fuzzy logic, relevant logic,
para-consistent logic, non monotonic logic,
etc.
The rapid development of computer science, since the fifties of the 20th century initiated by work of giants like G\"odel, Church and Turing,
ultimately brought to
the front scene other logics as well, like  logics of programs
and lambda calculus (the last can be traced back to the work of Church).
Universal logic 
owes its birth  as a response to the
explosion of new logics
After a while  it became noticeable that certain patterns of
concepts kept being repeated albeit in different
logics. But then the time  was ripe to make in retrospect an inevitable abstraction
like as the case with the field of abstract model theory (Lindstrom's theorem is an example here).
Abstract algebraic logic on the other hand is  a major model theoretic trend of universal algebra that formalizes 
the intuitive notion of a logical system, including syntax, semantics, and the satisfaction relation between them.
and it developed to an  important foundational theory.
Universal logic addresses different logical systems simultaneously in essentially four ways.
Either abstracting common features, or building new bridges between them, or
constructing new logics from old ones, or, last but not least, combining logics.
In this paper we do all four things. 
We construct  new predicate logics that can be seen as 
modal expansions of 
first order logic given an algebraizable (in the standard Blok-Pigozzi sense) formalism. 
For such logics, in the first part of the paper,  we study properties known to hold for $L_{\omega, \omega}$ like the celebrated Orey-Henkin Omitting types Theorem 
We show that $\sf OTT$ remains to hold for modal expansions
of first order logic as long as there are variables existing {\it oustide} (atomic) formulas, but these results do not generalize any further. 
Our results cover ordinary predicate first order logic, possibly expanded with modalites. 
Our results apply to cylindric algebras when $\sf L=S5$, to basic Temporal cylindric algebras defined by Georgesco, and so called Tense algebras, that reflects time a the name might suggest, 
and to  so-called topological cylindric algebras, that expressed spatial properties of space in a modal simple setting obtained when 
$\sf L=S4$. In the last case, the modalities can be redefined topologically using the interior operator relative to some so-called Alexandrof 
topology defined on the base of a cylindric set algebra.

{\bf Topological logic:}
Topological logic provides a framework for studying the confluence of the topological semantics for
$\sf S4$ modalities, based on topological spaces
rather than Kripke frames, with the $\sf S4$
modality induced by the interior operator.
Motivated by questions like:  which spatial structures may be characterized by means of modal logic,
what is the logic of space, how to encode
in modal logic different geometric relations, 
topological logics are apt for dealing with {\it logic} and {\it space}.
The overall point is to take a common mathematical model of space (like a topological space)
and then to fashion logical tools to work with it.
One of the things which blatantly strikes one when studying elementary topology is that notions like open, closed, dense
are intuitively very transparent, and their
basic properties are absolutely straightforward to prove. However, topology
uses second order notions as it reasons with sets and subsets of `points'. This might suggest that like
second order logic, topology ought to be computationally very complex.
This apparent dichotomy between the two paradigms
vanishes when one realizes that a large portion of
topology can be formulated as a  simple modal logic,
namely, $\sf S4$.
This is for sure an asset
for modal logics tend to be much easier to handle than first order logic
let alone second order. 
We can summarize the above discussion in the following neat theorem, that we can and will
attribute to McKinsey, Tarski
and  Kripke; this historically is not very accurate. For a topological space $X$ and $\phi$ an $\sf S4$
formula we write $X\models \phi$, if $\phi$ is valid topologically in
$X$ (in either of the senses above). For example, $w\models \Box \phi$ $\iff$
for all  $w'$ if $w\leq w'$, then $w'\models \phi$, where $\leq$ is the relation $x\leq y$
$\iff$ $y\in {\sf cl}\{x\}$ where ${\sf cl}$ abbreviate `the closure of'.
\begin{theorem} (McKinsey-Tarski-Kripke)
Suppose that $X$ is a dense in itself metric space (every point is a limit point)
and $\phi$ is a modal $\sf S4$ formula. Then the following are equivalent:
\begin{enumerate}
\item $\phi\in \sf S4.$
\item$\models \phi.$
\item $X\models \phi.$
\item $\mathbb{R}\models \phi.$
\item $Y\models \phi$ for every finite topological space $Y.$
\item $Y\models \phi$ for every Alexandrov space $Y.$
\end{enumerate}
\end{theorem}
One can say that finite topological space or their natural extension to Alexandrov topological spaces reflect faithfully
the $\sf S4$ semantics, and that arbitrary topological spaces generalize $\sf S4$ frames. On the other hand,
every topological space gives rise to a normal modal logic. Indeed $\sf S4$ is the modal logic of $\mathbb{R}$, or
any metric that is  separable and dense in itself
space, or all topological spaces, as indicated above. Also $\sf S4$
is  the modal logic of the Cantor set, which is known to be Baire isomorphic to
$\mathbb{R}$.
But, on the other hand,  modal logic is too weak to detect interesting properties
of $\mathbb{R}$, for example it cannot distinguish between $[0, 1]$ and $\mathbb{R}$ despite
their topological dissimilarities, the most striking one being compactness; $[0, 1]$ is compact, but $\mathbb{R}$ is
not. However, when we step into the realm of the predicate topological logic, the expressive power becomes substantially stronger.

{\bf $\sf OTT$s  for topological predicate logic:}
It would seem to be a simple matter to outfit a modal logic with the quantifiers. One would simply add the standard rules for quantifiers to the principles of 
whichever propositional modal logic one chooses. 
However, adding quantifiers to modal logic involves a number of difficulties. 
The main points of disagreement concerning the quantifier rules are about how to handle the domain of quantification. 
The simplest alternative, the fixed-domain approach, assumes a single domain of quantification that contains all the possible objects. 
Another interpretation, the world-relative interpretation, assumes that the domain of quantification changes from world to world, 
and contains only the objects that actually exist in a given world. Each of these two alternatives has its pros and cons.
Here we adopt the fixed-domain approach which requires 
no major adjustments to the classical machinery for the quantifiers. Furthermore, we assume that this world carries an Alexandrov topology, inducing
infinitely many modalities whose semantics coincide with $\sf S4$.
The fixed-domain interpretation has advantages of simplicity and familiarity.
Other topological interpretations of propositional topological logic
were recently extended in a natural way to arbitrary theories of full first order logic by
Awodey and Kishida using so-called {\it topological pre-sheaves} to
interpret domains of quantification \cite{ak}.
They prove that $\sf S4\forall$ (predicate $\sf S4$ logic) is  complete with respect to such extended topological semantics, using
techniques related to recent work in topos theory. Indeed, historically Sheaf semantics was
first introduced by topoi theorists for higher order intuitionistic logic, and has been applied
to first order modal logic, by both modal and categorical logicians. Here, our syntax is similar to {\it op.cit} but completeness and $\sf OTT$ (the former deduced as a byproduct) 
is different than Sheaf semantics.

{\bf Topological, tense and Heyting polyadic algebras:} In this paper we investigate the $\sf OTT$ to topological predicate logic with $\alpha$ many variables, 
briefly  $\sf TopPL_{\alpha}$, where $\alpha$ is an ordinal. 
Georgescu \cite{g, g2, g3, g4, g5} applied algebraic logic outside the realm of first order logic. He applied the well developed theory of Halmos' theory 
of polyadic algebras to intuitionistic, modal, temporal  and topological logic.
Georgescu studied Chang, modal, topological and tense {\it locally  polyadic algebras} of infinite dimension, which are essentially equivalent to Tarski's locally finite cylindric 
algebras. 
The work of Georgescu in \cite{g5} is substantially generalized in \cite{Heyting} by relaxing the condition of local 
finiteness and studying besides representability, various amalgamation properties for Heyting polyadic algebras. The work in this paper, preceded with the results established 
in  \cite{Heyting},  can be seen as a far reaching 
generalization of the work in Georgescu's remaining aforementioned  references which dealt only with representation theorems of {\it locally finite} algebras. 
The work in \cite{Heyting} can be seen as yet another far reaching generalization of he work of Georgescu in the remaining aforementioned references. 
While in \cite{Heyting}, we studied interpolation for intuitionistic fragment of Keisler's logic topological predicate logic, here we go further by studying various forms of representations 
for various subclasses 
of topological and tense cylindric algebras of infinite dimension, reflecting the rich interplay between the syntax and semantics of predicate modal logics of space and time respectively; 
in the latter case dealing with basic temporal predicate 
logic. Furthermore, we formulate and prove an omitting types theorem for tense predicate logic.
We recover the Henkin-Orey $\sf OTT$ for ${\sf TopPl}_{\alpha}$ when $\alpha$ is an infinite countable ordinal for countable theories. We 
prove negative $\sf OTT$s for ${\sf TopPL}_n$  when $n$ is finite $>2$, and the types are required to be omitted  with respect to 
certain generalized semantics.
We address the case when a single non-principal type is required to be omitted  with respect to 
so-called  {\it $m$--square} models with $2<n<m<\omega$. 
An $m$-square modal is only 'locally square' with $m$ measuring the {\it degree} of squareness; it is only an approximation of an ordinary model which is $\omega$-square. 
The idea is that if we approach this $m$-square model using a movable window, then there will become a certain point, determined by $m$  where we will mistake 
this $m$-square model, for an ordinary genuine ($\omega$-square) model. For $2<n<l<k$ a $k$-square model is $l$ square, but the converse may fail, 
both are approximations to ordinary models, with the $k$-square one a better (closer) approximation. The larger the degree of squareness, the closer the model is to an ordinary one.
Figuratively speaking, the limit of this sequence of infinite locally relativized models is the $\omega$-square Tarskian models. 
These locally relativized models (representations) are invented by Hirsch and 
Hodkinson in the context of relation algebras \cite[Chapter 13]{HHbook} and is adapted to expansions of cyindric algebras here. 
Considering such 
{\it clique-guarded} semantics 
swiftly leads us to rich territory.

{\bf Topological, tense and Heyting polyadic algebras:} In this paper we investigate the $\sf OTT$ to topological predicate logic with $\alpha$ many variables, 
briefly  $\sf TopPL_{\alpha}$, where $\alpha$ is an ordinal. 
Georgescu \cite{g, g2, g3, g4, g5} applied algebraic logic outside the realm of first order logic. He applied the well developed theory of Halmos' theory 
of polyadic algebras to intuitionistic, modal, temporal  and topological logic.
Georgescu studied Chang, modal, topological and tense {\it locally  polyadic algebras} of infinite dimension, which are essentially equivalent to Tarski's locally finite cylindric 
algebras. 
The work of Georgescu in \cite{g5} is substantially generalized in \cite{Heyting} by relaxing the condition of local 
finiteness and studying besides representability, various amalgamation properties for Heyting polyadic algebras. The work in this paper, preceded with the results established 
in  \cite{Heyting},  can be seen as a far reaching 
generalization of the work in Georgescu's remaining aforementioned  references which dealt only with representation theorems of {\it locally finite} algebras. 
The work in \cite{Heyting} can be seen as yet another far reaching generalization of he work of Georgescu in the remaining aforementioned references. 
While in \cite{Heyting}, we studied interpolation for intuitionistic fragment of Keisler's logic topological predicate logic, here we go further by studying various forms of representations 
for various subclasses 
of topological and tense cylindric algebras of infinite dimension, reflecting the rich interplay between the syntax and semantics of predicate modal logics of space and time respectively; 
in the latter case dealing with basic temporal predicate 
logic. Furthermore, we formulate and prove an omitting types theorem for tense predicate logic.
We recover the Henkin-Orey $\sf OTT$ for ${\sf TopPl}_{\alpha}$ when $\alpha$ is an infinite countable ordinal for countable theories. We 
prove negative $\sf OTT$s for ${\sf TopPL}_n$  when $n$ is finite $>2$, and the types are required to be omitted  with respect to 
certain generalized semantics.
We address the case when a single non-principal type is required to be omitted  with respect to 
so-called  {\it $m$--square} models with $2<n<m<\omega$. 
An $m$-square modal is only 'locally square' with $m$ measuring the {\it degree} of squareness; it is only an approximation of an ordinary model which is $\omega$-square. 
The idea is that if we approach this $m$-square model using a movable window, then there will become a certain point, determined by $m$  where we will mistake 
this $m$-square model, for an ordinary genuine ($\omega$-square) model. For $2<n<l<k$ a $k$-square model is $l$ square, but the converse may fail, 
both are approximations to ordinary models, with the $k$-square one a better (closer) approximation. The larger the degree of squareness, the closer the model is to an ordinary one.
Figuratively speaking, the limit of this sequence of infinite locally relativized models is the $\omega$-square Tarskian models. 
These locally relativized models (representations) are invented by Hirsch and 
Hodkinson in the context of relation algebras \cite[Chapter 13]{HHbook} and is adapted to expansions of cyindric algebras here. 
Considering such 
{\it clique-guarded} semantics 
swiftly leads us to rich territory.

{\bf Results on cylindric algbras:} Fix finite $n>2$. Let ${\sf CRCA}_n$ denote the class of completely representable $\CA_n$s 
and ${\sf LCA}_n={\bf El}\CRCA_n$ be the class of algebras satisfying the Lyndon conditions.
For a class $\sf K$ of Boolean algebras with operators, let $\sf K\cap \bf At$ denote the class of atomic algebras in $\sf K$.  By modifying the games coding the Lyndon conditions 
allowing \pa\ to reuse the pebble pairs on the board, we will show 
that ${\sf LCA}_n={\bf El}\CRCA_n={\bf El}\bold S_c\Nr_n\CA_{\omega}\cap \bf At$. 
Define an $\A\in \CA_n$ to be {\it strongly representable} $\iff$ $\A$ is atomic and the complex algebra of its atom structure, equivalently its \de\ completion, in symbols $\Cm\At\A$ is in $\RCA_n$. 
This is a strong form of representability; of course $\A$ itself will be in $\RCA_n$, because $\A$ embeds into $\Cm\At\A$ and $\RCA_n$ is a variety, {\it a fortiori} closed under forming subalgebras.
We denote the class of strongly representable atomic algbras of dimension $n$ by ${\sf SRCA}_n$. Nevertheless,  there are atomic simple countable algebras that are representable, 
but not strongly representable. In fact, we shall see that there is a countable simple 
atomic algebra in $\RCA_n$ such that $\Cm\At\A\notin \bold S\Nr_n\CA_{t(n}(\supset \RCA_n)$. 
So in a way some algebras are more representable than others. In fact, the following inclusions are known to hold:  
$${\sf CRCA}_n\subsetneq {\sf LCA}_n\subsetneq {\sf SRCA}_n\subsetneq {\sf RCA}_n\cap \bf At.$$ 
In this paper we delve into a new notion, that of {\it degrees of representability}. Not all algebras are representable in the same way or strength. 
If $\C\subseteq \Nr_n\D$, with $\D\in \CA_m$ for some ordinal (possibly infinite) $m$,  we say that $\D$ is an $m$-dilation of $\C$ or simply a dilation if $m$ is clear from context. 
Using this jargon of 'dilating algebras' we  say that $\A\in {\sf RCA}_n$ is {\it strongly representable up to $m>n$} $\iff$ $\Cm\At\A\in \bold S\Nr_n\CA_m$.
This means that, though $\A$ itself is in $\RCA_n$,  the \de\ completion of $\A$ is not representable, but nevertheless it has some neat embedding property; it is `close' to bieng representable. 
Using this jargon, $\A$ admits a dilation of a bigger dimension. The bigger the dimension of the dilation of the representable algebra the more representable the algebra is, 
the closer it is to being strongly representable 
Through the unfolding of this paper, we will investigate and make precise the notion of an algebra being more representable than another.
It is known that ${\sf LCA}_n$ is an elementary class, but ${\sf SRCA}_n$ is not. 
We shall prove below that $\sf Str(\bold O\Nr_n\CA_{n+3}=\F: \Cm\F\in \bold O\Nr_n\CA_{n+3}\}$  
is not elementary with $\bold O\in \bold S_c, S_d, I$ as defind in the abstrcat.
We prove that any class $\bold K$, of $\CA_n$ atom structures obtained from $\bold K$ 
(Kripke frames) such that 
$\At\Nr_n\TCA_{\omega}\subseteq \bold K\subseteq  \bold \At\bold S_c\Nr_n\TCA_{n+3}$, $\bold K$ is not elementary and 
lifting from atom structures to atomic algebras, we show that any class $\sf K(\subseteq {RCA}_n)$ with
$\bold S_d\Nr_n\CA_{\omega}\cap \CRCA_n\subseteq \sf K\subseteq \bold S_c\Nr_n\CA_{n+3}$,  
$\sf K$ is not  elementary either. Here $\Nr_n$ is the operator of taking $n$ 
neat reducts defined similarly to cylindric algebras, 
while recall recall  $\bold S_c(\bold S_d)$ is the operation of forming complete (dense) subalgebras.

\section{ Predicate topological logic via expansions of cylindric algebras}

\subsection{Basic notions} 
Let $\alpha$ be an arbitrary ordinal $>0$ and $X$ be a set. Then 
$\B(X)$ denotes the Boolean set algebra $\langle \wp(X), \cup, \cap, \sim, \emptyset, X\rangle$.
Let $U$ be a non-empty set.
For $s,t\in {}^{\alpha}U$ write $s\equiv_i t$ if $s(j)=t(j)$ for all $j\neq i$.
For $X\subseteq {}^{\alpha}U$ and $i,j<\alpha,$ let
$${\sf c}_iX=\{s\in {}^{\alpha}U: (\exists t\in X) (t\equiv_i s)\}$$
and
$${\sf d}_{ij}=\{s\in {}^{\alpha}U: s_i=s_j\}.$$
The algebra $\langle \B(^{\alpha}U), {\sf c}_i, {\sf d}_{ij}\rangle_{i,j<\alpha}$ is called the {\it full cylindric set algebra of dimension $\alpha$
with unit (or greatest or top element) $^{\alpha}U$}. $^{\alpha}U$ is called a {\it Cartesian space}.
Instead of taking ordinary set algebras, 
as in the case of cylindric algebras, with units of the form $^{\alpha}U$, one 
may require {\it that the base $U$ is endowed with some topology}. This enriches the algebraic structure. 
For given such an algebra, 
for each $k<\alpha$, one defines an {\it interior operator} on $\wp(^{\alpha}U)$ by
$$I_k(X)=\{s\in {}^{\alpha}U; s_k\in {\sf int}\{a\in U: s_a^k\in X\}\}, X\subseteq {}^{\alpha}U.$$
Here $s_a^k$ is the sequence that agrees with $s$ except possibly at  $k$ where its value is $a$.
This gives a {\it topological cylindric set algebra of dimension $\alpha$}.

The interior operators, as well as the box operators can also be defined on {\it weak spaces}, 
that is, sets of sequences agreeing co-finitely with a given fixed sequence. This makes a difference only when $\alpha$ is infinite.

\begin{definition} A {\it weak space of dimension $\alpha$} is a set of the 
form $\{s\in {}^{\alpha}U: |\{i\in \alpha: s_i\neq p_i\}|<\omega\}$ 
for a given fixed in advance $p\in {}^{\alpha}U$.
Now for $k<\alpha$, define
$$I_k(X)=\{s\in {}^{\alpha}U^{(p)}: \{s_k\in {\sf int}\{u\in U: s_k^u\in X\}\}.$$
\end{definition}
But we can even go further. Such operations 
also extend to the class of representable algebras $\CA$s, briefly $\sf RCA_{\alpha}$. $\sf RCA_{\alpha}$ is defined to be the class
${\bf SP}\sf Cs_{\alpha}$. This class is also equal to ${\bf SP}\sf Ws_{\alpha}$, and 
it is known that $\sf RCA_{\alpha}$,  is a variety, hence closed under $\bf H$, 
though infinitely many schema of 
equations are required to axiomatize it.  
An algebra  in $\sf RCA_{\alpha}$ is isomorphic to a set algebra with universe $\wp(V)$;
the top element $V$ is a {\it generalized} space which is a set of the form $\bigcup_{i\in I}{}^{\alpha}U_i$, $I$ a set $U_i\neq \emptyset$ $(i\in I)$,
and  $U_i\cap U_j=\emptyset$ for $i\neq j$. The class of all such concrete algebras is denoted by $\sf Gs_{\alpha}$. We refer to $\A\in {\sf Gs}_{\alpha}$
as a {\it generalized set algebra of dimension $\alpha$}.
So let $\A\in {\sf RCA}_{\alpha}$, and assume that $\A\cong \B$ where  
$\B\in \sf Gs_{\alpha}$ 
has top element  the  generalized space $V$.
The base of $V$ is the set $U=\bigcup_{s\in V}\rng s.$
Then one defines the interior operator $I_k$ on $\B$ by:
$$I_k(X)=\{s\in V: s_k\in {\sf int}\{a\in U: s_a^k\in X\}\}, X\subseteq V.$$
and, for that matter  the box operator relative to a Chang system $V:U\to \wp(\wp(U))$ as follows
$$s\in \Box_k(X)\Longleftrightarrow \{a\in U: s_a^k\in X\}\in V(s_k), X\subseteq V.$$

The following lemma is very easy to prove, so we omit the proof.
Formulated only for set algebras, it also holds for weak set algebras. 
\begin{lemma}\label{box}  For any  ordinal $\mu>1$, $\A\in {\sf Cs}_{\mu}$ and  $k<\mu$, 
let $I_k$ and $\Box_k$ be as defined above. Then
if $\A\in \sf Cs_{\alpha}$ has top element $^{\alpha}U$ and $\beta>\alpha$, then 
the following hold for any $Y\subseteq {}^{\alpha}U$ and $k<\alpha:$
\begin{enumerate}
\item  $I_k(Y)\subseteq Y,$ $\Box_k(Y)\subseteq Y,$ 

\item  If $f: \wp(^{\alpha}U)\to {}\wp(^{\beta}U)$ is defined via
$$X\mapsto \{s\in {}^{\beta}U: s\upharpoonright \alpha\in X\},$$
then $f(I_kX)=I_k(f(X))$ and $f(\Box_kX)=\Box_k(f(X))$,
for any $X\subseteq {}^{\alpha}U$.
\end{enumerate}
\end{lemma}
In cylindric algebra theory a subdirect product of set algebras is isomorphic to a generalized set algebra.
We show that this phenomena persists when the bases carry topologies; we need to describe the topology on the base 
of the resulting generalized set algebra in terms
of the topologies on the bases of the set algebras 
involved in the subdirect product.

\begin{definition} Let $\{X_i:i\in I\}$ be a family of topological spaces indexed by $I$. 
Let $X=\bigcup X_i$ be the disjoint union of the underlying sets. For each $i\in I$ let
$\phi_i: X_i\to X$ be the canonical injection. 
The {\it coproduct on $X$} is defined as the finest topology on $X$ for which the canonical injections are continuous.
\end{definition}
That is a  subset $U$ of $X$ is open in the coproduct topology on $X$ $\iff$ its preimage $\phi_i^{-1}(U)$ is open in $X_i$ for each $i\in I$ $\iff$ its intersection with $X_i$ is open relative to $X_i$ for each $i\in I$.

\begin{theorem} Let $\B$ be the $\sf Gs_{\alpha}$ with unit $V=\bigcup_{i\in I}{}^{\alpha}U_i$ where $U_i\cap U_j=\emptyset$ and base 
$\bigcup_{i\in I}U_i$ carrying a topology.  Assume that $\B$ has universe $\wp(V)$. 
Let $\A_i$ be the $\sf Cs_{\alpha}$ 
with base $U_i$, $U_i$ having the subspace topology and universe $\wp(^{\alpha}U_i)$. 
Then $f:\B\to \prod_{i\in I}\A_i$ 
defined by
$X\mapsto (X\cap {}^{\alpha}U_i:I\in I)$ is an isomorphism of cylindric algebras; 
furthermore it respects the interior operators stimulated by the topologies
on the bases.
\end{theorem}
${\sf TCs}_{\alpha}({\sf TGs}_{\alpha})$ 
denotes the class of topological (generalized) set algebras.
Now such concrete set algebras lend itself to an abstract formulation aiming to capture the concrete set algebras; or rather the variety 
generated by them.
This consists of expanding the signature of cylindric algebras 
by unary operators,  or modalities, one for each $k<\alpha$, satisfying certain identities.
The axiomatizations we give are actually simpler than  those stipulated by Georgescu in \cite{g, g2}, 
although  locally finite polyadic algebras and locally finite cylindric algebras are  equivalent.
We use only substitutions corresponding to replacements; in the case of dimension complemented algebras
all substitutions corresponding to finite transformations are term 
definable from these, cf. \cite{HMT2}. This makes axiom $(A8)$ on p.1 of 
\cite{g} superfluous.
We start with the standard definition of cylindric algebras \cite[Definition 1.1.1]{HMT2}:
\begin{definition}
Let $\alpha$ be an ordinal. A {\it cylindric algebra of dimension $\alpha$}, a  $\CA_{\alpha}$ for short, 
is defined to be an algebra
$$\C=\langle C, +, \cdot, -, 0, 1, \cyl{i}, {\sf d}_{ij}  \rangle_{i,j\in \alpha}$$
obeying the following axioms for every $x,y\in C$, $i,j,k<\alpha$

\begin{enumerate}

\item The equations defining Boolean algebras,

\item $\cyl{i}0=0,$

\item $x\leq \cyl{i}x,$

\item $\cyl{i}(x\cdot \cyl{i}y)=\cyl{i}x\cdot \cyl{i}y,$

\item $\cyl{i}\cyl{j}x=\cyl{j}\cyl{i}x,$

\item ${\sf d}_{ii}=1,$

\item if $k\neq i,j$ then ${\sf d}_{ij}=\cyl{k}({\sf d}_{ik}\cdot {\sf d}_{jk}),$

\item If $i\neq j$, then ${\sf c}_i({\sf d}_{ij}\cdot x)\cdot {\sf c}_i({\sf d}_{ij}\cdot -x)=0.$

\end{enumerate}
\end{definition}
For a cylindric algebra $\A$, we set ${\sf q}_ix=-{\sf c}_i-x$ and ${\sf s}_i^j(x)={\sf c}_i({\sf d}_{ij}\cdot x)$.
%We consider only infinite dimensional cylindric algebras.
Now we want to abstract equationally the prominent features of the concrete interior operators defined on cylindric 
set and weak set algebras.
We expand the signature of $\CA_{\alpha}$ 
by a unary operation $I_i$ 
for each $i\in \alpha.$ In what follows $\oplus$ denotes the operation of symmetric difference, that is, 
$a\oplus b=(\neg a+b)\cdot (\neg b+a)$.
For $\A\in \CA_{\alpha}$ and $p\in \A$, $\Delta p$, {\it the dimension set of $p$}, is defined to be the set 
$\{i\in \alpha: {\sf c}_ip\neq p\}.$ In polyadic terminology $\Delta p$ is called {\it the support of $p$}, and if $i\in \Delta p$, then $i$ 
is said to {\it  support $p$} \cite{g,g2}.

\begin{definition}\label{topology}\cite{g} A {\it topological cylindric algebra of dimension $\alpha$}, $\alpha$ an ordinal,
is an algebra of the form $(\A,I_i)_{i<\alpha}$ where $\A\in \sf CA_{\alpha}$ 
and for each $i<\alpha$, $I_i$ is a unary operation on 
$A$ called an {\it  interior operator} satisfying the following equations for all $p, q\in A$ and $i, j\in \alpha$:
\begin{enumerate}
\item ${\sf q}_i(p\oplus q)\leq {\sf q}_i (I_ip\oplus I_iq),$
\item $I_ip\leq p,$
\item $I_ip\cdot I_ip=I_i(p\cdot q),$
\item $I_ip\leq I_iI_ip,$
\item $I_i1=1,$
\item ${\sf c}_kI_ip=I_ip, k\neq i, k\notin \Delta p,$
\item ${\sf s}_j^iI_ip=I_j{\sf s}_j^ip, j\notin \Delta p.$ 
\end{enumerate}
\end{definition}
The class of all such topological cylindric  
algebras are denoted by ${\sf TCA}_{\alpha}.$ 
\subsection{Basic Lemmas}

%Such (topological algbras) algebras arise when $\sf L$ is taken to be $\sf S4$. 
For $\sf K$ any class of $\CA$s, we write $\sf TK$ for the corresponding class of topological $\CA$s. 
For example, for any ordinal $\beta$, (recall that) $\sf TCA_{\beta}$ and we write $\sf TDc_{\beta}$ to denote the classes of all topological $\CA_{\beta}$s 
and {\it dimension complemented} topological $\CA_{\beta}$s, respectively.
Throughout this section, unless otherwise indicated, $\alpha$ denotes an infinite ordinal. Then $\A\in \sf TDc_{\alpha}$ $\iff$ $\alpha\neq \Delta x$ for all $x\in A$ $\iff$$\alpha\sim \Delta x$ is infinite for all $x\in A$. 
Furthermore, for every $\A\in \sf TDc_{\alpha}$ and every finite transformation $\tau$ we have a unary operation
${\sf s}_{\tau}$ that happens to be a Boolean endomorphism on $\A$ 
\cite[Theorem 1.11.11]{HMT2}.
\begin{lemma}\label{diagonal} 
\begin{enumerate}
\item Let $\C\in \CA_{\alpha}$ and let $F$ be a Boolean filter on $\C$.
Define the relation $E$ on $\alpha$ by $(i,j)\in E$ if and only if
${\sf d}_{ij}\in F$. Then $E$ is an equivalence relation on $\alpha$.
\item Let $\C\in \CA_{\alpha}$ and $F$ be a Boolean filter of $\C$. Let $V=\{\tau\in {}^{\alpha}\alpha: |\{i\in \alpha: \tau(i)\neq i\}|<\omega\}$.
For $\sigma, \tau\in V$,  write 
$$\sigma\equiv_E\tau\textrm {  iff } 
(\forall i\in \alpha) (\sigma (i),\tau(i))\in E.$$ 
and let 
$$\bar{E}=\{(\sigma,\tau)\in {}^2V: \sigma\equiv_E \tau\}.$$
Then $\bar{E}$ is an euiqvalence relation on $V$. Let $W= V/\bar{E}.$ 
For $h\in W,$ write $h=\tau/\bar{E}$ for $\tau\in V$ such that
$\tau(j)/E=h(j)$ for all $j\in \alpha$. 
Let 
$f(x)=\{ \bar{\tau} \in W: {\sf s}_{\tau}x\in F\}.$  
Then $f$ is well defined. 
Furthemore, $W$ can be identified with the weak space $^{\alpha}[U/E]^{(\bar{p})}$ where $\bar{p}=(p(i)/E: i <\alpha)$ 
via $\tau/\bar{E}\mapsto [\tau]$, 
where $[\tau](i)=\tau(i)/E.$ Accordingly, we write $W={}^{\alpha}[U/E]^{(\bar{p})}.$ 
\end{enumerate}
\end{lemma}

\begin{definition} Let $\A$ be an algebra having a cylindric reduct of dimension $\alpha$. 
A Boolean ultrafilter $F$ of $\A$ is said to be {\it Henkin} if for all $k<\alpha$, for all $x\in A$, whenever
${\sf c}_kx\in F$, then there exists $l\notin \Delta x$ such
that ${\sf s}_l^kx\in F$.
\end{definition}

\begin{lemma} Let everything be as in the previous lemma, and assume that $F$ is a Henkin ultrafilter.
Then $f$ as defined in the previous lemma is a $\sf CA$ homomorphsim.
\end{lemma}
\begin{proof}\cite{IGPL}.
\end{proof}
\begin{definition}\label{interior}
Let everything be as in the hypothesis of lemma \ref{essence}.
For $s\in W$ and $k<\alpha$ we write $s^k_u$ for $s^k_{u/E}$. 
For $k\in \alpha,$ then  $I_k$ is the (interior) operator  on $\wp(W)$ 
defined by $I_k(X)=\{s\in W: s_k\in {\sf int}\{u\in U: s^k_{u}\in X\}\}.$  Similarly, if $V: U/E\to \wp(\wp(U/E))$ is a Chang system
then $\Box_k$ is defined on $\wp(W)$ by  
$s\in \Box_k(X)\Longleftrightarrow \{u/E\in U/E: {s}^k_{u}\in X\}\in V[s(i/E)].$
\end{definition}

From now on, we replace the interior operator $I_i$ by $\Box_i$.
\begin{lemma}\label{essence}
Assume that $\C\in \sf TDc_{\alpha}$, $F$ is a Henkin  ultrafilter of $\C$ and $a\in F.$
Then there exist a non-empty set $U$, $p\in {}^{\alpha}U,$
a topology on $U/E$ and a homomorphism 
$f:\C\to (\wp(W), \Box_i)_{i<\alpha}$ with $f(a)\neq 0$, 
where $W={}^{\alpha}[U/E]^{\bar{p}}$, with $E$ as defined in lemma \ref{diagonal}
and  $\Box_i$ $(i<\alpha)$ is the concrete interior operator 
defined  in Definition \ref{interior}.
\end{lemma}
\begin{proof}
Let $W={}^{\alpha}[\alpha/E]^{(\bar{Id})}.$
Define, as we did before,  $f:\A\to \wp(W)$ via
$$p\mapsto \{\bar{\tau} \in W: {\sf s}_{\tau}p\in F\}.$$
For $i\in \alpha$ and $p\in \A$, let 
$$O_{p,i}=\{k/E\in \alpha/E: {\sf s}_i^kI(i)p\in F\}.$$
Let $${\cal B}=\{O_{p,i} : i\in \alpha, p\in A\}.$$
Then it is easy to check that ${\cal B}$ 
is the base for a topology on $\alpha/E$.
To define the interior operations, we set
for each $i<\alpha$ 
$$\Box_i: \wp(W)\to \wp (W)$$
by $$[x]\in J_iX\Longleftrightarrow \exists U\in {\cal B}(x_i/E\in U\subseteq \{u/E\in \alpha/E: [x]^i_{u/E}\in X\}),$$
where $X\subseteq V$. Note that 
$[x]^i_{u/E}=[x^i_u]$.
We now check that $f$ preserves the interior operators $\Box_i$ $(i<\alpha)$, too.
We need to show 
$$\psi(\Box_ip)=\Box_i(\psi(p)).$$
The reasoning is like \cite{g}; the difference is that in \cite{g}, the constants
denoted by  $x_i$ are endomorphisms on $\A$; the value $x_i$ at $j$ 
corresponds in our adopted approach to ${\sf s}^j_u$ where $u=x_i(j)$.
Let $[x]$ be in $\psi(\Box_ip)$. Let 
$${\sf sup}(x)=\{k\in \alpha: x_k\neq k\}.$$
Then, by definition,  ${\sf s}_x\Box_ip\in F$. Hence 
$${\sf s}_{x_i}^i I_i{\sf s}^{i_1}_{x_1}\ldots {\sf s}^{i_n}_{x_n} p\in F,$$
where
$${\sf sup}(x)\sim \{i\}=\{j_1,\ldots, j_n\}.$$
Let 
$$y=[j_1|x_1]\circ \ldots [j_n| x_n].$$
Then $x_i/E\in \{u/E: {\sf s}_u^i I(i){\sf s}_yp\in F\}\in q.$
But $I_i{\sf s}_yp\leq {\sf s}_y p,$ hence
$$U=\{u/E: {\sf s}_u^i I_i{\sf s}_yp\in F\}\subseteq \{u/E: {\sf s}_u^i{\sf s}_yp\in F\}.$$
It follows that $x_i/E\in U\subseteq \{u/E: x^i_u\in \Psi(p)\}.$
Thus $[x]\in \Box_i\psi(p).$
Now we prove the other harder direction.
Let $[x]\in \Box_i\Psi(p)$. Let $U\in \B$ be such that 
$$x_i/E\in U\subseteq \{u/E\in \alpha/E: {\sf s}_u^i{\sf s}_xp\in F\}.$$
Assume that $U=O_{r,j}$, where  $r\in \A$ and $j\in \alpha$.
Let $u\in \alpha\sim [\Delta p\cup \Delta r\cup \{i,j\}]$. By dimension complementedness 
such a $u$ exists. 
Then we have:  
\begin{align*}
{\sf s}_u^j\Box_jr\in F&\Longleftrightarrow {\sf s}_u^i{\sf s}_xp\in F,\\
{\sf s}_u^j\Box_jr\cdot {\sf s}_u^i{\sf s}_xp\in F&\Longleftrightarrow {\sf s}_u^j\Box_jr\in F.
\end{align*}
But
${\sf s}_u^j\Box_jr={\sf s}_u^i\Box_i{\sf s}_j^ir$, so we have
\begin{align*}
{\sf s}_u^i\Box_i{\sf s}_j^ir\cdot {\sf s}_u^i{\sf s}_xp &\oplus {\sf s}^i_u\Box_i{\sf s}_j^ir\in F,\\
{\sf s}_u^i[\Box_i{\sf s}_j^ir\cdot {\sf s}_xp&\oplus \Box_i{\sf s}_j^ir]\in F,\\
{\sf q}_i[\Box_i{\sf s}_j^ir\cdot {\sf s}_xp&\oplus \Box_i{\sf s}_j^ir]\in F,\\
{\sf q}_i[\Box_i{\sf s}_j^ir\cdot \Box_i{\sf s}_xp&\oplus \Box_i{\sf s}_j^ir]\in F,\\
{\sf s}^i_{x_i}[\Box_i{\sf s}_j^ir\cdot \Box_i{\sf s}_xp&\oplus \Box_i{\sf s}_j^ir]\in F,\\
{\sf s}^j_{x_i}I_jr\cdot s^i_{x_i}\Box_i{\sf s}_xp&\oplus {\sf s}^j_{x_i}\Box_jr\in F.\\
\end{align*}
But ${\sf s}^j_{x_i}\Box_jr\in F,$ hence 
${\sf s}^i_{x_i}\Box_i{\sf s}_xp\in F$, and so 
$x\in \Psi(\Box_ip)$ as required.
\end{proof}

We also need the notion of {\it compressing} dimensions 
and, dually,  {\it dilating them}; expressed by the notion of neat reducts.
\begin{definition} 
Let $\alpha<\beta$ be ordinals and $\B\in \sf TCA_{\beta}$. Then $\Nrr_{\alpha}\B$ is the algebra with universe 
$Nr_{\alpha}\A=\{a\in \A: \Delta a\subseteq \alpha\}$ and operations obtained by discarding the operations of $\B$ 
indexed by ordinals in $\beta\sim \alpha$. 
$\Nrr_{\alpha}\B$ is called the {\it neat $\alpha$ reduct of $\B$}. If $\A\subseteq \Nrr_{\alpha}\B$, with $\B\in \sf TCA_{\beta}$,
then we say that $\B$ is  {\it a $\beta$
dilation of $\A$}, or simply {\it a dilation} of $\A$ if $\beta$ is clear from context.
\end{definition}

\section{$\sf OTT$s in case of the presence of infinitely many variables}

Now we prove an omitting types theorem for $\TDc_{\alpha}$ and $\sf TLf_{\alpha}$ when $\alpha$ is a countable
infinite ordinal; also generalizing the result in \cite{g} which addresses only topological locally finite algebras.
%The proof adopted herein, we find is simpler than the proof
%in \cite{g}; and it resorts to the Baire category theorem for compact Hausdorff spaces
%as is often the case with `omitting types constructions' though they are rarely presented this way.
The proof is similar to the proof of \cite[Theorem 3.2.4]{Sayed} having at our disposal
lemma \ref{essence}.
We omit the parts of the proof that overlap with those in \cite{Sayed}.
%But we still need some preparing to do.
Given $\A\in \TCA_{\alpha}$, $X\subseteq \A$ is called a {\it finitary type}, if
$X\subseteq \Nrr_n\A$ for some $n\in \omega$. It is non-principal if $\prod X=0$.

\begin{definition} A {\it representation} of   $\A\in \sf TDc_{\alpha}$ is a non-zero  homomorphism
$f:\A\to \B$ where $\B$ is a weak set algebra. If $\A$ is simple then $f$ is necessarily one to one.
$X\subseteq \A$ is {\it omitted} by $f$ if $\bigcap_{x\in X} f(x)=\emptyset$,
otherwise it is {\it realized} by $f$.
\end{definition}

We define certain cardinals; it is consistent that such cardinal are uncountable.  Throughout this paper
we do not assume the continuum hypothesis.
\begin{definition}
1. A subset $X\subseteq \mathbb{R}$ is {\it meager} if it is a countable union of nowhere dense sets. Let $\sf covK$ be the least cardinal $\kappa$
such that $\mathbb{R}$ can be covered by $\kappa$ many nowhere dense sets.
Let $\mathfrak{p}$ be the least cardinal  $\kappa$ such that there are $\kappa$ many meager sets  of $\mathbb{R}$ whose union is not meager.

2. A {\it Polish space} is a topological space that is metrizable by a complete separable metric.

\end{definition}

Examples of Polish spaces are $\mathbb{R}$, the Cantor set $^{\omega}2$ and the Baire space $^{\omega}\omega$.
These are called {\it real spaces} because they are Baire isomorphic. Any second countable compact Hausdorff space,
like the Stone space of a countable Boolean algebra, is a Polish space \cite{Fre}.
\begin{theorem}\label{car}
\begin{enumerate}
\item The cardinals $\sf cov K$ and $\mathfrak{p}$ are uncountable cardinals, such that
$\mathfrak{p}\leq \sf cov K\leq 2^{\omega}$.
\item The cardinal $\sf cov K$ is the least cardinal such the  Baire category theorem
for Polish spaces fails, and it is also the largest for which Martin's axiom for countable Boolean
algebras holds.
\item If $X$ is a Polish space, then it cannot be covered by $<\sf covK$ many meager sets.
If $\lambda<\mathfrak{p}$, and $(A_i: i<\lambda)$ is a family of meager subsets of $X$,
then $\bigcup_{i\in \lambda}A_i$ is meager.
\end{enumerate}
\end{theorem}
\begin{proof} For the definition and required properties of $\mathfrak{p}$, witness \cite[pp. 3, pp. 44-45, corollary 22c]{Fre}.  
For properties of $\sf cov K,$
witness \cite[the remark on. pp. 217]{Sayed}.
%The statements relating the two cardinals are easy to establish using the
%definitions.
\end{proof}
Both cardinals $\sf cov K$ and $\mathfrak{p}$  have an extensive literature, witness \cite{Fre} and the references therein.
It is consistent that $\omega<\mathfrak{p}<\sf cov K\leq 2^{\omega}$ 
so that the two cardinals are generally different, but it is also consistent that they are equal; equality holds for example in the Cohen
real model of Solovay and
Cohen. Martin's axiom implies that
they are both equal to the continuum.
Let $\A$ be any Boolean algebra. The set of ultrafilters of $\A$ is denoted
by $\mathfrak{U}(\A)$. The Stone topology  makes $\mathfrak{U}(\A)$ a compact Hausdorff space.
We denote this space by $\A^*$. Recall that the Stone topology has as its basic open sets the sets $\{N_x:x\in A\}$ where
$$N_x=\{F\in\mathfrak{U}(\A):x\in F\}.$$
%Note that, if $A$ is countable, then by theorem \ref{polish} $\A^*$ is a Polish space.
Let $x\in A$, $Y\subseteq A$ and suppose that $x=\sum Y.$
We say that an ultrafilter $F\in\mathfrak{U}(\A)$ \emph{preserves}
$Y$ $\iff$ whenever $x\in F$, then $y\in F$ for some $y\in Y$.
Now let $\A\in \sf TLf_\omega$. For each $i\in\omega$ and $x\in A$ let
$$\mathfrak{U}_{i,x}=\{F\in\mathfrak{U}(\A):F\mbox{ preserves }\{{\sf s}^i_jx:j\in\omega\}\}.$$
Then
\begin{align*}\mathfrak{U}_{i,x}&=\{F\in\mathfrak{U}(\A): {\sf c}_ix\in F\Rightarrow (\exists j\in\omega) {\sf s}^i_jx\in F\}\\
&=N_{-{\sf c}_ix}\cup\bigcup_{j<\omega}N_{{\sf s}^i_jx}.
\end{align*}
Let $$\mathcal{H}(\A)=\bigcap_{i\in\omega,x\in A}\mathfrak{U}_{i,x}(\A)\cap\bigcap_{i\neq j}N_{-{\sf d}_{ij}}.$$
It is clear that $\mathcal{H}(\A)$ is a $G_\delta$ set in $\A^*$.
For $F\in\mathfrak{U}(\A),$ let $$rep_F(x)=\{\tau\in{}^\omega\omega: {\sf s}_\tau^{\A} x\in F\},$$ for all $x\in A.$
Here for $\tau\in {}^{\omega}\omega$, ${\sf s}_{\tau}^{\A}x$ by definition is ${\sf s}_{\tau\upharpoonright \Delta x}^{\A}x$.
The latter is well defined because  $|\Delta x|<\omega.$
When $a\in F$, then $rep_F$ is a representation of $\A$ such that $rep_F(a)\neq 0$.
The following theorem establishes a one to one correspondence between representations
of locally finite cylindric algebras and Henkin
ultrafilters. $\sf Cs_{\omega}^{reg}$ denotes the class of {\it regular set algebras}; a a set algebra with top element
$^{\alpha}U$ is such, if whenever $f, g\in {}^{\alpha}U,$ $f\upharpoonright \Delta x=g\upharpoonright \Delta x$, and
$f\in X$ then $g\in X$. This reflects
the metalogical property that if two assignments agree on
the free variables occurring in a formula then both satisfy the
formula or none does.

\begin{theorem}\label{th1}
If $F\in\mathcal{H}(\A)$, then $rep_F$ is a homomorphism from $\A$ onto an element
of $\sf Lf_\omega\cap \sf Cs_\omega^{reg}$
with base $\omega$. Conversely, if $h$ is a homomorphism from $\A$ onto an element of
$\sf Lf_\omega\cap \sf Cs_\omega^{reg}$ with base $\omega$, then there is a unique $F\in \mathcal{H}(\A)$ such that $h=rep_F.$
\end{theorem}
The next theorem is due to Shelah, and will be used to show that in certain cases uncountably many non-principal types can be
omitted.
\begin{theorem}\label{Shelah} Suppose that $T$ is a theory,
$|T|=\lambda$, $\lambda$ regular, then there exist models $\M_i: i<{}^{\lambda}2$, each of cardinality $\lambda$,
such that if $i(1)\neq i(2)< \chi$, $\bar{a}_{i(l)}\in M_{i(l)}$, $l=1,2,$, $\tp(\bar{a}_{l(1)})=\tp(\bar{a}_{l(2)})$,
then there are $p_i\subseteq \tp(\bar{a}_{l(i)}),$ $|p_i|<\lambda$ and $p_i\vdash \tp(\bar{a}_ {l(i)})$
($\tp(\bar{a})$ denotes the complete type realized by
the tuple $\bar{a}$).
\end{theorem}
\begin{proof} \cite[Theorem 5.16, Chapter IV]{Shelah}.
\end{proof}
We shall use the algebraic counterpart of the following corollary obtained by restricting Shelah's theorem to the countable case:
\begin{corollary}\label{Shelah2} For any countable theory, there is a
family of $< {}^{\omega}2$ countable models that overlap only on principal types.
\end{corollary}

\begin{theorem}\label{infinite}
\begin{enumerate}
\item Let $\A\in \TDc_{\omega}$ be  countable. Assume that
$\kappa<\mathfrak{p}$. Let $(\Gamma_i:i\in \kappa)$ be a set of non-principal types in $\A$.
Then there is a topological  weak set algebra $(\B, \Box_i)_{i<\omega}$, that is,  $\B$ has top element a weak space,
and a  homomorphism  $f:\A\to (\B, \Box_i)_{i<\omega},$
such that for all $i\in \kappa$, $\bigcap_{x\in X_i}f(x)=\emptyset$, and $f(a)\neq 0.$ If $\A$ is simple, then $\mathfrak{p}$ can 
be replaced by $\sf covK$.

\item If $\A\in \sf TLf_{\omega}$, and $(\Gamma_i: i\in \kappa)$
is a family of finitary non-principal types
then there is a topological  set algebra $(\B, \Box_i)_{i<\omega}$, that is, $\B$ has  top element a
Cartesian square,  and $\B\in {\sf Cs}_{\omega}^{reg}\cap \sf Lf_{\omega}$ together
with a homomorphism  $f:\A\to (\B, \Box_i)_{i<\omega}$
such that $\bigcap_{x\in X_i}f(x)=\emptyset$, and $f(a)\neq 0.$
If the family of given types are  ultrafilters 
then $\mathfrak{p}$ can be replaced by $2^{\omega}$, so that $<2^{\omega}$ types can be omitted.
\end{enumerate}
\end{theorem}
\begin{proof}
For the first part, we have by \cite[1.11.6]{HMT2} that
\begin{equation}\label{t1}
\begin{split} (\forall j<\alpha)(\forall x\in A)({\sf c}_jx=\sum_{i\in \alpha\smallsetminus \Delta x}
{\sf s}_i^jx.)
\end{split}
\end{equation}
Now let $V$ be the weak space $^{\omega}\omega^{(Id)}=\{s\in {}^{\omega}\omega: |\{i\in \omega: s_i\neq i\}|<\omega\}$.
For each $\tau\in V$ for each $i\in \kappa$, let
$$X_{i,\tau}=\{{\sf s}_{\tau}x: x\in X_i\}.$$
Here ${\sf s}_{\tau}$
is the unary operation as defined in  \cite[1.11.9]{HMT2}.
For each $\tau\in V,$ ${\sf s}_{\tau}$ is a complete
Boolean endomorphism on $\A$ by \cite[1.11.12(iii)]{HMT2}.
It thus follows that
\begin{equation}\label{t2}\begin{split}
(\forall\tau\in V)(\forall  i\in \kappa)\prod{}^{\A}X_{i,\tau}=0
\end{split}
\end{equation}
Let $S$ be the Stone space of the Boolean part of $\A$, and for $x\in \A$, let $N_x$
denote the clopen set consisting of all
Boolean ultrafilters that contain $x$.
Then from \ref{t1}, \ref{t2}, it follows that for $x\in \A,$ $j<\beta$, $i<\kappa$ and
$\tau\in V$, the sets
$$\bold G_{j,x}=N_{{\sf c}_jx}\setminus \bigcup_{i\notin \Delta x} N_{{\sf s}_i^jx}
\text { and } \bold H_{i,\tau}=\bigcap_{x\in X_i} N_{{\sf s}_{\bar{\tau}}x}$$
are closed nowhere dense sets in $S$.
Also each $\bold H_{i,\tau}$ is closed and nowhere
dense.
Let $$\bold G=\bigcup_{j\in \beta}\bigcup_{x\in B}\bold G_{j,x}
\text { and }\bold H=\bigcup_{i\in \kappa}\bigcup_{\tau\in V}\bold H_{i,\tau.}$$
By properties of $\mathfrak{p}$, $\bold H$ can be reduced to a countable collection of nowhere dense sets.
By the Baire Category theorem  for compact Hausdorff spaces, we get that ${\mathfrak H}(\A)=S\sim \bold H\cup \bold G$ is dense in $S$.
Let $F$ be an ultrafilter in $N_a\cap X$.
By the very choice of $F$, it follows that $a\in F$ and  we have the following.
\begin{equation}
\begin{split}
(\forall j<\beta)(\forall x\in B)({\sf c}_jx\in F\implies
(\exists j\notin \Delta x){\sf s}_j^ix\in F.)
\end{split}
\end{equation}
and
\begin{equation}
\begin{split}
(\forall i<\kappa)(\forall \tau\in V)(\exists x\in X_i){\sf s}_{\tau}x\notin F.
\end{split}
\end{equation}
Let $V={}^{\omega}\omega^{Id)}$ and let
$W$ be the quotient of $V$ as defined above.
That is $W=V/\bar{E}$ where $\tau\bar{E}\sigma$ if ${\sf d}_{\tau(i), \sigma(i)}\in F$
for all $i\in \omega$.
Define $f$ as before by $$f(x)=\{\bar{\tau}\in W:  {\sf s}_{\tau}x\in F\}, \text { for } x\in \A.$$
and the interior operators
for each $i<\alpha$ by
$$\Box_i: \wp(W)\to \wp (W)$$
by $$[x]\in \Box_iX\Longleftrightarrow \exists U\in {\cal B}(x_i/E\in U\subseteq \{u/E\in \alpha/E:[x]^i_{u/E}\in X\}),$$
where $X\subseteq W$; here $W$ and $E$ are as defined in lemmas, \ref{diagonal},  and $\cal B$
is the base for the topology on $U/E$ defined as in the proof of theorem \ref{essence}.
Then by lemma \ref{essence} $f$
is a homomorphism
such that $f(a)\neq 0$ and it can be easily checked  that
$\bigcap f(X_i)=\emptyset$ for all $i\in \kappa$,
hence the desired conclusion. If $\A$ is simple, then  by the properties of $\sf cov K$, ${\mathfrak H}(\A)=S\sim \bold H\cup \bold G$ is non-empty. 
Let $F\in H(\A)$ and let $a\in F$. 
The representation built using such $F$ as above, call it $f$, has $f(a)\neq 0$, By simplicity of $\A$, $f$ is an injection,  
because $ker f=\{0\}$,  since $a\notin ker f$ 
and by simplicity, either $ker f=\{0\}$ or $ker f=\A$.

2. One proceeds exactly like in the previous item, but using, as indicated above, the fact that
the operations ${\sf s}_{\tau}$ for {\it any} $\tau\in {}^\omega\omega$
which are definable in locally finite algebras, via ${\sf s}_{\tau}x= {\sf s}_{\tau\upharpoonright \Delta x}x$, for
any $x\in A$. Furthermore, ${\sf s}_{\tau}\upharpoonright \Nrr_n\A$ is a complete
Boolean endomorphism, so that we guarantee that infimums are preserved and the sets
$\bold H_{i,\tau}=\bigcap_{x\in X_i} N_{{\sf s}_{\bar{\tau}}x}$
remain no-where dense in the Stone topology.
Now for the second part. Let  $\A\in \sf TLf_{\alpha}$, $\lambda <2^{\omega}$ and $\bold F=(X_i:i<\lambda)$ be a family
of maximal non-principal finitary types, so that for each $i<\lambda$, there exists $n\in \omega$ such that
$X_i\subseteq \Nrr_n\A$, and $\prod X_i=0$; that is $X_i$ is a Boolean ultrafilter in $\Nrr_n\A$.
Then by Theorem \ref{Shelah}, or rather its direct algebraic counterpart,
there are $^{\omega}2$ representations such that
if $X$ is an ultrafilter in $\mathfrak{Nr}_n\A$ (some $n\in \omega)$) that is realized in two such
representations, then $X$  is necessarily principal.
That is there exist a family of countable locally finite set algebras, each  with countable base, call it
$(\B_{j_i}: i<{}2^{\omega})$, and isomorphisms $f_i:\A\to \B_{j_i}$
such that if $X$ is an ultrafilter in $\mathfrak{Nr}_n\A$, for which there exists distinct $k, l\in {}2^{\omega}$
with $\bigcap f_l(X)\neq \emptyset$ and $\bigcap f_j(X)\neq \emptyset$,
then $X$ is principal, so that  from Corollary \ref{Shelah2} such representations overlap only
on maximal principal  types. By Theorem \ref{th1}, there exists a family $(F_i: i< 2^{\omega})$
of Henkin ultrafilters
such that $f_i=h_{F_i}$, and by theorem \ref{essence} we can assume that
$h_{F_i}$ is a $\TCA_{\alpha}$ isomorphism as follows. Denote $F_i$ by $G$.
For $p\in \A$ and $i<\alpha$, let
$O_{p,i}=\{k\in \alpha: {\sf s}_i^kI(i)p\in G\}$
and let ${\cal B}=\{O_{p,i} : i\in \alpha, p\in A\}.$
Then ${\cal B}$
is the base for a topology on $\alpha$ and the concrete interior operations are defined
for each $i<\alpha$ via
$\Box_i: \wp(^{\alpha}\alpha)\to \wp (^{\alpha}\alpha)$
$$x\in \Box_iX\Longleftrightarrow \exists U\in {\cal B}(x_i\in U\subseteq \{u\in \alpha: x^i_{u}\in X\}),$$
where $X\subseteq W$.
Assume, for contradiction,  that there is no representation (model)
that omits $\bold F$.
Then for all $i<{}2^{\omega}$,
there exists $F$ such that $F$ is realized in $\B_{j_i}$. Let $\psi:{}2^{\omega}\to \wp(\bold F)$, be defined by
$\psi(i)=\{F: F \text { is realized in  }\B_{j_i}\}$.  Then for all $i<2^{\omega}$, $\psi(i)\neq \emptyset$.
Furthermore, for $i\neq k$, $\psi(i)\cap \psi(k)=\emptyset,$ for if $F\in \psi(i)\cap \psi(k)$ then it will be realized in
$\B_{j_i}$ and $\B_{j_k}$, and so it will be principal.  This implies that
$|\bold F|=2^{\omega}$ which is impossible.

\end{proof}
\subsection{A positive $\sf OTT$ for $L_n$ theories}

Unless otherwise indicated $n$ is a finite ordinal $>1$. 
In $L_{\omega, \omega}$ an atomic model for a countable atomic theory  (which exists by $\sf VT$) 
omits all non--principal types. The last item in the next theorem is the 
$`L_n$ version (expressed algebraically)' of  this property.
The condition of maximality expressed in `ultrafilters' meaning 'maximal filters' delineates the edge of 
an independent statement to a provable one, cf. \cite[Theorem 3.2.8]{Sayed} and the  fourth item the next theorem. 
\begin{theorem}\label{sh} Let $n<\omega$. Let $\A\in \bold S_c\Nr_n\TCA_{\omega}$ be countable. 
Let $\lambda< 2^{\aleph_0}$ and let 
$\bold X=(X_i: i<\lambda)$ be a family of non-principal types  of $\A$.
Then the following hold:
\begin{enumerate}
\item If $\A\in \Nr_n\CA_{\omega}$ and the $X_i$s are non--principal ultrafilters,  then $\bold X$ can be omitted in a ${\sf TGs}_n$.
\item Every subfamily of $\bold X$ of cardinality $< \mathfrak{p}$ can be omitted in a ${\sf Gs}_n$; in particular, every countable 
subfamily of $\bold X$ can be omitted in a ${\sf Gs}_n$,
If $\A$ is simple, then every subfamily 
of $\bold X$ of cardinality $< \sf covK$ can be omitted in a ${\sf Cs}_n$.
\item Every subfamily of $\bold X$ of cardinality $< \mathfrak{p}$ can be omitted in a ${\sf Gs}_n$; in particular, every countable
\item It is consistent, but not provable (in $\sf ZFC$), that $\bold X$ can be omitted in a ${\sf Gs}_n$,
\item If $\A\in \Nr_n\CA_{\omega}$ and $|\bold X|<\mathfrak{p}$, then $\bold X$ can be omitted $\iff$ every countable subfamily of $\bold X$ can be omitted.   
If $\A$ is simple, we can replace $\mathfrak{p}$ by $\sf covK$. 
\item If $\A$ is atomic, {\it not necessarily countable}, but have countably many atoms, 
then any family of non--principal types can be omitted in an atomic ${\sf Gs}_n$; in particular, 
$\bold X$ can be omitted in an atomic ${\sf Gs}_n$; if $\A$ is simple, we can replace ${\sf Gs}_n$ by 
${\sf Cs}_n$.

\end{enumerate}
\end{theorem}
\begin{proof}
To substantially simplify the proof while retaining the gist of ideas used in the more general case  
for the first item we assume that $\A$ is countable {\it and simple}, that is to say, has no proper ideal. 
This means that $\A$ `algebraically represents' a complete countable theory. 
We have $\prod ^{\B}X_i=0$ for all $i<\kappa$ because,
$\A$ is a complete subalgebra of $\B$. 
To see why, assume that $S\subseteq \A$ and $\sum ^{\A}S=y$, and for contradiction that there exists $d\in \B$ such that
$s\leq d< y$ for all $s\in S$. Then, assuming that $A$ generates $\B$, we can infer that $d$ uses finitely many dimensions in $\omega\sim n$, 
$m_1,\ldots, m_n$, say.
Now let $t=y\cdot -{\sf c}_{m_1}\ldots {\sf c}_{m_n}(-d)$.
We claim that $t\in \A=\Nrr_{n}\B$ and $s\leq t<y$ for all $s\in S$. This contradicts
$y=\sum^{\A}S$.
The first required follows from the fact that $\Delta y\subseteq n$ and that all indices in $\omega\sim n$ that occur in $d$
are cylindrified.  In more detail, put $J=\{m_1, \ldots, m_n\}$ and let 
$i\in \omega\sim n$, then
${\sf c}_{i}t={\sf c}_{i}(-{\sf c}_{(J)} (-d))={\sf c}_{i}-{\sf c}_{(J)} (-d)$
$={\sf c}_{i} -{\sf c}_{i}{\sf c}_{(J)}( -d)
=-{\sf c}_{i}{\sf c}_{(J)}( -d)
=-{\sf c}_{(J)}( -d)=t.$
We have shown that ${\sf c}_it=t$ for all $i\in \omega\sim n$, thus $t\in \Nr_{n}\B=\A$. If $s\in S$, we show that $s\leq t$. We know that $s\leq y$. Also $s\leq d$, so $s\cdot -d=0$.
Hence $0={\sf c}_{m_1}\ldots {\sf c}_{m_n}(s\cdot -d)=s\cdot {\sf c}_{m_1}\ldots {\sf c}_{m_n}(-d)$, so
$s\leq -{\sf c}_{m_1}\ldots {\sf c}_{m_n}( -d)$, hence  $s\leq t$ as required. 
We finally check that $t<y$. If not, then
$t=y$ so $y \leq -{\sf c}_{m_1}\ldots {\sf c}_{m_n}(-d)$ and so $y\cdot  {\sf c}_m\ldots {\sf c}_{m_n}(-d)=0$.
But $-d\leq {\sf c}_m\ldots {\sf c}_{m_n}(-d)$,  hence $y\cdot -d\leq y\cdot  {\sf c}_m\ldots {\sf c}_{m_n}(-d)=0.$
Hence $y\cdot -d =0$ and this contradicts that $d<y$. We have proved that $\sum^{\B}X=1$ showing that $\A$ is indeed a complete subalgebra of $\B$.
Since $\B$ is a locally finite,  we can assume 
that $\B=\Fm_T$ for some countable consistent theory $T$.
For each $i<\kappa$, let $\Gamma_i=\{\phi/T: \phi\in X_i\}$.
Let ${\bold F}=(\Gamma_j: j<\kappa)$ be the corresponding set of types in $T$.
Then each $\Gamma_j$ $(j<\kappa)$ is a non-principal and {\it complete $n$-type} in $T$, because each $X_j$ is a maximal filter in $\A=\mathfrak{Nr}_n\B$.
(*) Let $(\Mo_i: i<2^{\omega})$ be a set of countable
models for $T$ that overlap only on principal maximal  types; these exist by Theorem \ref{Shelah}.

The rest is exactly like in the proof of item (2) of Theoren \ref{infinite} resorting to Shelah's Theorem \ref{Shelah}. 
Assume for contradiction that for all $i<2^{\omega}$, there exists
$\Gamma\in \bold F$, such that $\Gamma$ is realized in $\Mo_i$.
Let $\psi:{}2^{\omega}\to \wp(\bold F)$,
be defined by
$\psi(i)=\{F\in \bold F: F \text { is realized in  }\Mo_i\}$.  Then for all $i<2^{\omega}$,
$\psi(i)\neq \emptyset$.
Furthermore, for $i\neq j$, $\psi(i)\cap \psi(j)=\emptyset,$ for if $F\in \psi(i)\cap \psi(j)$, then it will be realized in
$\Mo_i$ and $\Mo_j$, and so it will be principal.
This implies that $|\bold F|=2^{\omega}$ which is impossible. Hence we obtain a model $\Mo\models T$ omitting $\bold X$
in which $\phi$ is satisfiable. The map $f$ defined from $\A=\Fm_T$ to ${\sf Cs}_n^{\Mo}$ (the set algebra based on $\Mo$ \cite[4.3.4]{HMT2})
via  $\phi_T\mapsto \phi^{\Mo},$ where the latter is the set of $n$--ary assignments in
$\Mo$ satisfying $\phi$, omits $\bold X$. Injectivity follows from the facts that $f$ 
is non--zero and $\A$ is simple.  The second item: We can assume that 
$\A\subseteq_c \Nr_n\B$, and $A$ generates $\B$, thus we can assume that $\B\in \sf Lf_{\omega}$ is countable. 
Since $\A\subseteq_c\Nr_n\B\subseteq_c\B$, then $\A\subseteq_c \B$, so that the non-principal types in $\A$ remain so in $\B$. By Theorem \ref{infinite} 
there exists 
$f:\B\to \C$, with $\C\in \sf Cs_{\omega}$, such that $f$ omitting $\bold X$. Then $g=f\upharpoonright \A$ 
will omit $\bold X$ in $\Nr_n\C(\in \Cs_n)$. 
For (2) and (3), we can assume that $\A\subseteq_c \Nrr_n\B$, $\B\in \Lf_{\omega}$. 
We work in $\B$. Using the notation on \cite[p. 216 of proof of Theorem 3.3.4]{Sayed} replacing $\Fm_T$ by $\B$, we have $\bold H=\bigcup_{i\in \lambda}\bigcup_{\tau\in V}\bold H_{i,\tau}$
where $\lambda <\mathfrak{p}$, and $V$ is the weak space ${}^{\omega}\omega^{(Id)}$,  
can be written as a countable union of nowhere dense sets, and so can 
the countable union $\bold G=\bigcup_{j\in \omega}\bigcup_{x\in \B}\bold G_{j,x}$.  
So for any $a\neq 0$,  there is an ultrafilter $F\in N_a\cap (S\setminus \bold H\cup \bold G$)
by the Baire category theorem. This induces a homomorphism $f_a:\A\to \C_a$, $\C_a\in {\sf Cs}_n$ that omits the given types, such that
$f_a(a)\neq 0$. (First one defines $f$ with domain $\B$ as on p.216, then restricts $f$ to $\A$ obtaining $f_a$ the obvious way.) 
The map $g:\A\to \bold P_{a\in \A\setminus \{0\}}\C_a$ defined via $x\mapsto (g_a(x): a\in \A\setminus\{0\}) (x\in \A)$ is as required. 
In case $\A$ is simple, then by properties of $\sf covK$, $S\setminus (\bold H\cup \bold G)$ is non--empty,  so
if $F\in S\setminus (\bold H\cup \bold G)$, then $F$ induces a non--zero homomorphism $f$ with domain $\A$ into a $\Cs_n$ 
omitting the given types. By simplicity of $\A$, $f$ is injective.

To prove independence, it suffices to show that $\sf cov K$ many types may not be omitted because it is consistent that $\sf cov K<2^{\omega}$. 
Fix $2<n<\omega$. Let $T$ be a countable theory such that for this given $n$, in $S_n(T)$, the Stone space of $n$--types,  the isolated points are not dense.
It is not hard to find such theories.
An example, is the  theory of random graphs.
This condition excludes the existence of a prime model for $T$ because $T$ has a prime model $\iff$ the isolated points in 
$S_n(T)$ are dense for all $n$. A prime model which in this context is an atomic model,  omits any family of non--principal types (see the proof of the last item).
We do not want this to happen.
Using exactly the same argument in \cite[Theorem 2.2(2)]{CF}, one can construct 
a family $P$ of non--principal $0$--types (having no free variable) of  $T$,
such that $|P|={\sf covK}$ and $P$ cannot be omitted.
Let $\A=\Fm_T$ and for $p\in P$, let $X_p=\{\phi/T:\phi\in p\}$. Then $X_p\subseteq \Nrr_n\A$, and $\prod X_p=0$,
because $\Nrr_n\A$ is a complete subalgebra of $\A$.
Then we claim that for any $0\neq a$, there is no set algebra $\C$ with countable base
and $g:\A\to \C$ such that $g(a)\neq 0$ and $\bigcap_{x\in X_p}f(x)=\emptyset$.
To see why, let $\B=\Nrr_n\A$. Let $a\neq 0$. Assume for contradiction, that there exists
$f:\B\to \D'$, such that $f(a)\neq 0$ and $\bigcap_{x\in X_i} f(x)=\emptyset$. We can assume that
$B$ generates $\A$ and that $\D'=\Nrr_n\B'$, where $\B'\in \Lf_{\omega}$.
Let $g=\Sg^{\A\times \B'}f$.
We show that $g$ is a homomorphism with $\dom(g)=\A$, $g(a)\neq 0$, and $g$ omits $P$, and for this, 
it suffices to show by symmetry that $g$ is a function with domain $A$. It obviously has co--domain $B'$. 
Settling the domain is easy: $\dom g=\dom\Sg^{\A\times \B'}f=\Sg^{\A}\dom f=\Sg^{\A}\Nrr_{n}\A=\A.$ 
Let $\bold K=\{\A\in \CA_{\omega}: \A=\Sg^{\A}\Nrr_{n}\A\}(\subseteq \sf Lf_{\omega}$). 
We show that $\bold K$ is closed under finite direct products. Assume that $\C$, $\D\in \bold K$, then we have 
$\Sg^{\C\times \D}\Nrr_{n}(\C\times \D)= \Sg^{\C\times \D}(\Nrr_{n}\C\times \Nrr_{n}\D)=
\Sg^{\C}\Nrr_{n}\C\times  \Sg^{\D}\Nrr_{n}\D=\C\times \D$.
Observe that (*):
 $$(a,b)\in g\land \Delta[(a,b)]\subseteq n\implies f(a)=b.$$
(Here $\Delta[(a,b)]$ is the {\it dimension set} of $(a,b)$ defined via $\{i\in \omega: {\sf c}_i(a,b)\neq (a,b)\}$).
Indeed $(a,b)\in \Nrr_{n}\Sg^{\A\times \B}f=\Sg^{\Nrr_{n}(\A\times \B)}f=\Sg^{\Nrr_{\alpha}\A\times \Nrr_{n}\B}f=f.$
Now suppose that $(x,y), (x,z)\in g$. We need to show that $y=z$ proving that $g$ is a function.
Let $k\in \omega\setminus n.$ Let $\oplus$ denote `symmetric difference'. Then (1):
$$(0, {\sf c}_k(y\oplus z))=({\sf c}_k0, {\sf c}_k(y\oplus z))={\sf c}_k(0,y\oplus z)={\sf c}_k((x,y)\oplus(x,z))\in g.$$
Also (2),
$${\sf c}_k(0, {\sf c}_k(y\oplus z))=(0,{\sf c}_k(y\oplus z)).$$
From (2) by observing that $k$ is arbitrarly chosen in $\omega\sim n$, we get (3): 
$$\Delta[(0, {\sf c}_k(y\oplus z))]\subseteq n$$
From (1) and (3) and  (*),  we get  $f(0)={\sf c}_k(y\oplus z)$ for any  $k\in \omega\setminus n.$
Fix $k\in \omega\sim n$. Then by the above, upon observing that $f$ is a homomophism, so that in particular $f(0)=0$, we get 
${\sf c}_k(y\oplus z)=0$. But  $y\oplus z\leq {\sf c}_k(x\oplus z)$, so $y\oplus z=0$, thus $y=z$. 
We have shown that $g$ is a function, $g(a)\neq 0$,  $\dom g=\A$ and $g$ omits $P$.
This contradicts that  $P$, by its construction,  cannot be omitted.
Assuming Martin's axiom, 
we get $\sf cov K=\mathfrak{p}=2^{\omega}$. Together with the above arguments this proves (4).

We now prove (5). Let $\A=\Nrr_n\D$, $\D\in {}\sf Lf_{\omega}$ is countable. Let $\lambda<\mathfrak{p}.$ Let $\bold X=(X_i: i<\lambda)$ be as in the hypothesis.
Let $T$ be the corresponding first order theory, so
that $\D\cong \Fm_T$. Let $\bold X'=(\Gamma_i: i<\lambda)$ be the family of non--principal types in $T$ corresponding to $\bold X$. 
If $\bold X'$ is not omitted, then there is a (countable) realizing tree
for $T$, hence there is a realizing tree for a countable subfamily of $\bold X'$ in the sense of \cite[Definition 3.1]{CF}, 
hence a countable subfamily of $\bold X'$ 
cannot be omitted. Let $\bold X_{\omega}\subseteq \bold X$ be the corresponding countable 
subset of $\bold X$. Assume that $\bold X_{\omega}$ can be omitted in a $\sf Gs_n$, via $f$ say. 
Then  by the same argument used in proving item (4)  $f$ can be lifted to $\Fm_T$ omitting $\bold X'$, 
which is a contradiction.  We leave the part when $\A$ is simple to the reader. 

For (6): If $\A\in {\bold S}_n\Nr_n\CA_{\omega}$, is atomic and has countably many atoms, 
then any complete representation of $\A$, equivalently, an atomic representation of $\A$, equivalently, a representation of $\A$ 
omitting the set of co--atoms is as required. 

\end{proof}
%f $\A$ as above happens to be atomic, then $\bold X$, and in fact any family of non--principal types, will be omitted in a complete representation 
%f $\A$ which exists by \cite[Theorem 5.3.6]{Sayed}. This is an $L_n$ version of the fact that in $L_{\omega, \omega}$, 
%tomic models of atomic theories omit any given family of non-principle types (regardless of their cardinality) which in turn is a model--theoretic expression 
%hat atomic and complete representations in the $\CA_n$ context are the 
%ame. 
Using the full power of Theorem \ref{Shelah}  together with the argument in item (1) of Theorem \ref{sh}, one can replace in the last item of the last corollary $\omega$ 
by any regular uncountable cardinal $\mu$ as explicitly formulated next, 
\begin{theorem}\label{i2} Let $\kappa$ be a regular infinite cardinal and $n<\omega$. Assume that $\A\in \Nr_n\CA_{\omega}$ with $|A|\leq {\kappa}$, 
that $\lambda$ is a cardinal $< 2^{\kappa}$, and that $\bold X=(X_i: i<\lambda)$ is a family of non-principal types  of $\A$.
If the $X_i$s are non--principal ultrafilters of $\A$,  then $\bold X$ can be omitted in a ${\sf Gs}_n$.
%Furthermore, the condition of maximality cannot be dispensed with.
\end{theorem}
We show (algebraically) that the maximality condition cannot be removed when we consider uncountable theories.
\begin{theorem}\label{bsl} Let $\kappa$ be an infinite cardinal. Then there exists an atomless  $\C\in \CA_{\omega}$ such that  for all 
$2<n<\omega$, $|\mathfrak{Nr}_n\C|=2^{\kappa}$, $\mathfrak{Nr}_n\C\in {\sf LCA}_n(={\bf El}\CRCA_n)$, 
but $\mathfrak{Nr}_n\C$ is not completely representable. Thus the non--principal type of co--atoms of $\mathfrak{Nr}_n\C$ 
cannot be omitted. In particular, the condition of maximality in Theorem \ref{i2} cannot be removed.  
\end{theorem}
\begin{proof} We use the following uncountable version of Ramsey's theorem due to
Erdos and Rado:
If $r\geq 2$ is finite, $k$  an infinite cardinal, then
$exp_r(k)^+\to (k^+)_k^{r+1}$,
where $exp_0(k)=k$ and inductively $exp_{r+1}(k)=2^{exp_r(k)}$.
The above partition symbol describes the following statement. If $f$ is a coloring of the $r+1$
element subsets of a set of cardinality $exp_r(k)^+$
in $k$ many colors, then there is a homogeneous set of cardinality $k^+$
(a set, all whose $r+1$ element subsets get the same $f$-value). We will construct the requred $\C\in \CA_{\omega}$ from a relation algebra (to be denoted in a while by $\A$) 
having an `$\omega$-dimensional cylindric basis.' 
To define the  relation algebra, we specify its atoms and forbidden triples.
Let $\kappa$ be the given cardinal in the hypothesis of the Theorem. The atoms are $\Id, \; \g_0^i:i<2^{\kappa}$ and $\r_j:1\leq j<
\kappa$, all symmetric.  The forbidden triples of atoms are all
permutations of $({\sf Id}, x, y)$ for $x \neq y$, \/$(\r_j, \r_j, \r_j)$ for
$1\leq j<\kappa$ and $(\g_0^i, \g_0^{i'}, \g_0^{i^*})$ for $i, i',
i^*<2^{\kappa}.$ 
Write $\g_0$ for $\set{\g_0^i:i<2^{\kappa}}$ and $\r_+$ for
$\set{\r_j:1\leq j<\kappa}$. Call this atom
structure $\alpha$.  
Consider the term algebra $\A$ defined to be the subalgebra of the complex algebra of this atom structure generated by the atoms.
We claim that $\A$, as a relation algebra,  has no complete representation, hence any algebra sharing this 
atom structure is not completely representable, too. Indeed, it is easy to show that if $\A$ and $\B$ 
are atomic relation algebras sharing the same atom structure, so that $\At\A=\At\B$, then $\A$ is completely representable $\iff$ $\B$ is completely representable.
Assume for contradiction that $\A$ has a complete representation with base $\Mo$.  Let $x, y$ be points in the
representation with $\Mo \models \r_1(x, y)$.  For each $i< 2^{\kappa}$, there is a
point $z_i \in \Mo$ such that $\Mo \models \g_0^i(x, z_i) \wedge \r_1(z_i, y)$.
Let $Z = \set{z_i:i<2^{\kappa}}$.  Within $Z$, each edge is labelled by one of the $\kappa$ atoms in
$\r_+$.  The Erdos-Rado theorem forces the existence of three points
$z^1, z^2, z^3 \in Z$ such that $\Mo \models \r_j(z^1, z^2) \wedge \r_j(z^2, z^3)
\wedge \r_j(z^3, z_1)$, for some single $j<\kappa$.  This contradicts the
definition of composition in $\A$ (since we avoided monochromatic triangles).
Let $S$ be the set of all atomic $\A$-networks $N$ with nodes
$\omega$ such that $\{\r_i: 1\leq i<\kappa: \r_i \text{ is the label
of an edge in $N$}\}$ is finite.
Then it is straightforward to show $S$ is an amalgamation class, that is for all $M, N
\in S$ if $M \equiv_{ij} N$ then there is $L \in S$ with
$M \equiv_i L \equiv_j N$, witness \cite[Definition 12.8]{HHbook} for notation.
We have $S$ is symmetric, that is, if $N\in S$ and $\theta:\omega\to \omega$ is a finitary function, in the sense
that $\{i\in \omega: \theta(i)\neq i\}$ is finite, then $N\theta$ is in $S$. It follows that the complex
algebra $\Ca(S)\in \QEA_\omega$.
Now let $X$ be the set of finite $\A$-networks $N$ with nodes
$\subseteq\kappa$ such that:

\begin{enumerate}
\item each edge of $N$ is either (a) an atom of
$\A$ or (b) a cofinite subset of $\r_+=\set{\r_j:1\leq j<\kappa}$ or (c)
a cofinite subset of $\g_0=\set{\g_0^i:i<2^{\kappa}}$ and

\item  $N$ is `triangle-closed', i.e. for all $l, m, n \in \nodes(N)$ we
have $N(l, n) \leq N(l,m);N(m,n)$.  That means if an edge $(l,m)$ is
labelled by $\sf Id$ then $N(l,n)= N(m,n)$ and if $N(l,m), N(m,n) \leq
\g_0$ then $N(l,n)\cdot \g_0 = 0$ and if $N(l,m)=N(m,n) =
\r_j$ (some $1\leq j<\omega$) then $N(l,n)\cdot \r_j = 0$.
\end{enumerate}
For $N\in X$ let $\widehat{N}\in\Ca(S)$ be defined by
$$\set{L\in S: L(m,n)\leq
N(m,n) \mbox{ for } m,n\in \nodes(N)}.$$
For $i\in \omega$, let $N\restr{-i}$ be the subgraph of $N$ obtained by deleting the node $i$.
Then if $N\in X, \; i<\omega$ then $\widehat{\cyl i N} =
\widehat{N\restr{-i}}$.
The inclusion $\widehat{\cyl i N} \subseteq (\widehat{N\restr{-i})}$ is clear.
Conversely, let $L \in \widehat{(N\restr{-i})}$.  We seek $M \equiv_i L$ with
$M\in \widehat{N}$.  This will prove that $L \in \widehat{\cyl i N}$, as required.
Since $L\in S$ the set $T = \set{\r_i \notin L}$ is infinite.  Let $T$
be the disjoint union of two infinite sets $Y \cup Y'$, say.  To
define the $\omega$-network $M$ we must define the labels of all edges
involving the node $i$ (other labels are given by $M\equiv_i L$).  We
define these labels by enumerating the edges and labeling them one at
a time.  So let $j \neq i < \kappa$.  Suppose $j\in \nodes(N)$.  We
must choose $M(i,j) \leq N(i,j)$.  If $N(i,j)$ is an atom then of
course $M(i,j)=N(i,j)$.  Since $N$ is finite, this defines only
finitely many labels of $M$.  If $N(i,j)$ is a cofinite subset of
$\g_0$ then we let $M(i,j)$ be an arbitrary atom in $N(i,j)$.  And if
$N(i,j)$ is a cofinite subset of $\r_+$ then let $M(i,j)$ be an element
of $N(i,j)\cap Y$ which has not been used as the label of any edge of
$M$ which has already been chosen (possible, since at each stage only
finitely many have been chosen so far).  If $j\notin \nodes(N)$ then we
can let $M(i,j)= \r_k \in Y$ some $1\leq k < \kappa$ such that no edge of $M$
has already been labelled by $\r_k$.  It is not hard to check that each
triangle of $M$ is consistent (we have avoided all monochromatic
triangles) and clearly $M\in \widehat{N}$ and $M\equiv_i L$.  The labeling avoided all
but finitely many elements of $Y'$, so $M\in S$. So
$\widehat{(N\restr{-i})} \subseteq \widehat{\cyl i N}$.
Now let $\widehat{X} = \set{\widehat{N}:N\in X} \subseteq \Ca(S)$.
Then we claim that the subalgebra of $\Ca(S)$ generated by $\widehat{X}$ is simply obtained from
$\widehat{X}$ by closing under finite unions.
Clearly all these finite unions are generated by $\widehat{X}$.  We must show
that the set of finite unions of $\widehat{X}$ is closed under all cylindric
operations.  Closure under unions is given.  For $\widehat{N}\in X$ we have
$-\widehat{N} = \bigcup_{m,n\in \nodes(N)}\widehat{N_{mn}}$ where $N_{mn}$ is a network
with nodes $\set{m,n}$ and labeling $N_{mn}(m,n) = -N(m,n)$. $N_{mn}$
may not belong to $X$ but it is equivalent to a union of at most finitely many
members of $\widehat{X}$.  The diagonal $\diag ij \in\Ca(S)$ is equal to $\widehat{N}$
where $N$ is a network with nodes $\set{i,j}$ and labeling
$N(i,j)=\sf Id$.  Closure under cylindrification is given.
Let $\C$ be the subalgebra of $\Ca(S)$ generated by $\widehat{X}$.
Then $\A = \mathfrak{Ra}\C$.
To see why, each element of $\A$ is a union of a finite number of atoms,
possibly a co--finite subset of $\g_0$ and possibly a co--finite subset
of $\r_+$.  Clearly $\A\subseteq\mathfrak{Ra}\C$.  Conversely, each element
$z \in \mathfrak{Ra}\C$ is a finite union $\bigcup_{N\in F}\widehat{N}$, for some
finite subset $F$ of $X$, satisfying $\cyl i z = z$, for $i > 1$. Let $i_0,
\ldots, i_k$ be an enumeration of all the nodes, other than $0$ and
$1$, that occur as nodes of networks in $F$.  Then, $\cyl
{i_0} \ldots
\cyl {i_k}z = \bigcup_{N\in F} \cyl {i_0} \ldots
\cyl {i_k}\widehat{N} = \bigcup_{N\in F} \widehat{(N\restr{\set{0,1}})} \in \A$.  So $\mathfrak{Ra}\C\subseteq \A$.
Thus $\A$ is the relation algebra reduct of $\C\in\CA_\omega$, but $\A$ has no complete representation.
Let $n>2$. Let $\B=\Nrr_n \C$. Then
$\B\in {\sf Nr}_n\CA_{\omega}$, is atomic, but has no complete representation for plainly a complete representation of $\B$ induces one of $\A$. 
In fact, because $\B$  is generated by its two dimensional elements,
and its dimension is at least three, its
$\Df$ reduct is not completely representable \cite[Proposition 4.10]{AU}.
We show that $\B$ is in  ${\bf El}\CRCA_n={\sf LCA}_n$. By Lemma \ref{n} \pe\ has a \ws\ in $G_{\omega}(\At\B)$, hence  \pe\ has a \ws\ in $G_k(\At\B)$ for all $k<\omega$.
Using ultrapowers and an elementary chain argument, we get that there is a countable $\C$ such that
that $\B\equiv \C$, so that $\C$ is atomic and \pe\ has a \ws\ in $G_{\omega}(\At\C)$. Since $\C$ is countable
then by \cite[Theorem 3.3.3]{HHbook2} it is completely representable. 
It remains to show that the $\omega$--dilation $\C$ is atomless. 
For any $N\in X$, we can add an extra node 
extending
$N$ to $M$ such that $\emptyset\subsetneq M'\subsetneq N'$, so that $N'$ cannot be an atom in $\C$.
\end{proof}

\section{Clique guarded semantics}

Fix $2<n<\omega$ (locally well--behaved) relativized representations, in analogy to the relation algebra case dealt with in 
\cite[Chapter 13]{HHbook}; such localized representations are called $m$--square with $2<n<m\leq \omega$, with $\omega$-square representations coinciding with ordinary ones. 
It will always be the case, unless otherwise explicitly indicated,
that $1<n<m<\omega$;  $n$ denotes the dimension. 
But first we recall certain relativized set algebras. A set $V$ ($\subseteq {}^nU$)  is {\it diagonizable} if  
$s\in V\implies s\circ [i|j]\in V$.
We say that $V\subseteq {}^nU$ is {\it locally square} if whenever $s\in V$ and
$\tau: n\to n$, then $s\circ \tau\in V$. 
Let ${\sf D}_n$ (${\sf G}_n$) be the class of set 
algebras whose top elements are diagonizable (locally square) and operations are defined like cylindric set algebra 
of dimension $n$ relativized to the top element $V$ . 
\begin{theorem}\label{m}
\cite{AT}.  
Fix  $2<n<\omega$. Then  ${\sf D}_n$ and ${\sf G}_n$ are 
finitely axiomatizable and have a decidable universal (hence equational) theory. 
%In particular, atomic theories have atomic models for the corresponding guarded fragments of $L_n$, that is to say, Vaught's theorem holds for such logics with no restriction on cardinalities of languages considered.
\end{theorem}
We identify notationally a 
set algebra with its universe.  Let $\Mo$ be a {\it relativized representation} of $\A\in \CA_n$, that is, there exists an injective
homomorphism $f:\A\to \wp(V)$ where $V\subseteq {}^n\Mo$ and $\bigcup_{s\in V} \rng(s)=\Mo$. For $s\in V$ and $a\in \A$,
we may write $a(s)$ for $s\in f(a)$. This notation does not refer to $f$, but whenever used 
then  either $f$ will be clear from context, or immaterial in the context. We may also write $1^{\Mo}$ for $V$.  
We assume that $\Mo$ carries an Alexandrov topology. Let  $\L(\A)^m$ be the first order signature using $m$ variables
and one $n$--ary relation symbol for each element of $\A$.  

{\it An $n$--clique}, or simply a clique,  is a set $C\subseteq \Mo$ such
$(a_0,\ldots, a_{n-1})\in V=1^{\Mo}$
for all distinct $a_0, \ldots, a_{n-1}\in C.$
Let
$${\sf C}^m(\Mo)=\{s\in {}^m\Mo :\rng(s) \text { is an $n$ clique}\}.$$
Then ${\sf C}^m(\Mo)$ is called the {\it $n$--Gaifman hypergraph}, or simply Gaifman hypergraph  of $\Mo$, with the $n$--hyperedge relation $1^{\Mo}$.
The {\it $n$-clique--guarded semantics}, or simply clique--guarded semantics,  $\models_c$, are defined inductively. 
Let $f$ be as above. For an atomic $n$--ary formula $a\in \A$, $i\in{}^nm$, 
and $s\in {}^m\Mo$, $\Mo, s\models_c a(x_{i_0},\ldots x_{i_{n-1}})\iff\ (s_{i_0}, \ldots s_{i_{n-1}})\in f(a).$ 
For equality, given $i<j<m$, $\Mo, s\models_c x_i=x_j\iff s_i=s_j.$
Boolean connectives, and infinitary disjunctions,
are defined as expected.  Semantics for existential quantifiers
(cylindrifiers) are defined inductively for $\phi\in \L(A)^m_{\infty, \omega}$ as follows:
For $i<m$ and $s\in {}^m\Mo$, $\Mo, s\models_c \exists x_i\phi \iff$ there is a $t\in {\sf C}^m(\Mo)$, $t\equiv_i s$ such that 
$\Mo, t\models_c \phi$. Finally, $\Mo\models \Box_i \phi$ $\iff$ $\exists t\in {\sf C}^m(\Mo): t_k\in {\sf int}\{u\in \Mo: \Mo, s_i^u\models \phi\}$.
\begin{definition}\label{cl}
Let $\A$ be an algebra having the signature of $\CA_n$,  $\Mo$ a relativized representation of $\A$ carrying an Alexandrov topology and $\L(\A)^m$  be as above. 
Then $\Mo$ is said to be {\it $m$--square},
if witnesses for cylindrifiers can be found on $n$--cliques. More precisely,
for all  $\bar{s}\in {\sf C}^m(\Mo), a\in \A$, $i<n$,
and for any injective map  $l:n\to m$, if $\Mo\models {\sf c}_ia(s_{l(0)}\ldots, s_{l(n-1)})$,
then there exists $\bar{t}\in {\sf C}^m(\Mo)$ with $\bar{t}\equiv _i \bar{s}$,
and $\Mo\models a(t_{l(0)},\ldots, t_{l(n-1)})$. 
\end{definition}
\begin{definition}\label{network}  
\item An {\it $n$--dimensional atomic network} on an atomic algebra $\A\in \CA_n$  is a map $N: {}^n\Delta\to  \At\A$, where
$\Delta$ is a non--empty finite set of {\it nodes} carrying a topology (hence an Alexandrov topology because the underlying set of nodes is finite), 
denoted by $\nodes(N)$, satisfying the following consistency conditions for all $i<j<n$: 
\begin{enumroman}
\item If $\bar{x}\in {}^n\nodes(N)$  then $N(\bar{x})\leq {\sf d}_{ij}\iff\bar{x}_i=\bar{x}_j$,
\item If $\bar{x}, \bar{y}\in {}^n\nodes(N)$, $i<n$ and $\bar{x}\equiv_i \bar{y}$, then  $N(\bar{x})\leq {\sf c}_iN(\bar{y})$,
\end{enumroman}
\end{definition}
%We denote the class of completely representable ${\sf TCA}_n$s  by ${\sf TCRA}_n$. 
The proof of the  following lemma can be distilled
from its $\sf RA$ analogue \cite[Theorem 13.20]{HHbook},  by reformulating deep concepts
originally introduced by Hirsch and Hodkinson for $\sf RA$s in the $\CA$ context, involving the notions of 
hypernetworks and hyperbasis. 
In the coming proof, we highlight
the main ideas needed to perform such a transfer from $\sf RA$s to $\CA$s
\cite[Definitions 12.1, 12.9, 12.10, 12.25, Propositions 12.25, 12.27]{HHbook}. 
Fix $1<n<\omega$. For $\A\in {\sf D}_n$ with top element $V$, the base of $\A$ is $\Mo=\bigcup_{s\in V} \rng(s)$, so that $V\subseteq {}^n\Mo$ and $\Mo$ is the smallest such set.
We let ${\sf D}_n^{\sf top}$ be the variety of 
${\sf D}_n$'s whose base carry an Alexandrof topology and with interior operators defined like in cylindric set algebras.
In all cases, the $m$--dimensional dilation is a set algebra in ${\sf D}_n^{\sf top}$ as stipulated in the statement of the theorem, will have
top element ${\sf C}^m(\Mo)$, where $\Mo$ is the $m$--relativized representation of the given algebra, and the operations of the dilation
are induced by the $n$-clique--guarded semantics. For a class $\sf K$ of $\sf BAO$s, $\sf K\cap \bf At$ denotes the class of atomic algebras in $\sf K$.
\begin{lemma}\label{flat}\cite[Theorems 13.45, 13.36]{HHbook}.
Assume that $2<n<m<\omega$ and let $\A$ have the signature of $\TCA_n$. Then $\A\in \bold S\Nr_n{\sf D}_m^{\sf top}\iff \A$ has an $m$--square representation. Furthermore, 
if $\A$ is atomic, then $\A$ has a complete $m$--square representation $\iff$ $\A\in \bold S_c\Nr_n({\sf D}_m^{\sf top}\cap \bf At)$.
%We  can replace infinitary $m$-flat and $\CA_m$ by $m$-square and ${\sf D}_m$, respectively.
\end{lemma}

\subsection{Non atom-canonicity and omitting types}

We recall that a class $\sf K$ of Boolean algebras with operators ($\sf BAO$s) is {\it atom--canonical} if whenever $\A\in \sf K$ is atomic and completely additive, then its \de\ completion, 
namely, the complex algebra of its atom structure, namely,  $\Cm\At\A$ is also in $\sf K$.
We use in what follows instances of the so--called blow up and blur construction.  But first a Lemma:
\begin{lemma}\label{n}
Let $2<n<m\leq \omega$. Let $\A\in \CA_n$.  If \pa\ has  \ws\ in $G_{\omega}^m(\At\A$), 
then $\A$ does not have an $m$-square representation.
\end{lemma}
%So fix $2<n<m< \omega$.  
\begin{definition} 
A ${\sf TCA}_n$ atom structure $\bf At$  is {\it weakly representable} if there is an atomic $\A\in {\sf RTCA}_n$ such that $\bf At=\At\A$; it is  {\it strongly representable} if $\Cm {\bf At}\in {\sf RTCA}_n$.
\end{definition}
These two notions are distinct for $2<n<\omega$, cf. \cite{Hodkinson} for the $\CA$ case and the next Theorem.

\subsection{Blowing up and blurring finite rainbow cylindric algebas}

In \cite{ANT} a single blow up and blur construction was used to prove non-atom--canonicity of $\sf RRA$ and ${\sf RCA}_n$ for $2<n<\omega$.
To obtain finer results, we use {\it two blow up and blur constructions} applied to rainbow algebras.
%For the $\RA$ case, following Hirsch and Hodkinson,  we blow up and blur 
%the finite rainbow relation algebra (denoted below by) $\bold R_{4,3}$. For the $\CA$ case we blow up and blur 
%the finite rainbow $\CA_n$  (denoted below by) $\A_{n+1, n}$. 
To put things into a unified perspective, we formulate a definition: 

\begin{definition}\label{blow} Let $\bold M$ be a variety 
of completely additive $\sf BAO$s.

(1) Let $\A\in \bold M$ be a finite algebra. We say that {\it $\D\in \bold M$ is obtained by blowing up and blurring $\A$} if $\D$ is atomic,  
$\A$ does not embed in $\D$, but $\A$ embeds into $\Cm\At\D$. 

(2) Assume that $\sf K\subseteq L\subseteq \bold M$, such that $\bold S\sf L=\sf L$.

(a) We say that {\it $\sf K$ is \underline{not} atom-canonical with respect to $\sf L$} if there exists an atomic $\D\in \sf K$ 
such that $\Cm\At\D\notin \sf L$.  In particular, $\sf K$ is not atom--canonical $\iff$ $\sf K$  not atom-canoincal with respect to itself.

(b)  We say that a finite algebra $\A\in \bold M$ {\it detects} that 
$\bold K$ is not atom--canonical with respect to $\sf L$, if $\A\notin \sf L$, 
and there is a(n atomic)  $\D\in \sf K$ 
obtained by blowing up and blurring $\A$.

\end{definition}
The next proposition and its proof present  the construction in  \cite{ANT} in the framework of definition \ref{blow}.
\begin{proposition}\label{cp} Let $2<n<\omega$. Then for any finite $j>0$, ${\sf RRA}\cap \Ra\CA_{2+j}$ is not atom-canonical with respect to $\sf RRA$, 
and $\RCA_n\cap \Nr_n\CA_{n+j}$ is not atom--canonical with respect to $\RCA_n.$
\end{proposition}

Till the end of this subsection, fix $2<n<\omega$.
The most general exposition of $\CA$ rainbow constructions is given
in \cite[Section 6.2, Definition 3.6.9]{HHbook2} in the context of constructing atom structures from classes of models.
Our models are just coloured graphs \cite{HH}.
Let $\sf G$, $\sf R$ be two relational structures. Let $2<n<\omega$.
Then the colours used are:
\begin{itemize}

\item greens: $\g_i$ ($1\leq i\leq n-2)$, $\g_0^i$, $i\in \sf G$,

\item whites : $\w_i: i\leq n-2,$

\item reds:  $\r_{ij}$ $(i,j\in \sf R)$,

\item shades of yellow : $\y_S: S\text { a finite subset of } \omega$ or $S=\omega$.

\end{itemize}
A {\it coloured graph} is a graph such that each of its edges is labelled by the colours in the above first three items,
greens, whites or reds, and some $n-1$ hyperedges are also
labelled by the shades of yellow.
Certain coloured graphs will deserve special attention.
\begin{definition}
Let $i\in \sf G$, and let $M$ be a coloured graph  consisting of $n$ nodes
$x_0,\ldots,  x_{n-2}, z$. We call $M$ {\it an $i$ - cone} if $M(x_0, z)=\g_0^i$
and for every $1\leq j\leq n-2$, $M(x_j, z)=\g_j$,
and no other edge of $M$
is coloured green.
$(x_0,\ldots, x_{n-2})$
is called  the {\it base of the cone}, $z$ the {\it apex of the cone}
and $i$ the {\it tint of the cone}.
\end{definition}
The rainbow algebra depending on $\sf G$ and $\sf R$ from the class $\bold K$ consisting
of all coloured graphs $M$ such that:
\begin{enumerate}
\item $M$ is a complete graph
and  $M$ contains no triangles (called forbidden triples)
of the following types:
%\vspace{-.2in}
\begin{eqnarray}
&&\nonumber\\
%(1', x, y)&&\mbox{unless }x=y\label{forb:id}\\
(\g, \g^{'}, \g^{*}), (\g_i, \g_{i}, \w_i)
&&\mbox{any }1\leq i\leq  n-2,  \\
%(\g^j_0, \y, \w_f)&&\mbox{unless }f\in P, i\in dom(f)\\
(\g^j_0, \g^k_0, \w_0)&&\mbox{ any } j, k\in \sf G,\\
%\label{forb:pim}(\g^i_0, \g^j_0, \r_{kl})&&\mbox{unless } \set{(i, k), (j, l)}\mbox{ is an order-}\\
%&&\mbox{ preserving partial function }\Z\to\N\nonumber\\
%\label{forb:pim2}(\g_i, \g_j, \r_{kl})&&\mbox{if } i=j\mbox{ but }k\neq l\\
%\label{forb: black}(\y,\y,\y), (\y,\y,{\sf split})\\
\label{forb:match}(\r_{ij}, \r_{j'k'}, \r_{i^*k^*})&&\mbox{unless }  |\set{(j, k), (j', k'), (j^*, k^*)}|=3\\
\end{eqnarray}
and no other triple of atoms is forbidden.

\item If $a_0,\ldots,   a_{n-2}\in M$ are distinct, and no edge $(a_i, a_j)$ $i<j<n$
is coloured green, then the sequence $(a_0, \ldots, a_{n-2})$
is coloured a unique shade of yellow.
No other $(n-1)$ tuples are coloured shades of yellow. Finally, if $D=\set{d_0,\ldots,  d_{n-2}, \delta}\subseteq M$ and
$M\upharpoonright D$ is an $i$ cone with apex $\delta$, inducing the order
$d_0,\ldots,  d_{n-2}$ on its base, and the tuple
$(d_0,\ldots, d_{n-2})$ is coloured by a unique shade
$\y_S$ then $i\in S.$
\end{enumerate}
Let $\sf G$ and $\sf R$ be relational structures as above. Take the set $\sf J$ consisting of all surjective maps $a:n\to \Delta$, where $\Delta\in \bold K$
and define an equivalence relation $\sim$  on this set  relating two such maps iff they essentially define the same graph \cite{HH};
the nodes are possibly different but the graph structure is the same.
Let $\At$ be the atom structure with underlying set $J\sim$. We denote the equivalence class of $a$ by $[a]$. Then define, for $i<j<n$,
the accessibility
relations corresponding to $ij$th--diagonal element, and $i$th--cylindrifier, as follows:

(1) \ \  $[a]\in E_{ij} \text { iff } a(i)=a(j),$

(2) \ \ $[a]T_i[b] \text { iff }a\upharpoonright n\smallsetminus \{i\}=b\upharpoonright n\smallsetminus \{i\},$

This, as easily checked, defines a $\CA_n$
atom structure. The complex $\CA_n$ over this atom structure will be denoted by
$\A_{\sf G, \sf R}$. The dimension of $\A_{\sf G, \sf R}$, always finite and
$>2$, will be clear from context.
For rainbow atom structures, there is a one to one correspondence between atomic networks and coloured graphs \cite[Lemma 30]{HH}, 
so for $2<n<m\leq \omega$, we use the graph versions of the games $G^m_k$, $k\leq \omega$, and  $\bold G^m$ played on rainbow atom 
structures of dimension $m$ \cite[pp.841--842]{HH}. 
The  the atomic $k$ rounded game game $G^m_k$  where the number of nodes are limited to $n$ 
to games on coloured graphs \cite[lemma 30]{HH}.
The game $\bold G^{m}$ lifts to a game on coloured graphs, that  is like the graph games
$G^m_{\omega}$ \cite{HH}, where the number of nodes of graphs played during the $\omega$ rounded
game does not exceed $m$, but \pa\ has the option to re-use nodes.
The typical \ws\ for \pa\ in the graph version 
of both atomic games is bombarding \pe\ with cones having a common base and {\it green} 
tints until she runs out of (suitable) {\it reds}, that is to say, reds whose indicies do not match \cite[4.3]{HH}.

\begin{definition} A $\CA_n$ atom structure $\bf At$  is {\it weakly representable} if there is an atomic $\A\in {\sf RCA}_n$ such that $\bf At=\At\A$; 
 it is  {\it strongly representable} if $\Cm {\bf At}\in {\sf RCA}_n$.
\end{definition}
These two notions are distinct, cf. \cite{Hodkinson} and the following  Theorem\ref{can}; see also the forthcoming Theorem \ref{fl}.
Let $\sf V\subseteq W$ be varieties of $\sf BAO$s. W say that $\sf V$ is {\it atom-canonical with respet to $\sf W$} if for any atomic $\A\in  \V$, 
its \de\ completion, namely, $\Cm\At\A$ is in $\sf W$. Let $\Rd_{ca}$ denote `cylindric reduct'
\begin{theorem}\label{can}
Let $2<n<\omega$ and $t(n)=n(n+1)/2+1.$ The variety ${\sf TRCA}_n$ is not-atom canonical 
with respect to $\bold S{\sf Nr}_n\CA_{t(n)}$. In fact, there is a countable atomic simple $\A\in {\sf TRCA}_n$ 
such that $\Rd_{ca}\Cm\At\A$ 
does not have an $t(n)$-square,{\it a fortiori} $t(n)$- flat,  representation.
\end{theorem}
\begin{proof} 
The proof is long and uses many ideas in  \cite{Hodkinson}. 
We will highlight only the differences in detail from the proof in \cite{Hodkinson} needed to make our result work.
When parts of the proof coincide we will be more sketchy.
The proof is divided into four parts:

1: {\bf Blowing up and blurring  $\B_f$ forming a weakly representable atom structure $\bf At$}:
Take the finite rainbow ${\sf CA}_n$,  $\B_f$
where the reds $\sf R$ is the complete irreflexive graph $n$, and the greens
are  $\{\g_i:1\leq i<n-1\}
\cup \{\g_0^{i}: 1\leq i\leq n(n-1)/2\}$, endowed with the cylindric operations.
We will show $\B$ detects that ${\sf RCA}_n$ is not atom-canonical with respect to $\bold S\Nr_n\CA_{t(n)}$ with $t(n)$ as specified in the statement of the 
theorem. 
Denote the finite atom structure of $\B_f$ by ${\bf At}_f$; 
so that ${\bf At}_f=\At(\B_{f})$.
One then defines a larger the class of coloured graphs like in \cite[Definition 2.5]{Hodkinson}. Let $2<n<\omega$.
Then the colours used are like above except that each red is `split' into $\omega$ 
many having `copies' the form $\r_{ij}^l$ with $i<j<n$ and $l\in \omega$, with an additional shade of red $\rho$ such that 
the consistency conditions for the new reds (in addition to the usual rainbow consistency conditions) are as follows:
\begin{itemize}
\item $(r^{i}_{jk}, r^{i}_{j'k'}, r^{i^*}_{j^*k^*})$ unless $i=i'=i^*$ 
and $|\{(j, k), (j', k'), (j^*, k^*)\}|=3$
 \item $(\r, \rho, \rho)$ and $(\r, \r^*, \rho)$, where $\r, \r^*$ are any reds. 
%Let $\GG$ be the class of all models of this {\it extended rainbow first order theory}.
%The extra shade of red $\rho$  will be used as a label.
\end{itemize}
The consistency conditions can be coded in an $L_{\omega, \omega}$  theory $T$ having signature the reds with $\rho$ 
together with all other colours like in \cite[Definitio 3.6.9]{HHbook2}. The theory $T$ is only a first order theory (not an $L_{\omega_1, \omega}$ theory) 
because the number of greens is finite which is not the case with \cite{HHbook2} where the number of available greens are countably infinite coded by an infinite disjunction.
One construct an $n$-homogeneous  model $\Mo$ is as a countable limit of finite models of $T$ 
using a game played between \pe\ and \pa like in \cite[Theorem 2.16]{Hodkinson}.  
In the rainbow game 
\pa\ challenges \pe\  with  {\it cones} having  green {\it tints $(\g_0^i)$}, 
and \pe\ wins if she can respond to such moves. This is the only way that \pa\ can force a win.  \pe\ 
has to respond by labelling {\it appexes} of two successive cones, having the {\it same base} played by \pa.
By the rules of the game, she has to use a red label. She resorts to  $\rho$ whenever
she is forced a red while using the rainbow reds will lead to an inconsistent triangle of reds;  \cite[Proposition 2.6, Lemma 2.7]{Hodkinson}. The number of greens make  
\cite[Lemma 3.10]{Hodkinson} work with the same proof.
using only finitely many green and not infinitely many. 
The \ws\ is implemented by \pe\  using the red label $\rho$  
that comes to her rescue  whenever she runs out of `rainbow reds', so she can always and consistently respond with an extended coloured graph. 
This proof will  
is implemented in the framework of an entirely analagous blow up and blur constructions applied to strikingly similar finite rainbow atom structures in \cite{HHbook}.
In both cases, the relational structures 
$\sf G$ and $\sf R$ used
satisfy $|\sf G|=|\sf R|+1$. For $\sf RA$, ${\sf R}=3$ and for $\CA_n$s, ${\sf R}=n$ (the dimension), where
the finite ordinals $3$ and $n$ are viewed as complete irreflexive graphs. 
\footnote{Worthy of note, is that it is commonly accepted that relation algebras have dimension three being a natural habitat for three variable first order 
logic.  Nevertheless, sometimes it is argued that the dimension should be three and a half in the somewhat loose sense that $\sf RA$s 
lie `halfway' between $\CA_3$ and $\CA_4$ manifesting behaviour of each.} 
From Hodkinson's construction in \cite{Hodkinson}, we know that $\Cm\Bb(\B_f,  \r, \omega)\notin \bold S\Nr_n\CA_m$ for some finite $m>n$, where ${\Bb}(\B_f,  \r, \omega)$ denotes the result
of blowing up $\B_f$ by splitting each red atom into $\omega$-many ones, to be denoted henceforth by $\A$.
 The (semantical) argument used in \cite{Hodkinson} does not give any information on 
the value of such $m$. By truncating  the greens to be $n(n+1)/2$ (instead of the `overkill' of infinitely many in \cite{Hodkinson}), and using a syntactical blow up and blur construction,   
we could pin down such a value of $m$, namely, $m=n+t(n)$ (=number of greens +$t(n)$)
by showing in a moment that that although ${\Bb}(\B_f,  \r, \omega))$ containing the term algbra
is representable, but {\it not  completely representable}. On the other hand, its completion, namely, 
$\Cm\At({\Bb}(\B_f,  \r, \omega))$ will be outside $\bold S\Nr_n\CA_{t(n)}$. Proving representability  
$\Tm \At({\Bb}(\B_f,  \r, \omega))$ can be done by completely representing its canonical 
extension, in a fairly simple step by step manner. The atom structure $\bf At={\sf Uf}({\Bb}(\B_f,  \r, \omega))$ of $\A$ consists of principal ultrafilters generated by atoms, 
together with only one non-princilple ultraflter, 
that can be identified with the shade of red $\rho$. This is needed for representing $\A$, but not completely; the atom structure $\bf At$ is not and cannot be completely representable; 
it is not even stongly representable. As a matter of fact, it is just  weakly representable, with all these notions of reprsentability for atom structures are taken from  \cite{HHbook2}.  

2. {\bf Representing a term algebra (and its completion) as (generalized) set algebras:}
Having $\Mo$ at hand, one constructs  two atomic $n$--dimensional set algebras based on $\Mo$, sharing the same atom structure and having 
the same top element.  
The atoms of each will be the set of coloured graphs, seeing as how, quoting Hodkinson \cite{Hodkinson} such coloured graphs are `literally indivisible'. 
Now $L_n$ and $L_{\infty, \omega}^n$ are taken in the rainbow signature (without $\rho$). Continuing like in {\it op.cit}, deleting the one available red shade, set
$W = \{ \bar{a} \in {}^n\Mo : \Mo \models ( \bigwedge_{i < j <n} \neg \rho(x_i, x_j))(\bar{a}) \},$
and for $\phi\in L_{\infty, \omega}^n$, let
$\phi^W=\{s\in W: \Mo\models \phi[s]\}.$
Here $W$ is the set of all $n$--ary assignments in
$^n\Mo$, that have no edge labelled by $\rho$.
%We note that $\rho$ is used by \pe\ infinitely many times during the game forming a `red clique' in $M$ \cite{Hodkinson}.
Let $\A$  be the relativized set algebra with domain
$\{\varphi^W : \varphi \,\ \textrm {a first-order} \;\ L_n-
\textrm{formula} \}$  and unit $W$, endowed with the
usual concrete quasipolyadic operations read off the connectives.
Classical semantics for {\it $L_n$ rainbow formulas} and their
semantics by relativizing to $W$ coincide \cite[Proposition 3.13]{Hodkinson} {\it but not with respect to 
$L_{\infty,\omega}^n$ rainbow formulas}.
Hence the set algebra $\A$ is isomorphic to a cylindric  set algebra of dimension $n$ 
having top element $^n\Mo$, so $\A$
is simple, in fact its $\Df$ reduct is simple.
Let $\E=\{\phi^W: \phi\in L_{\infty, \omega}^n\}$
\cite[Definition 4.1]{Hodkinson}
with the operations defined like on $\A$ the usual way. $\Cm\bf At$ is a complete $\CA_n$ and, so like in \cite[Lemma 5.3]{Hodkinson}
we have an isomorphism from $\Cm\bf At$  to $\E$ defined
via $X\mapsto \bigcup X$.
Since $\At\A=\At\Tm(\At\A)$, which we refer to only by $\bf At$,  
and $\Tm\At\A\subseteq \A$, hence $\Tm\At\A= Tm\bf At$  is representable.
The atoms of $\A$, $\Tm\At\A$ and $\Cm\At\A=\Cm \bf At$ are the coloured graphs whose edges are {\it not labelled} by $\rho$.
These atoms are uniquely determined by the interpretation in $\Mo$ of so-called $\sf MCA$ formulas in the rainbow signature of $\bf At$  as in
\cite[Definition 4.3]{Hodkinson}. Giving $\Mo$ the descrete topology make both algebras topological set algebras, who extra unary modal operators
all coincide with th identity operator, which we denote by adding with a sight abue of notation, denote by the notation used for their cylindric reducts. No confusion is like to ensue. 
Though the shades of red is  {\it outside}  signature, it was as a label
during an $\omega$--rounded game played on labelled finite graphs--which can be seen as finite models in  the extended signature having size $\leq n$--
in which \pe\ had a \ws,  enabling her to
construct the required $\Mo$ as a countable limit of the finite graphs played during the game. The construction entails that any subgraph (substructure)
of $\Mo$ of size $\leq n$, is independent of its location in $\Mo$;
it is uniquely determined by its isomorphism type.
A relativized set algebra $\A$ based on $\Mo$ was constructed
by discarding all assignments whose edges are labelled
by these shades of reds,  getting a set of $n$--ary sequences $W\subseteq {}^n\Mo$. This $W$ is definable in $^n\Mo$ by an $L_{\infty, \omega}$ formula
and the semantics with respect to $W$ coincides with classical Tarskian semantics (on $^n\Mo$) for formulas of the
signature taken in $L_n$ (but not for formulas taken in $L_{\infty, \omega}^n$).
This was proved in both cases using certain $n$ back--and--forth systems, thus $\A$ is representable classically,
in fact it (is isomorphic to a set algebra that) has base $\Mo$.

{\bf The heart and soul of the proof;} In the set algebra $A$, 
one replaces the red label by suitable
non--red binary relation symbols within an $n$ back--and--forth system, so that one  can
adjust that the system maps a tuple $\bar{b} \in {}^n \Mo \backslash W$ to a tuple
$\bar{c} \in W$ and this will preserve any formula
containing the non--red symbols that are
`moved' by the system.  In fact, all
injective maps of size $\leq n$ defined on $\Mo$ modulo an appropriate
permutation of the reds will
form an $n$ back--and--forth system.
This set algebra $\A$
was further atomic, countable, and simple (with top element $^n\Mo$). The subgraphs of size $\leq n$ of $\Mo$ whose edges are not labelled by any shade of red are 
the atoms of $\A$, expressed syntactically by $\sf MCA$ formulas. The  \de\  of $\A$, in symbols $\Cm\At\A,$
has top element $W$, but it is not in $\bold S\Nr_n\CA_{t(n)}$ in case of the rainbow construction, let alone  representable,
In this constructions `the shades of  red' -- which can be intrinsically identified with
non--principal ultrafilters in $\A$,  were used as colours, together with the principal ultrafilters
to completely represent  $\A^+$, inducing a representation of $\A$.  Non--representability for  Monk like constructions 
use an uncontollablr Ramsey number determined 
by Ramsey's theory. The non neat--embeddability of thr rainbow like algebra 
$\Cm\At\A$ in the present more stronger cas, we used {\it a finite number of greens} that gave us more delicate information on
when $\Cm\At\A$   stops to be representable.  
The reds, particularly $\rho$ acting as a non-princple ultrafilter  had to do 
with representing $\A$ using non-atomic networks.

3. {\bf Embedding $\A_{n+1, n}$ into $\Cm(\At({{\mathfrak{Bb}}}(\A_{n+1, n}, \r, \omega)))$:} Let ${\sf CRG}_f$ be the class of coloured graphs on 
${\bf At}_f$ and $\sf CRG$ be the class of coloured graph on $\bf At$. We 
can (and will) assume that  ${\sf CRG}_f\subseteq \sf CRG$.
Write $M_a$ for the atom that is the (equivalence class of the) surjection $a:n\to M$, $M\in \sf CGR$.
Here we identify $a$ with $[a]$; no harm will ensue.
We define the (equivalence) relation $\sim$ on $\bf At$ by
$M_b\sim N_a$, $(M, N\in {\sf CGR}):$
\begin{itemize}
\item $a(i)=a(j)\Longleftrightarrow b(i)=b(j),$

\item $M_a(a(i), a(j))=\r^l\iff N_b(b(i), b(j))=\r^k,  \text { for some $l,k$}\in \omega,$

\item $M_a(a(i), a(j))=N_b(b(i), b(j))$, if they are not red,

\item $M_a(a(k_0),\dots, a(k_{n-2}))=N_b(b(k_0),\ldots, b(k_{n-2}))$, whenever
defined.
\end{itemize}
We say that $M_a$ is a {\it copy of $N_b$} if $M_a\sim N_b$ (by symmetry $N_b$ is a copy of $M_a$.) 
Indeed, the relation `copy of' is an equivalence relation on $\bf At$.  An atom $M_a$ is called a {\it red atom}, if $M_a$ has at least one red edge. 
Any red atom has $\omega$ many copies, that are {\it cylindrically equivalent}, in the sense that, if $N_a\sim M_b$ with one (equivalently both) red,
with $a:n\to N$ and  $b:n\to M$, then we can assume that $\nodes(N) =\nodes(M)$ 
and that for all $i<n$, $a\upharpoonright n\sim\{i\}=b\upharpoonright n\sim \{i\}$.
In $\Cm\bf At$, we write $M_a$ for $\{M_a\}$ 
and we denote suprema taken in $\Cm\bf At$, possibly finite, by $\sum$.
Define the map $\Theta$ from $\A_{n+1, n}=\Cm{\bf At}_f$ to $\Cm\bf At$,
by specifying first its values on ${\bf At}_f$,
via $M_a\mapsto \sum_jM_a^{(j)}$ where $M_a^{(j)}$ is a copy of $M_a$. 
So each atom maps to the suprema of its  copies.  
This map is well-defined because $\Cm\bf At$ is complete. 
We check that $\Theta$ is an injective homomorphism. Injectivity is easy..
%In \cite{mlq}, where only the $\CA$ case is addressed preservation of  cylindrifiers and diagonal elements is proved. 
We check preservation of all the $\CA_n$ extra Boolean operations.  
%The Boolean join is obvious.
\begin{itemize}

\item Diagonal elements. Let $l<k<n$. Then:
\begin{align*}
M_x\leq \Theta({\sf d}_{lk}^{\Cm{\bf At}_f})&\iff\ M_x\leq \sum_j\bigcup_{a_l=a_k}M_a^{(j)}\\
&\iff M_x\leq \bigcup_{a_l=a_k}\sum_j M_a^{(j)}\\
&\iff  M_x=M_a^{(j)}  \text { for some $a: n\to M$ such that $a(l)=a(k)$}\\
&\iff M_x\in {\sf d}_{lk}^{\Cm\bf At}.
\end{align*}

\item Cylindrifiers. Let $i<n$. By additivity of cylindrifiers, we restrict our attention to atoms 
$M_a\in {\bf At}_f$ with $a:n\to M$, and $M\in {\sf CRG}_f\subseteq \sf CRG$. Then: 

$$\Theta({\sf c}_i^{\Cm{\bf At}_f}M_a)=f (\bigcup_{[c]\equiv_i[a]} M_c)
=\bigcup_{[c]\equiv_i [a]}\Theta(M_c)$$
$$=\bigcup_{[c]\equiv_i [a]}\sum_j M_c^{(j)}=\sum_j \bigcup_{[c]\equiv_i [a]}M_c^{(j)}
=\sum _j{\sf c}_i^{\Cm\bf At}M_a^{(j)}$$
$$={\sf c}_i^{\Cm\bf At}(\sum_j M_a^{(j)})
={\sf c}_i^{\Cm\bf At}\Theta(M_a).$$

\item Modlities: Coinciding with the identity map on both sides ae trivially presrved.
\end{itemize}

4.: {\bf \pa\ has  a  \ws\ in $G^{t(n)}\At(\B_f)$; and the required result:} It is straightforward to show that 
\pa\ has \ws\ first in the  \ef\ forth  private 
game played between \pe\ and \pa\ on the complete
irreflexive graphs $n+1(\leq n(n-1)/2+1)$ and $n$ in 
$n+1$ rounds
${\sf EF}_{n+1}^{n+1}(n+1, n)$ \cite [Definition 16.2]{HHbook2}
since $n+1$ is `longer' than $n$.  Using (any) $p>n$ many pairs of pebbles available on the board \pa\ can win this game in $n+1$ many rounds.
 \pa\  lifts his \ws\ from the lst private \ef\ forth game to the graph game on ${\bf At}_f=\At(\B_f)$ 
\cite[pp. 841]{HH} forcing a
win using $t(n)$ nodes. 
One uses the $n(n-1)/2+2$ green relations in the usual way to force a red clique $C$, say with $n(n-1)/2+2$.  
Pick any point $x\in C$.  Then there are $>n(n-1)/2$ points $y$ in $C\{x\}$.  There are only $n(n-1)/2$ red relations.  
So there must be distinct $y, z\in C\{x\}$ such that $(x,y)$ and $(x,z)$ both have the same red label 
(it will be some $\r^m_{ij}$ for $i<j<n$).  But $(y,z)$ is also red, and this contradicts \cite [Definition 2.5(2), 4th bullet point]{Hodkinson}. 
In more detail, \pa\ bombards  \pe\ with cones
having  common
base and distinct green  tints until \pe\ is forced to play an inconsistent red triangle (where indicies of reds do not match).
He needs $n-1$ nodes as the base of cones, 
plus $|P| + 2$ more nodes, where $P=\{(i, j): i<j<n\}$ forming a red clique,  
triangle with two edges satisfying the same $r^m_p$ for $p\in P$.  
Calculating,  we get $t(n)=n-1+n(n-1)/2+2=n(n+1)/2+1$
%He needs $n-1$ nodes for bases of cones plus $n+1$ appexes of cones played to reveal that the reds are not enough. The total number is $n-1+n+1=2n\leq t(n)$.
By Lemma \ref{n}, $\B_f\notin
\bold S\Nr_n\CA_{t(n)}$ when $2<n<\omega)$. Since $\B_f$ is finite, then $\B_f\notin \bold S\Nr_n\CA_{t(n)}$, 
because $\B_f$ coincides with its canonical extension and for any $\D\in \CA_n$, $\D\in \bold S\Nr_n\CA_{2n}\implies \D^+\in \bold S_c\Nr_n\CA_{2n}$.
But $\B_f$ embeds into $\Rd_{ca}\Cm\At\A^{\sf top}$,
hence $\Rd_{ca}\Cm\At\A^{\sf top}$
is outside the variety $\bold S{\sf Nr}_n\CA_{t(n)}$, as well. 
By the second part of Lemma \ref{n}, the required follows. 
\end{proof}

\subsection{On non-elementary classes related to the class of completly representatable algebras}

\begin{definition}\label{hypernetwork} For an $n$--dimensional atomic network $N$  on an atomic $\CA_n$ and for  $x,y\in \nodes(N)$, set  $x\sim y$ if
there exists $\bar{z}$ such that $N(x,y,\bar{z})\leq {\sf d}_{01}$.
Define the  equivalence relation $\sim$ over the set of all finite sequences over $\nodes(N)$ by
$\bar x\sim\bar y$ iff $|\bar x|=|\bar y|$ and $x_i\sim y_i$ for all
$i<|\bar x|$. (It can be easily checked that this indeed an equivalence relation).
A \emph{ hypernetwork} $N=(N^a, N^h)$ over an atomic $\CA_n$
consists of an $n$--dimensional  network $N^a$
together with a labelling function for hyperlabels $N^h:\;\;^{<
\omega}\!\nodes(N)\to\Lambda$ (some arbitrary set of hyperlabels $\Lambda$)
such that for $\bar x, \bar y\in\; ^{< \omega}\!\nodes(N)$
if $\bar x\sim\bar y \Rightarrow N^h(\bar x)=N^h(\bar y).$
If $|\bar x|=k\in \N$ and $N^h(\bar x)=\lambda$, then we say that $\lambda$ is
a $k$-ary hyperlabel. $\bar x$ is referred to as a $k$--ary hyperedge, or simply a hyperedge.
%We may remove the superscripts $a$ and $h$ if no confusion is likely to ensue.
A hyperedge $\bar{x}\in {}^{<\omega}\nodes(N)$ is {\it short}, if there are $y_0,\ldots, y_{n-1}$
that are nodes in $N$, such that
$N(x_i, y_0, \bar{z})\leq {\sf d}_{01}$
or $\ldots N(x_i, y_{n-1},\bar{z})\leq {\sf d}_{01}$
for all $i<|x|$, for some (equivalently for all)
$\bar{z}.$ Otherwise, it is called {\it long.}
This game involves, besides the standard 
cylindrifier move,  
two new amalgamation moves.
Concerning his moves, this game with $m$ rounds ($m\leq \omega$), call it  $\bold H_m$, \pa\ can play a cylindrifier move, like before but now played on $\lambda$---
neat hypernetworks ($\lambda$ a constant label).
Also \pa\ can play a \emph{transformation move} by picking a
previously played hypernetwork $N$ and a partial, finite surjection
$\theta:\omega\to\nodes(N)$, this move is denoted $(N, \theta)$.  \pe's
response is mandatory. She must respond with $N\theta$.
Finally, \pa\ can play an
\emph{amalgamation move} by picking previously played hypernetworks
$M, N$ such that
$M\restr {\nodes(M)\cap\nodes(N)}=N\restr {\nodes(M)\cap\nodes(N)},$
and $\nodes(M)\cap\nodes(N)\neq \emptyset$.
This move is denoted $(M,
N).$
To make a legal response, \pe\ must play a $\lambda_0$--neat
hypernetwork $L$ extending $M$ and $N$, where
$\nodes(L)=\nodes(M)\cup\nodes(N)$.
\end{definition}

%The next lemma is the $\sf CA$ analogue of \cite[Theorem 39]{r}. The proof is very similar, but it proves more by allowing infinite conjunctions in constructing a certain 
%model (denoted by $\cal M$) as clarified below. This is the essential diference from the 
%proof in {\it op.cit} which can (and will be) strengthened using the same `infinitary addition.'
The next Lemma will be needed to prove Theorem \ref{rainbow1} and Corollary \ref{rainbow2}, which are the main results in this section. With Theorem \ref{can}, they constitute the core of this article.
\begin{theorem}\label{gripneat} Let $\alpha$ be a countable atom structure. If \pe\ has a \ws\ in $\bold H_{\omega}(\alpha)$, 
then there exists a complete $\D\in \RCA_{\omega}$ such that 
$\Cm\alpha\cong \Nrr_n\D$
%and $\alpha\cong \At\Nrr_n\D$. 
In particular, $\Cm\alpha\in \Nr_n\CA_{\omega}$ 
%and $\alpha\in \At\Nr_{n}\CA_{\omega}$.  
\end{theorem} 
\begin{proof} 
Fix some $a\in\alpha$. The game $\bold H_{\omega}$ is designed so that using \pe\ s \ws\ in the game $\bold H_{\omega}(\alpha)$ 
one can define a
nested sequence $M_0\subseteq M_1,\ldots$ of $\lambda$--neat hypernetworks
where $M_0$ is \pe's response to the initial \pa-move $a$, such that:
If $M_r$ is in the sequence and $M_r(\bar{x})\leq {\sf c}_ia$ for an atom $a$ and some $i<n$,
then there is $s\geq r$ and $d\in\nodes(M_s)$
such that  $M_s(\bar{y})=a$,  $\bar{y}_i=d$ and $\bar{y}\equiv_i \bar{x}$.
In addition, if $M_r$ is in the sequence and $\theta$ is any partial
isomorphism of $M_r$, then there is $s\geq r$ and a
partial isomorphism $\theta^+$ of $M_s$ extending $\theta$ such that
$\rng(\theta^+)\supseteq\nodes(M_r)$ (This can be done using \pe's responses to amalgamation moves).
Now let $\M_a$ be the limit of this sequence, that is $\M_a=\bigcup M_i$, the labelling of $n-1$ tuples of nodes
by atoms, and hyperedges by hyperlabels done in the obvious way using the fact that the $M_i$s are nested.
Let $L$ be the signature with one $n$-ary relation for
each $b\in\alpha$, and one $k$--ary predicate symbol for
each $k$--ary hyperlabel $\lambda$.
{\it Now we work in $L_{\infty, \omega}.$}
For fixed $f_a\in\;^\omega\!\nodes(\M_a)$, let
$\U_a=\set{f\in\;^\omega\!\nodes(\M_a):\set{i<\omega:g(i)\neq
f_a(i)}\mbox{ is finite}}$.
We  make $\U_a$ into the base of an $L$ relativized structure 
${\cal M}_a$  allowing a clause for infinitary disjunctions.
In more detail,  for $b\in\alpha,\; l_0, \ldots, l_{n-1}, i_0 \ldots, i_{k-1}<\omega$, \/ $k$--ary hyperlabels $\lambda$,
and all $L$-formulas $\phi, \phi_i, \psi$, and $f\in U_a$:
\begin{eqnarray*}
{\cal M}_a, f\models b(x_{l_0}\ldots,  x_{l_{n-1}})&\iff&{\cal M}_a(f(l_0),\ldots,  f(l_{n-1}))=b,\\
{\cal M}_a, f\models\lambda(x_{i_0}, \ldots,x_{i_{k-1}})&\iff&  {\cal M}_a(f(i_0), \ldots,f(i_{k-1}))=\lambda,\\
{\cal M}_a, f\models\neg\phi&\iff&{\cal M}_a, f\not\models\phi,\\
{\cal M}_a, f\models (\bigvee_{i\in I} \phi_i)&\iff&(\exists i\in I)({\cal M}_a,  f\models\phi_i),\\
{\cal M}_a, f\models\exists x_i\phi&\iff& {\cal M}_a, f[i/m]\models\phi, \mbox{ some }m\in\nodes({\cal M}_a).
\end{eqnarray*}
%We are now
%working with (weak) set algebras  whose semantics is induced by $L_{\infty, \omega}$ formulas in the signature $L$,
%instead of first order ones.
For any such $L$-formula $\phi$, write $\phi^{{\cal M}_a}$ for
$\set{f\in\U_a: {\cal M}_a, f\models\phi}.$
Let
$D_a= \set{\phi^{{\cal M}_a}:\phi\mbox{ is an $L$-formula}}$ and
$\D_a$ be the weak set algebra with universe $D_a$. 
Let $\D=\bold P_{a\in \alpha} \D_a$. Then $\D$ is a  generalized {\it complete} weak set algebra \cite[Definition 3.1.2 (iv)]{HMT2}.
%By complete we mean that infinite suprema exists. 
%This is true because we chose to work with $L_{\infty, \omega}$ while forming the dilations $\D_a$ $(a\in \alpha)$. 
%Each $\D_a$ is complete, hence so is 
%their product $\D$.
Now we show 
%that $\alpha\cong \At\mathfrak{Nr}_n\D$ 
$\Cm\alpha\cong \mathfrak{Nr}_n\D$.
Let  $X\subseteq \mathfrak{Nr}_n\D$. Then by completeness of $\D$, we get that
$d=\sum^{\D}X$ exists.  Assume that  $i\notin n$, then
${\sf c}_id={\sf c}_i\sum X=\sum_{x\in X}{\sf c}_ix=\sum X=d,$
because the ${\sf c}_i$s are completely additive and ${\sf c}_ix=x,$
for all $i\notin n$, since $x\in \mathfrak{Nr}_n\D$.
We conclude that $d\in \mathfrak{Nr}_n\D$, hence $d$ is an upper bound of $X$ in $\mathfrak{Nr}_n\D$. Since 
$d=\sum_{x\in X}^{\D}X$ there can be no $b\in \mathfrak{Nr}_n\D$ $(\subseteq \D)$ with $b<d$ such that $b$ is an upper bound of $X$ for else it will be an upper bound of $X$ in $\D$. 
Thus $\sum_{x\in X}^{\mathfrak{Nr}_n\D}X=d$ 
We have shown that   
$\mathfrak{Nr}_n\D$ is complete. 
Making the legitimate identification  
$\mathfrak{Nr}_n\D\subseteq_d \Cm\alpha$ by density,
we get that  $\mathfrak{Nr}_n\D=\Cm\alpha$ 
(since $\mathfrak{Nr}_n\D$ is complete),  
hence $\Cm\alpha\in {\sf Nr}_n\CA_{\omega}$. 
\end{proof}
\begin{theorem}\label{rainbow1}Let $\bold O\in \{\bold S_c, \bold S_d, \bold I\}$, where $\bold S_d$ denotes the operation of forming dense subalgebras and let $k\geq 3$. 
Then the class of frames ${\sf K}_k=\{\F: \Cm\F\in \bold O\Nr_n\CA_{n+k}\}$ 
is not elementary. In particular, the class of extremely representable algebras up to $n+k$ is not elementary.
%\item The class of frames ${\sf L}_m=\{\F: \Cm\F\in \Ra\CA_m\}$ is not elementary for $m\geq 6$.  
\end{theorem}
\begin{proof} 

{\bf (1)  Defining a rainbow-like atom structure $\alpha$:}  We use the algebra in \cite[Theorem 5.12]{mlq}.
The algebra $\C_{\Z, \N}(\in \RCA_n$) based on $\Z$ (greens) and $\N$ (reds) denotes the rainbow-like algebra used in {\it op.cit} 
which is defined as follows:  
%Take the a rainbow--like $\CA_n$, call it $\C$, based on the ordered structure $\Z$ and $\N$.
The reds ${\sf R}$ is the set $\{\r_{ij}: i<j<\omega(=\N)\}$ and the green colours used 
constitute the set $\{\g_i:1\leq i <n-1\}\cup \{\g_0^i: i\in \Z\}$. 
In complete coloured graphs the forbidden triples are like 
the usual rainbow constructions based on $\Z$ and $\N$, with a significant addition:   
First  the colours used are:
\begin{itemize}

\item greens: $\g_i$ ($1\leq i\leq n-2)$, $\g_0^i$, $i\in \Z$,

\item whites : $\w_i: i\leq n-2,$

\item reds:  $\r_{ij}$ $(i,j\in \N)$,

\item shades of yellow : $\y_S: S\text { a finite subset of } \omega$ or $S=\omega$.

\end{itemize}

The rainbow algebra depending on $\N$ and $\Z$ from the class $\bold K$ consisting
of all coloured graphs $M$ such that:
\begin{enumerate}
\item $M$ is a complete graph
and  $M$ contains no triangles (called forbidden triples)
of the following types:
%\vspace{-.2in}
\begin{eqnarray}
&&\nonumber\\
%(1', x, y)&&\mbox{unless }x=y\label{forb:id}\\
(\g, \g^{'}, \g^{*}), (\g_i, \g_{i}, \w_i)
&&\mbox{any }1\leq i\leq  n-2,  \\
%(\g^j_0, \y, \w_f)&&\mbox{unless }f\in P, i\in dom(f)\\
(\g^j_0, \g^k_0, \w_0)&&\mbox{ any } j, k\in A,\\
%\label{forb:pim}(\g^i_0, \g^j_0, \r_{kl})&&\mbox{unless } \set{(i, k), (j, l)}\mbox{ is an order-}\\
%&&\mbox{ preserving partial function }\Z\to\N\nonumber\\
%\label{forb:pim2}(\g_i, \g_j, \r_{kl})&&\mbox{if } i=j\mbox{ but }k\neq l\\
%\label{forb: black}(\y,\y,\y), (\y,\y,\bb)\\
\label{forb:match}(\r_{ij}, \r_{j'k'}, \r_{i^*k^*})&&\mbox{unless }i=i^*,\; j=j'\mbox{ and }k'=k^*
\end{eqnarray}
Observe that this 1.7 is not as item 1.3 in the proof of Theorem \ref{can}. Here inconsistent triples of reds are defined differently.

\item {\it The triple  $(\g^i_0, \g^j_0, \r_{kl})$ is also forbidden if $\{(i, k), (j, l)\}$ is not an order preserving partial function from
$\Z\to\N$}
%It can be shown that \pe\ has a \ws\ in $\bold H_k(\At\C_{\Z, \N} )$ for all $k\in \omega$.
\end{enumerate}
It is proved in {\it op.cit} that \pe\ has a \ws\ in $G_k(\At\C_{\Z, \N} )$ for all $k\in \omega$, so that $\C_{\Z, \N} \in {\bf El}{\sf CRCA}_n$. 
%One can define a $k$ rounded game $\bold H_k$ with $k\leq \omega$, such that for a $\CA_n$  atom structure $\alpha$, \pe\ has a \ws\ in 
%$\bold H_{\omega}(\alpha)$ implies that $\Cm\alpha\in \Nr_n\CA_{\omega}$ and $\alpha\in \At\Nr_n\CA_{\omega}$. 
With some more effort it can be proved that \pe\ has a \ws\ $\sigma_k$ say  in $\bold H_k(\At\C_{\Z, \N})$ for all $k\in \omega$.
Let $\alpha=\At\C_{\Z, \N}$.

{\bf (2) \pe\ has a \ws\ in $\bold H_{\omega}(\alpha)$:}  
We describe \pe's strategy in dealing with labelling hyperedges in $\lambda$--neat hypernetworks, where $\lambda$ is a constant label kept on short hyperedges
and, not to interrupt the main stream, we defer the rest of the highly technical proof to the appendix.
In a play, \pe\ is required to play $\lambda$--neat hypernetworks, so she has no choice about the
the short edges, these are labelled by $\lambda$. In response to a cylindrifier move by \pa\
extending the current hypernetwork providing a new node $k$,
and a previously played coloured hypernetwork $M$
all long hyperedges not incident with $k$ necessarily keep the hyperlabel they had in $M$.
All long hyperedges incident with $k$ in $M$
are given unique hyperlabels not occurring as the hyperlabel of any other hyperedge in $M$.
In response to an amalgamation move, which involves two hypernetworks required to be amalgamated, say $(M,N)$
all long hyperedges whose range is contained in $\nodes(M)$
have hyperlabel determined by $M$, and those whose range is contained in $\nodes(N)$ have hyperlabels determined
by $N$. If $\bar{x}$ is a long hyperedge of \pe\ s response $L$ where
$\rng(\bar{x})\nsubseteq \nodes(M)$, $\nodes(N)$ then $\bar{x}$
is given
a new hyperlabel, not used in 
any previously played hypernetwork and not used within $L$ as the label of any hyperedge other than $\bar{x}$.
This completes her strategy for labelling hyperedges.
In \cite{mlq} it is shown that \pe\ has a \ws\ in $G_k(\At\C_{\Z, \N})$ where $0<k<\omega$ is the number of rounds.
With some more effort it can be prove that \pe\ has a \ws\ $in \bold H_k(\At\C)$ for each $k<\omega$, call it $\sigma_k$. 
We can assume that $\sigma_k$ is deterministic.
Let $\D$ be a non--principal ultrapower of $\C_{\Z, \N}$.  Then \pe\ has a \ws\ $\sigma$ in $\bold H_{\omega}(\At\D)$ --- essentially she uses
$\sigma_k$ in the $k$'th component of the ultraproduct so that at each
round of $\bold H_{\omega}(\At\D)$,  \pe\ is still winning in co-finitely many
components, this suffices to show she has still not lost.
We can also assume that $\C_{\Z, \N}$  is countable by replacing it by the term algebra. 
%If not then replace it by its subalgebra generated by the countably many atoms
%(the term algebra); \ws\ s here will depend only on the atom structure, so they persist
%for both players. 
Now one can use an
elementary chain argument to construct countable elementary
subalgebras $\C_{\Z, \N}=\A_0\preceq\A_1\preceq\ldots\preceq\ldots \D$ in this manner.
One defines  $\A_{i+1}$ be a countable elementary subalgebra of $\D$
containing $\A_i$ and all elements of $\D$ that $\sigma$ selects
in a play of $G(\At\D)$ in which \pa\ only chooses elements from
$\A_i$. Now let $\B=\bigcup_{i<\omega}\A_i$.  This is a
countable elementary subalgebra of $\D$, hence necessarily atomic,  and \pe\ has a \ws\ in
$\bold H_{\omega}(\At\B)$ and $\B\equiv \C_{\Z, \N}$ . Thus 
$\At\B\in \At\Nr_n\CA_{\omega}$ and $\Cm\At\B\in \Nr_n\CA_{\omega}$. (This does not imply that $\B\in \Nr_n\CA_{\omega}$, cf. example \ref{SL}).  
Since $\B\subseteq_d \Cm\At\B$,
$\B\in \bold S_d\Nr_n\CA_{\omega}$, so $\B\in \bold S_c\Nr_n\CA_{\omega}$.  Being countable, it follows by \cite[Theorem 5.3.6]{Sayedneat} that $\B\in \CRCA_n$. 

{\bf  (3) \pa\ has a \ws\ in $\bold G^{n+3}(\alpha)$:} We now show that hat \pa\ has a \ws\ in $\bold G^{n+3}(\At\C_{\Z, \N} )$ (denoted in {\it op.cit} 
by $F^{n+3}(\At\C_{\Z, \N} )$), 
hence by 
Lemma \ref{n}, $\C_{\Z, \N} \notin \bold S_c\Nr_n\CA_{n+3}$. 
It can be shown that \pa\ has a \ws\ in the graph version of the game 
$\bold G^{n+3}(\At\C)$ played on coloured graphs \cite{HH}.
The rough idea here, is that, as is the case with \ws's of \pa\ in rainbow constructions, 
\pa\ bombards \pe\ with cones having distinct green tints demanding a red label from \pe\ to appexes of succesive cones.
The number of nodes are limited but \pa\ has the option to re-use them, so this process will not end after finitely many rounds.
The added order preserving condition relating two greens and a red, forces \pe\ to choose red labels, one of whose indices form a decreasing 
sequence in $\N$.  In $\omega$ many rounds \pa\ 
forces a win, 
so by lemma \ref{n}, $\C\notin \bold S_c{\sf Nr}_n\CA_{n+3}$.
%The detailed proof is in \cite[Theorem 5.12]{mlq}.
More rigorously, \pa\ plays as follows: In the initial round \pa\ plays a graph $M$ with nodes $0,1,\ldots, n-1$ such that $M(i,j)=\w_0$
for $i<j<n-1$
and $M(i, n-1)=\g_i$
$(i=1, \ldots, n-2)$, $M(0, n-1)=\g_0^0$ and $M(0,1,\ldots, n-2)=\y_{\Z}$. This is a $0$ cone.
In the following move \pa\ chooses the base  of the cone $(0,\ldots, n-2)$ and demands a node $n$
with $M_2(i,n)=\g_i$ $(i=1,\ldots, n-2)$, and $M_2(0,n)=\g_0^{-1}.$
\pe\ must choose a label for the edge $(n+1,n)$ of $M_2$. It must be a red atom $r_{mk}$, $m, k\in \N$. Since $-1<0$, then by the `order preserving' condition
we have $m<k$.
In the next move \pa\ plays the face $(0, \ldots, n-2)$ and demands a node $n+1$, with $M_3(i,n)=\g_i$ $(i=1,\ldots, n-2)$,
such that  $M_3(0, n+2)=\g_0^{-2}$.
Then $M_3(n+1,n)$ and $M_3(n+1, n-1)$ both being red, the indices must match.
$M_3(n+1,n)=r_{lk}$ and $M_3(n+1, r-1)=r_{km}$ with $l<m\in \N$.
In the next round \pa\ plays $(0,1,\ldots n-2)$ and re-uses the node $2$ such that $M_4(0,2)=\g_0^{-3}$.
This time we have $M_4(n,n-1)=\r_{jl}$ for some $j<l<m\in \N$.
Continuing in this manner leads to a decreasing
sequence in $\N$. We have proved the required.

{\bf (4): Proving the required} Let $\bold K$ be a class between $\bold S_d\Nr_n\CA_{\omega}\cap {\sf CRCA}_n$ 
and $\bold S_c\Nr_n\CA_{n+3}$. 
Then  $\bold K$ is not elementary, because $\C_{\Z, \N} \notin \bold S_c\Nr_n\CA_{n+3}(\supseteq \bold K)$, 
$\B\in \bold S_d\Nr_n\CA_{\omega}\cap {\sf CRCA}_n(\subseteq \bold K)$,
and $\C_{\Z, \N}\equiv \B$. 
It clearly suffices to show that $\bold K_k=\{\A\in \CA_n\cap {\bf At}: \Cm\At\A\in \bold O\Nr_n\CA_k\}$ is not elementary. 
\pe\ has a \ws\ in $\bold H_{\omega}(\alpha)$ for some countable atom structure $\alpha$, 
$\Tm\alpha\subseteq_d \Cm\alpha\in \bold {\sf Nr}_n\CA_{\omega}$ and $\Tm\alpha\in {\sf CRCA}_n$.
Since $\C_{\Z, \N}\notin \bold S_c{\sf Nr}_n\CA_{n+3}$,  then 
$\C_{\Z, \N}=\Cm\At\C_{\Z, \N}\notin {\bold K}_k$, $\C_{\Z, \N}\equiv \Tm\alpha$ 
and $\Tm\alpha\in {\bold K}_k$ 
because $\Cm\alpha\in \Nr_n\CA_{\omega}\subseteq \bold S_d\Nr_n\CA_{\omega}\subseteq \bold S_c\Nr_n\CA_{\omega}$.
We have shown that $\C_{\Z, \N}\in {\bf El}{\bold K}_k\sim {\bold K}_k$, proving the required. 
%(The above proof does not work for $\bold S\Nr_n\CA_{n+k}$ $(k\geq 3)$, because $\C_{\Z, \N}\in {\sf RCA}_n=\bold S\Nr_n\CA_{\omega}$).

%Suppose that $\bold K$ is any class between $\bold S_d\Ra\CA_{\omega}\cap \sf CRRA$ and $\bold S_c\Ra\CA_5$. Then 
%$\Cm\beta\notin \bold K(\subseteq \bold S_c\Ra\CA_{\omega}$),
%$\Tm\alpha\in \bold K(\supseteq \bold S_d\Ra\CA_{\omega}\cap \sf CRRA$) 
%and $\Tm\alpha\cong \Cm\beta$, hence $\bold K$ is not elementary.  
%Since $\alpha\equiv \beta$, $\alpha\in \At(\Ra\CA_{\omega}\cap \sf CRRA)$ and 
%$\beta\notin \At(\bold S_c\Ra\CA_5)$, we get the second required.

\end{proof}
Having the necessary tools at our disposal in the previous proof,  we obtain the following  subtantial generalization of Theorem \ref{HH}. Let $\bold S_d$ denote the operation of forming dense 
subagbras.
First observe  that from Theorem \ref{bsl}  the two classes $\CRCA_n$ and  $\bold S_d\Nr_n\CA_{\omega}$ are 
mutually distinct.
\begin{corollary} Let $2<n<\omega$ and $m\geq n+3$. Then any any class $\sf K$ such that 
$\bold S_d\Nr_n\CA_{\omega}\cap \CRCA_n\subseteq {\sf K}\subseteq \bold S_c\Nr_n\CA_m$ is not elementary.
\end{corollary}
\begin{proof} The rainbow-like atomic algebra algebra $\C_{\Z, \N}$  having countably many atoms is 
outside $\bold S_c\Nr_n\CA_{n+3}$. However, the algebra $\C_{\Z, \N}$ is elementary equivalent to a countable atomic  
algebra $\B$ such that \pe\ has a \ws\ $\in \bold H_{\omega}\At\B$, so by Lemma \ref{gripneat}, 
$\Cm\At\B\in \Nr_n\CA_{\omega}$. Thus $\B\in \bold S_d\Nr_n\CA_{\omega}\subseteq \bold S_c\Nr_n\CA_{\omega}$. Being countable, by 
\cite[Theorem 5.3.3]{Sayedneat}, we get $\B\in \CRCA_n$. 
\end{proof}

For a variety $\sf V$ of $\sf BAO$s, $\sf Str(V)=\{\F: \Cm\F\in \V\}$. Fix finite $k>2$.  Then $\V_k={\sf Str}(\bold S\Nr_n\CA_{n+k})$ is not elementary 
$\implies   \V_k$ is not-atom canonical because (argueing contrapositively) in the case of atom--canonicity, 
we get that ${\sf Str}(\bold S\Nr_n\CA_{n+k})={\sf At}(\bold S\Nr_n\CA_{n+k})$, and the last class is elementary \cite[Theorem 2.84]{HHbook}. However, the converse implication may fail. 
In particular, we do not know whether ${\sf Str}(\bold S\Nr_n\CA_{n+k})$, for a particular finite $k\geq 3$, is elementary or not.
Nevertheless, it is easy to show that {\it there has to be a finite $k<\omega$ such that $\V_j$ is not elementary for all $j\geq k$}:
\begin{theorem}\label{truncate}  
There is a finite $k$ such that the class of strongly representable algbras up to $n+k$
is not elementary. 
\end{theorem}
\begin{proof}
 We show that ${\sf Str}(\bold S\Nr_n\CA_{m})$ is not elementary for some finite $m\geq n+2$. 
Let $(\A_i: i\in \omega)$ be a sequence of (strongly) representable $\CA_n$s with $\Cm\At\A_i=\A_i$
and $\A=\Pi_{i/U}\A_i$ is not strongly representable with respect to any non-principal ultrafilter $U$ on $\omega$.
Such algebras exist \cite{HHbook2}. Hence $\Cm\At\A\notin \bold S\Nr_n\CA_{\omega}=\bigcap_{i\in \omega}\bold S\Nr_n\CA_{n+i}$, 
so $\Cm\At\A\notin \bold S\Nr_n\CA_{l}$ for all $l>m$, for some $m\in \omega$, $m\geq n+2$. 
But for each such $l$, $\A_i\in \bold S\Nr_n\CA_l(\subseteq {\sf RCA}_n)$, 
so $(\A_i:i\in \omega)$ is a sequence of algebras such that $\Cm\At(\A_i)\in \bold S\Nr_n\CA_{l}$ $(i\in I)$, but 
$\Cm(\At(\Pi_{i/U}\A_i))=\Cm\At(\A)\notin \bold S\Nr_n\CA_l$, for all $l\geq m$.
\end{proof}

\section{An application on $\sf OTT$s, Vaught's Theorem for $\sf TopL_{\alpha}$}

\begin{definition} Let $\R$ be an atomic  relation algebra.  An {$n$--dimensional basic matrix}, or simply a matrix  
on $\R$, is a map $f: {}^2n\to \At\R$ satsfying the 
following two consistency 
conditions $f(x, x)\leq \Id$ and $f(x, y)\leq f(x, z); f(z, y)$ for all $x, y, z<n$. For any $f, g$ basic matrices
and $x, y<m$ we write $f\equiv_{xy}g$ if for all $w, z\in m\setminus\set {x, y}$ we have $f(w, z)=g(w, z)$.
We may write $f\equiv_x g$ instead of $f\equiv_{xx}g$.  
\end{definition}
\begin{definition}\label{b}
An {\it $n$--dimensional cylindric basis} for an atomic relaton algebra 
$\R$ is a set $\cal M$ of $n$--dimensional matrices on $\R$ with the following properties:
\begin{itemize}
\item If $a, b, c\in \At\R$ and $a\leq b;c$, then there is an $f\in {\cal M}$ with $f(0, 1)=a, f(0, 2)=b$ and $f(2, 1)=c$
\item For all $f,g\in {\cal M}$ and $x,y<n$, with $f\equiv_{xy}g$, there is $h\in {\cal M}$ such that
$f\equiv_xh\equiv_yg$. 
\end{itemize}
\end{definition}
For the next lemma, we refer the reader to \cite[Definition 12.11]{HHbook} 
for the definition of of hyperbasis for relation algebras as well as to 
\cite[Chapter 13, Definitions 13.4, 13.6]{HHbook} for the notions 
of $n$--flat and $n$--square representations for  relation algebras ($n>2$) to be generalized below to cylindric algebras, cf. Definition \ref{cl}.
For a relation algebra $\R$, recall that $\R^+$ denotes its canonical extension.

\begin{lemma}\label{i} Let $\R$ be  a relation algebra and $3<n<\omega$.  Then the following hold:
\begin{enumerate}
\item $\R^+$ has an $n$--dimensional infinite basis $\iff\ \R$ has an infinite $n$--square representation.

\item $\R^+$ has an $n$--dimensional infinite hyperbasis $\iff\ \R$ has an infinite $n$--flat representation.
\end{enumerate}
\end{lemma}
\begin{proof} \cite[Theorem 13.46, the equivalence $(1)\iff (5)$ for basis, and the equivalence $(7)\iff (11)$ for hyperbasis]{HHbook}.
\end{proof}

One can construct a $\CA_n$ in a natural way from an $n$--dimensional cylindric basis which can be viewed as an atom structure of a $\CA_n$
(like in \cite[Definition 12.17]{HHbook} addressing hyperbasis).
For an atomic  relation algebra $\R$ and $l>3$, we denote by ${\sf Mat}_n(\At\R)$ the set of all $n$--dimensional basic matrices on $\R$.
${\sf Mat}_n(\At\R)$ is not always an $n$--dimensional cylindric basis, but sometimes it is,
as will be the case described next. On the other 
hand, ${\sf Mat}_3(\At\R)$ is always a  $3$--dimensional cylindric basis; a result of Maddux's, so that 
$\Cm{\sf Mat}_3(\At\R)\in \CA_3$. 
The following definition to be used in the sequel is taken from \cite{ANT}:
\begin{definition}\label{strongblur}\cite[Definition 3.1]{ANT}
Let $\R$ be a relation algebra, with non--identity atoms $I$ and $2<n<\omega$. Assume that  
$J\subseteq \wp(I)$ and $E\subseteq {}^3\omega$.
\begin{enumerate}
\item We say that $(J, E)$  is an {\it $n$--blur} for $\R$, if $J$ is a {\it complex $n$--blur} defined as follows:   
\begin{enumarab}
\item Each element of $J$ is non--empty,
\item $\bigcup J=I,$
\item $(\forall P\in I)(\forall W\in J)(I\subseteq P;W),$
\item $(\forall V_1,\ldots V_n, W_2,\ldots W_n\in J)(\exists T\in J)(\forall 2\leq i\leq n)
{\sf safe}(V_i,W_i,T)$, that is there is for $v\in V_i$, $w\in W_i$ and $t\in T$,
we have
$v;w\leq t,$ 
\item $(\forall P_2,\ldots P_n, Q_2,\ldots Q_n\in I)(\forall W\in J)W\cap P_2;Q_n\cap \ldots P_n;Q_n\neq \emptyset$.
\end{enumarab}
and the tenary relation $E$ is an {\it index blur} defined  as 
in item (ii) of \cite[Definition 3.1]{ANT}.

\item We say that $(J, E)$ is a {\it strong $n$--blur}, if it $(J, E)$ is an $n$--blur,  such that the complex 
$n$--blur  satisfies:
$$(\forall V_1,\ldots V_n, W_2,\ldots W_n\in J)(\forall T\in J)(\forall 2\leq i\leq n)
{\sf safe}(V_i,W_i,T).$$ 
\end{enumerate}
\end{definition}

The following theorem concisely summarizes 
the construction in \cite{ANT} and says some more easy facts.

\begin{theorem}\label{ANT} Let $2<n\leq l<\omega$.
Let $\R$ be a finite relation algebra with an $l$--blur $(J, E)$ where $J$ is the $l$--complex blur and $E$ is the index blur, as in definition \ref{strongblur}.  
\begin{enumerate}
\item Then for ${\cal R}={\Bb}(\R, J, E)$, with atom structure $\bf At$ obtained by blowing up and blurring $\R$ 
(with underlying set is denoted by $At$ on \cite[p.73]{ANT}), the set of $l$ by $l$--dimensional matrices 
${\bf At}_{ca}={\sf Mat}_l(\bf At)$ is an $l$--dimensional cylindric basis, that is a weakly representable atom structure \cite[Theorem 3.2]{ANT}. 
The algebra ${\Bb}_l(\R, J, E)$, with last notation as in \cite[Top of p. 78]{ANT} having atom structure ${\bf At}_{ra}$ is in $\RCA_l$. Furthermore, 
$\R$ embeds into $\Cm{\bf At}$ which embeds into $\Ra\Cm({\bf At}_{ca}).$ 

\item For very $n<l$, 
there is an $\R$ having a strong $l$--blur 
but no finite representations. Hence $\bf At$ obtained by blowing up and blurring $\R$ and 
the $\CA_n$ atom structure ${\bf At}_{ca}$ as in the previous item are 
not strongly representable.

\item Let $m<\omega$. If $\R$ is  as in the hypothesis, 
$(J, E)$ is a strong $l$--blur, and $\R$ has no $m$--dimensional hyperbasis, 
then $l<m$.

\item If $n=l<m<\omega$ and $\R$ as above has no infinite $m$--dimensional hyperbasis, then $\Cm\At{\Bb}_l(\R, J, E)\notin 
\bold S{\sf Nr}_n\CA_m$, and the latter class is not atom--canonical.

\item If $2<n\leq l<m\leq \omega$, and $(J, E)$ is a strong $m$--blur,  definition \ref{strongblur},
then $(J, E)$ is a strong $l$--blur,  ${\Bb}_l(\R, J, E)\cong \mathfrak{Nr}_l{\Bb}_m(\R, J, E)$ and 
${\cal R}\cong \Ra {\Bb}_l(\R, J, E)\cong \Ra{\Bb}_m(\R, J, E).$ 
\end{enumerate}
\end{theorem}
\begin{proof} Cf.  \cite[ For notation, cf. p.73, p.80, and for proofs cf. Lemmata 3.2, 4.2, 4.3]{ANT}. 
We start by an outline of (1).  Let $\R$ be as in the hypothesis. The idea is to blow up and blur $\R$ in place of the Maddux algebra 
$\mathfrak{E}_k(2, 3)$ dealt with in \cite[Lemma 5.1]{ANT}   (where $k<\omega$ is the number of non--identity atoms 
and it depends on $l$). 
Let $3<n\leq l$. We blow up and blurr $\R$ as in the hypothesis. $\R$ is blown up by splitting all of the atoms each to infinitely many.
$\R$ is blurred by using a finite set of blurs (or colours) $J$. This can be expressed by the product ${\bf At}=\omega\times \At \R\times J$,
which will define an infinite atom structure of a new
relation algebra. (One can view such a product as a ternary matrix with $\omega$ rows, and for each fixed $n\in \omega$,  we have the rectangle
$\At \R\times J$.)
Then two partitions are defined on $\bf At$, call them $P_1$ and $P_2$.
Composition is re-defined on this new infinite atom structure; it is induced by the composition in $\R$, and a ternary relation $E$
on $\omega$, that `synchronizes' which three rectangles sitting on the $i,j,k$ $E$--related rows compose like the original algebra $\R$.
This relation is definable in the first order structure $(\omega, <)$.
The first partition $P_1$ is used to show that $\R$ embeds in the complex algebra of this new atom structure, namely $\Cm \bf At$, 
The second partition $P_2$ divides $\bf At$ into {\it finitely many (infinite) rectangles}, each with base $W\in J$,
and the term algebra denoted in \cite{ANT} by ${\Bb}(\R, J, E)$ over $\bf At$, consists of the sets that intersect co--finitely with every member of this partition.
On the level of the term algebra $\R$ is blurred, so that the embedding of the small algebra into
the complex algebra via taking infinite joins, do not exist in the term algebra for only finite and co--finite joins exist
in the term algebra. 
The algebra ${\Bb}(\R, J, E)$ is representable using the finite number of blurs. These correspond to non--principal ultrafilters
in the Boolean reduct, which are necessary to
represent this term algebra, for the principal ultrafilter alone would give a complete representation,
hence a representation of the complex algebra and this is impossible.
Thereby, in particular, as stated in theorem \ref{ANT} an atom structure that is weakly representable but not strongly representable is obtained.
Because $(J, E)$ is a complex set of $l$--blurs, this atom structure has an $l$--dimensional cylindric basis, 
namely, ${\bf At}_{ca}={\sf Mat}_l(\bf At)$. The resulting $l$--dimensional cylindric term algebra $\Tm{\sf Mat}_l(\bf At)$, 
and an algebra $\C$ having tatom structure ${\bf At}_{ca}$ denoted in \cite{ANT} by 
${\Bb}_l(\R, J, E)$, such that $\Tm{\sf Mat}_l({\bf At})\subseteq \C\ \subseteq \Cm{\sf Mat}_l(\bf At)$ 
is shown to be  representable.

For (2):  The Maddux relation algebra $\mathfrak{E}_k(2,3)$ 
Like in \cite[Lemma 5.1]{ANT},  one take $l\geq 2n-1$, $k\geq (2n-1)l$, $k\in \omega$, and then take the finite
integral relation algebra ${\mathfrak E}_k(2, 3)$ 
where $k$ is the number of non-identity atoms in
${\mathfrak E}_k(2,3)$. 
with $k$ depending on $l$ as in \cite[Lemma 5.1]{ANT} is the required $\R$ in (2).

We prove (3). Let $(J, E)$ be the strong $l$--blur of $\R$. Assume  for contradiction that $m\leq l$. Then we get by \cite[item (3), p.80]{ANT},  
that  $\A={\sf Bb}_n(\R, J, E)\cong \mathfrak{Nr}_n{\Bb}_l(\R, J, E)$.  But the cylindric $l$--dimensional algebra ${\Bb}_l(\R, J, E)$ is atomic,  having atom structure  
${\sf Mat}_l \At({\Bb}(\R, J, E))$, so $\A$ has an atomic $l$--dilation.
Hence $\A=\Nr_n\D$ where $\D\in \CA_l$ is atomic.
But $\R\subseteq_c \Ra\Nr_n\D\subseteq_c \Ra\D$. 
Hence $\R$ has a complete $l$--flat representation, 
hence a complete $m$--flat representation, because $m<l$ and $l\in \omega$. 
This is a contradiction.  We prove (4). Assume that $\R$ is as in the hypothesis.
Take $\B={\Bb}_n(\R, J, E)$. Then by the above $\B\in \RCA_n$. We claim that $\C=\Cm\At\B\notin \bold S{\sf Nr}_n\CA_{m}$.
To see why, suppose for contradiction that $\C\subseteq \mathfrak{Nr}_n\D$,
where $\D$ is atomic, with $\D\in \CA_{m}$.  Then $\C$ has a (necessarily infinite $m$--flat representation), hence 
$\Ra\C$ has an infinite $m$--flat representation as an $\RA$.  But $\R$ embeds into $\Cm\At({\Bb}(\R, J, E)$) which, in turn, 
embeds  into $\Ra\C$, so $\R$ has an infinite $m$--flat representation. By lemma \ref{i}, 
$\R$ has a $m$--dimensional infinite hyperbases which
contradicts the hypothesis.    

Now we prove (the last) item (5). For $2<n\leq l<m<\omega.$
If the $m$--blur happens to be {\it strong}, in the sense of definition \ref{strongblur} and $n\leq l<m$
then we get by \cite[item (3) pp. 80]{ANT},  that ${\Bb}_l(\R, J, E)\cong \mathfrak{Nr}_l{\Bb}_m(\R, J, E)$.
This is proved by defining  an embedding 
$h:\Rd_l{\Bb}_m(\R, J, E)\to {\Bb}_l(\R, J, E)$ 
via  $x\mapsto \{M\upharpoonright l\times l: M\in x\}$ and showing that  
$h\upharpoonright  \mathfrak{Nr}_l{\Bb}_m(\R, J, E)$ 
is an isomorphism onto ${\Bb}_l(\R, J, E)$ \cite[p.80]{ANT}. 
Surjectiveness uses the condition $(J5)_l$.  The resulting $l$--dimensional cylindric term algebra $\Tm{\sf Mat}_l(\bf At)$, 
and an algebra $\C$ having tatom structure ${\bf At}_{ca}$ denoted in \cite{ANT} by 
$\Bb_l(\R, J, E)$, such that $\Tm{\sf Mat}_l({\bf At})\subseteq \C\ \subseteq \Cm{\sf Mat}_l(\bf At)$ 
is shown to be  representable. The complex algebra $\Cm{\sf Mat}_l(\bf At)=\Cm \At\C$
is outside in $\bold S{\sf Nr}_n\CA_{m}$, because $\R$ 
embeds into $\Cm\bf At$ which embeds into $\Ra\Cm{\sf Mat}_l(\bf At)$, so if 
$\Cm{\sf Mat}_l({\bf At})\in \bold S{\sf Nr}_n\CA_{m}$, 
then $\R\in {\sf Ra}\bold S{\sf Nr}_n\CA_m\subseteq \bold S{\sf Ra}\CA_m$
which is contrary to assumption.
\end{proof}

Fix $2<n\leq l<m\leq \omega$.  The statement $\Psi(l, m)$ is: 

{\it There is an atomic, countable, topological and complete $L_n$ theory $T$, such that the type $\Gamma$ consisting of co--atoms is realizable 
in every $m$-- square model, but any formula isolating this type has
to contain more than $l$ variables.}  

By an $m$--square model $\Mo$ of $T$ we understand an $m$--square representation of the algebra $\Fm_T$ with base $\Mo$.
Let ${\sf VT}(l, m))=\neg \Psi(l, m)$, short for {\bf Vaught's Theorem holds `at the parameters $l$ and $m$'} where
by definition, we stipulate that ${\sf VT}(\omega, \omega)$ is just Vaught's Theorem for $L_{\omega, \omega}$: Countable 
atomic theories  have countable atomic models.
For $2<n\leq l<m\leq \omega$  and $l=m=\omega$,  we investigate the likelihood and plausability of the following 
statement which we abbreviate by (**):
$\VT(l, m)\iff l=m=\omega.$
In the next Theorem several conditions are given implying $\Psi(l, m)$ for various values of $l$ and $m$.
$\Psi(l, m)_f$ is the formula obtained from $\Psi(l, m)$ be replacing square by flat.
In the first item by no infinite 
$\omega$--dimensional hyperbasis (basis), we understand no representation on an 
infinite base.  By $\omega$--flat (square) representation, we mean an ordinary representation, and by complete $\omega$--flat (square) representation, we mean 
a complete representation. 
\begin{theorem} \label{main} Let $2<n\leq l<m\leq \omega$. Then
every item implies the immediately following one. 
\begin{enumerate}
\item There exists a finite relation algebra $\R$ algebra with a strong $l$--blur and no infinite $m$--dimensional hyperbasis, 
\item There is a countable atomic $\A\in \Nr_n\CA_l\cap \sf RCA_n$ such that $\Cm\At\A$ does not have an 
$m$--flat representation,
\item There is a countable atomic $\A\in \Nr_n\CA_l\cap \sf RCA_n$ such that $\Cm\At\A\notin \bold S\Nr_n\CA_m$, 
\item There is a countable atomic $\A\in \Nr_n\CA_l\cap \sf RCA_n$ such that $\A$ has no complete infinitry $m$--flat  representation,
\item There is a countable atomic $\A\in \Nr_n\CA_l\cap \sf RCA_n$ such that $\A\notin \bold S_c\Nr_n\CA_m$,
\item $\Psi(l, m)_f$ is true,
\item $\Psi(l', m')_f$ is true for any $l'\leq l$ and $m'\geq m$.
\end{enumerate}
The same implications hold upon replacing infinite $m$--dimensional hyperbasis by $m$--dimensional relational basis (not necessarily infinite), 
$m$--flat by $m$--square and $\bold S\Nr_n\CA_m$ by $\bold S\Nr_n{\sf D}_m$. Furthermore, in the new chain of implications every item implies the corresponding item in 
Theorem \ref{main}. In particular, 
$\Psi(l, m)\implies \Psi(l, m)_f$.  
\end{theorem}

\begin{proof}  Let $\R$ be as in the hypothesis with strong $l$--blur $(J, E)$. 
The idea is to `blow up and blur' $\R$ in place of the Maddux algebra 
$\mathfrak{E}_k(2, 3)$ dealt with in \cite[Lemma 5.1]{ANT}, where $k<\omega$ is the number of non--identity atoms 
and  $l$ depends recursively on $k$. 
Let $2<n\leq l<\omega$.  The relation algebra $\R$ is blown up by splitting all of the atoms each to infinitely many.
$\R$ is blurred by using a finite set of blurs (or colours) $J$. 
Then two partitions are defined on $\bf At$, call them $P_1$ and $P_2$.
Composition is re-defined on this new infinite atom structure; it is induced by the composition in $\R$, and a ternary relation $E$
on $\omega$, that `synchronizes' which three rectangles sitting on the $i,j,k$ $E$--related rows compose like the original algebra $\R$.
(This relation is definable in the first order structure $(\omega, <)$ \cite{ANT}).
The first partition $P_1$ is used to show that $\R$ embeds in the complex algebra of this new atom structure, namely, $\Cm \bf At$. 
The second partition $P_2$ divides $\bf At$ into {\it finitely many (infinite) rectangles}, each with base $W\in J$,
and the term algebra denoted in \cite{ANT} by  ${\Bb}(\R, J, E)$ over $\bf At$ (where $(J, E)$ is the strong $l$--blur for $\R$ 
assumed to exist by hypothesis)  consists of the sets that intersect co--finitely with every member of this partition.
One proves that ${\Bb}(\R, J, E)$ with atom structure $\bf At$ is representable using the finite number of blurs in $J$. 
Because $(J, E)$ is a strong $l$--blur, then, by definition, it is a strong $j$--blur for all $n\leq j\leq l$, so the atom structure $\bf At$ has a $j$--dimensional cylindric basis for all $n\leq j\leq l$, 
namely, ${\sf Mat}_j(\bf At)$.
For all such $j$, there is an $\RCA_j$ denoted on \cite[Top of p. 78]{ANT} by ${\Bb}_j(\R, J, E)$ such 
that $\Tm{\sf Mat}_j({\bf At})\subseteq {\Bb}_j(\R, J, E)\subseteq \Cm{\sf Mat}_j(\bf At)$ 
and $\At{\Bb}_j(\R, J, E)$ is a weakly representable atom structure of dimension $j$. 
Now take $\A={\Bb} _n(\R, J, E)$. We claim that $\A$ endowed with the identity operators as modalties induces by a the discrete toplogy on it base as required.  
Since $\R$ has a strong $j$--blur $(J, E)$ for all $n\leq j\leq l$, then 
$\A\cong \mathfrak{Nr}_n{\Bb}_j(\R, J, E)$, with ${\Bb}_j(\R, J, E)$ expanded to a ${\sf TCA}_j$ 
the same way for all $n\leq j\leq l$ as proved in \cite[item (3) p. 80]{ANT} for 'cylindric educts. Identity operators on both sids 
are obviously preserved.  In particular, 
taking $j=l$, $\A\in {\sf TRCA}_n\cap \Nr_n{\sf CA}_l$. We show that $\Rd_{ca}\Cm\At\A$  
does not have an $m$--flat representation.  
Assume for contradicton that $\Cm\At\A$ does have an $m$--flat representation $\Mo$.
Then $\Mo$ is infinite of course.  Since $\R$ embeds into ${\Bb}(\R, J, E)$ which in turn embeds into $\mathfrak{Ra}\Cm\At\A$, then 
$\R$ has an $m$--flat representation with base $\Mo$. But since $\R$ is finite, $\R=\R^+$, and consequently $\R$ has an  
infinite $m$--dimensional hyperbasis.  This is contrary to our assumption and we are done.

$(2)\implies (3)$: 
Fix $2<n<m<\omega$. 
As above, let $\L(\A)^m$ denote the signature that contains
an $n$--ary predicate symbol for every $a\in A$.
We that the existence of $m$--flat representations, implies the existence of $m$--dilations.
Let $\Mo$ be an $m$--flat representation of $\A$. We show that $\A\subseteq \mathfrak{Nr}_n\D$, for some $\D\in \CA_m$,
and that $\A$ actually has an infinitary $m$--flat representation.
For $\phi\in \L(\A)^m$,
let $\phi^{\Mo}=\{\bar{a}\in {\sf C}^m(\Mo):\Mo\models_c \phi(\bar{a})\}$, where ${\sf C}^m(\Mo)$ is the $n$--Gaifman hypergraph.
Let $\D$ be the algebra with universe $\{\phi^{\Mo}: \phi\in \L(A)^m\}$ and with  cylindric
operations induced by the $n$-clique--guarded (flat) semantics. 
For $r\in \A$, and $\bar{x}\in {\sf C}^m(\Mo)$, we identify $r$ with the formula it defines in $\L(A)^m$, and 
we write $r(\bar{x})^{\Mo}\iff \Mo, \bar{x}\models_c r$.
Then certainly $\D$ is a subalgebra of the ${\sf Crs}_m$ (the class
of algebras whose units are arbitrary sets of $m$--ary sequences)
with domain $\wp({\sf C}^m(\Mo))$, so $\D\in {\sf Crs_m}$ with unit $1^{\D}={\sf C}^m(\Mo)$.
Since $\Mo$ is $m$--flat, then cylindrifiers in $\D$ commute, and so $\D\in \CA_m$.
Now define $\theta:\A\to \D$, via $r\mapsto r(\bar{x})^{\Mo}$. Then exactly like in the proof of \cite[Theorem 13.20]{HHbook},
$\theta$ is a neat embedding, that is, $\theta(\A)\subseteq \Nrr_n\D$.
It is straightforward to check that $\theta$ is a homomorphism.  We show that $\theta$ is injective.
Let $r\in A$ be non--zero. Then $\Mo$ is a relativized representation, so there is $\bar{a}\in \Mo$
with $r(\bar{a})$, hence $\bar{a}$ is a clique in $\Mo$,
and so $\Mo\models r(\bar{x})(\bar{a})$, and $\bar{a}\in \theta(r)$, proving the required.
$\Mo$ itself might not be  infinitary $m$--flat, but one can build an infinitary $m$--flat representation of $\A$, whose base is an $\omega$--saturated model
of the consistent first order theory, stipulating the existence of an $m$--flat representation \cite[Proposition 13.17, Theorem 13.46 items (6) and (7)]{HHbook}.
This idea (of using saturation) will be given in more detail in the last item.

$(3)\implies (4)$: A complete $m$--flat representation of (any) 
$\B\in \CA_n$ induces an $m$--flat representation of $\Cm\At\B$ which implies by Theorem \ref{flat} that $\Cm\At\B\in \bold S\Nr_n\CA_m$. 
To see why, assume that $\B$ has an $m$--flat  complete representable via $f:\B\to \D$, where $\D=\wp(V)$ and 
the base of the representation $\Mo=\bigcup_{s\in V} \rng(s)$ is $m$--flat. Let $\C=\Cm\At\B$.
For $c\in C$, let $c\downarrow=\{a\in \At\C: a\leq c\}=\{a\in \At\B: a\leq c\}$; the last equality holds because 
$\At\B=\At\C$. Define, representing $\C$,  
$g:\C\to \D$ by $g(c)=\sum_{x\in c\downarrow}f(x).$ The map $g$ is well defined because $\C$ is complete so arbitrary suprema exist in $\C$. 
Furthermore, it can be easily checked that $g$ is a homomorphism into $\wp(V)$ having base $\Mo$ (basically because by assumption $f$ is a homomorphism).

$(4)\implies (5)$: if $\A\in \bold S_c\Nr_n\CA_m$ then it has an $m$ flat  
complete representation. Essentially a completeness theorem, this is a `truncated version' of Henkin's
neat embedding theorem: Existence of atomic $m$--dimensional dilations $\implies$  
existence of infinitay complete $m$--flat representations.
One constructs an infinitary $m$--flat representation $M$ of $\A$
Suppose that  $\A\subseteq_c \mathfrak{Nr}_m\D$, and $\D$ is atomic. We first show that $\D$ has an $m$--dimensional hyperbasis, lifted from relation algebra the obvious way.
First, it is not hard to see that for every $n\leq l\leq m$, $\mathfrak{Nr}_l\D$ is atomic.
The set of non--atomic labels $\Lambda$ is the set $\bigcup_{k<m-1}\At\mathfrak{Nr}_k\D$.
Before proceeding we need a piece of notation that is somewhat techical. . Let $m$ be a finite ordinal $>0$. An $\sf s$ word is a finite string of substitutions $({\sf s}_i^j)$ $(i, j<m)$,
a $\sf c$ word is a finite string of cylindrifications $({\sf c}_i), i<m$;
an $\sf sc$ word $w$, is a finite string of both, namely, of substitutions and cylindrifications.
An $\sf sc$ word
induces a partial map $\hat{w}:m\to m$:
\begin{itemize}

\item $\hat{\epsilon}=Id,$

\item $\widehat{w_j^i}=\hat{w}\circ [i|j],$

\item $\widehat{w{\sf c}_i}= \hat{w}\upharpoonright(m\smallsetminus \{i\}).$
\end{itemize}
If $\bar a\in {}^{<m-1}m$, we write ${\sf s}_{\bar a}$, or
${\sf s}_{a_0\ldots a_{k-1}}$, where $k=|\bar a|$,
for an  arbitrary chosen $\sf sc$ word $w$
such that $\hat{w}=\bar a.$
Such a $w$  exists by \cite[Definition~5.23 ~Lemma 13.29]{HHbook}.
Resuming main stream proof, For each atom $a$ of $\D$, define a labelled  hypergraph $N_a$ as follows:
Let $\bar{b}\in {}^{\leq m}m$. Then if $|\bar{b}|=n$,  so that $\bar{b}$  has to get a label that is an atom of $\D$, one sets  $N_a(\bar{b})$ to be 
the unique $r\in \At\D+$ such that $a\leq {\sf s}_{\bar{b}}r$; notation here
is given as above.
If $n\neq |\bar{b}| <m-1$, $N_a(\bar{b})$ is the unique atom $r\in \mathfrak{Nr}_{|b|}\D$ such that $a\leq {\sf s}_{\bar{b}}r.$ Since
$\mathfrak{Nr}_{|b|}\D$ is atomic, this is well defined. Note that this label may be a non--atomic one; it 
might not be an atom of $\D$. But by definition it is a permitted label.
Now fix $\lambda\in \Lambda$. The rest of the labelling is defined by $N_a(\bar{b})=\lambda$.
Then $N_a$ as an $m$--dimensional
hypernetwork, for each 
such chosen $a$,  and $\{N_a: a\in \At\D^+\}$ is the required $m$--dimensional hyperbasis.
The rest of the proof consists of a fairly straightforward adaptation of the proof \cite[Proposition 13.37]{HHbook},
replacing edges by $n$--hyperedges.

$(5)\implies (6)$: By \cite[\S 4.3]{HMT2}, we can (and will) assume that $\A= \Fm_T$ for a countable, atomic theory $L_n$ theory $T$.  
Let $\Gamma$ be the $n$--type consisting of co--atoms of $T$. Then $\Gamma$ is realizable in every $m$--flat model, for if $\Mo$ is an $m$--flat model omitting 
$\Gamma$, then $\Mo$ would be the base of a complete $m$--flat  representation of $\A$, and so $\A\in \bold S_c\Nr_n\CA_m$ which is impossible.
%But $\A\in {\sf Nr}_n\CA_l$, so  using exactly the same  (terminology and) argument in \cite[Theorem 3.1]{ANT} we get that  
%any witness isolating $\Gamma$  needs more 
%than $l$--variables (see also the proof of item (2) of corollary \ref{OTT2}.) 
Suppose for contradiction that $\phi$ is an $l$ witness, so that $T\models \phi\to \alpha$, for
all $\alpha\in \Gamma$, where recall that $\Gamma$ is the set of coatoms.
Then since $\A$ is simple, we can assume
without loss  that $\A$ is a set algebra with 
base $M$ say.
Let $\Mo=(M,R_i)_{i\in \omega}$  be the corresponding model (in a relational signature)
to this set algebra in the sense of \cite[\S 4.3]{HMT2}. Let $\phi^{\Mo}$ denote the set of all assignments satisfying $\phi$ in $\Mo$.
We have  $\Mo\models T$ and $\phi^{\Mo}\in \A$, because $\A\in \Nr_n\CA_{m-1}$.
But $T\models \exists x\phi$, hence $\phi^{\Mo}\neq 0,$
from which it follows that  $\phi^{\Mo}$ must intersect an atom $\alpha\in \A$ (recall that the latter is atomic).
Let $\psi$ be the formula, such that $\psi^{\Mo}=\alpha$. Then it cannot
be the case
that $T\models \phi\to \neg \psi$,
hence $\phi$ is not a  witness,
contradiction and we are done.

$(6)\implies (7)$: follows from the definitions.

For squareness the proofs are essentially the same undergoing the obvious modifications. In the first implication `infinite' in the hypothesis is not needed because any finite 
relation algebra having an infinite $m$--dimensional relational basis has a finite one, cf. \cite[Theorem 19.18]{HHbook}. 
On the other hand,  
there are finite relation algebras having infinite $m$--dimensional hyperbasis for $m\geq 5$ 
but has no finite ones \cite[Prop. 19.19]{HHbook}.
 
\end{proof}

We show that (4) and (5) are actually equivalent. 
If $\A$ has a complete infinitary $m$-flat representation $\implies \Cm\At\A\in \bold S_c\Nr_n\CA_m$.
One works in $L_{\infty, \omega}^m$ instead of first order logic.
Let $\Mo$ be the given representation of $\A$. In this case, the dilation $\D$ of $\A$ having again top element the Gaifman hypergraph ${\sf C}^m(\Mo)$, 
where $\Mo$ is the  complete infinitary $m$--flat representation of $\A$,
will now have  (the larger) universe $\{\phi^{\Mo}: \phi\in \L(A)^m_{\infty, \omega}\}$ with operations also induced by the $n$-clique--guarded semantics extended to
$L_{\infty, \omega}^m$.
Like before $\D$ will be a $\CA_m$, but this time, it will be {\it an atomic} one.  To prove atomicity, let $\phi^{\Mo}$ be a non--zero element in $\D$.
Choose $\bar{a}\in \phi^{\Mo}$, and consider the following infinitary conjunction (which we did not have before when working in $L_m$)
\footnote{There are set--theoretic subtleties involved here, that we prefer to ignore.}:
$\tau=\bigwedge \{\psi\in \L(A)_{\infty,\omega}^m: \Mo\models_C \psi(\bar{a})\}.$
Then $\tau\in \L(A)_{\infty,\omega}^m$, and $\tau^{\Mo}$ is an atom below $\phi^{\Mo}$. 
The neat embedding will be an atomic one, hence
it will be a complete neat embedding \cite[p. 411]{HHbook}.

\begin{corollary}\label{fl} For $2<n<\omega$ and $n\leq l<\omega$, $\Psi(n, t(n)$ and $\Psi(l, \omega)$ hold.
\end{corollary}
\begin{proof} The first case, follows from Theorem \ref{can} and  \ref{main} (by taking $l=n$ and $m=n+3)$. 
For the second case, 
it suffices by Theorem \ref{main} (by taking $m=\omega$) to find a countable algebra $\C\in \Nr_n\CA_l\cap \RCA_n$
such that $\Cm\At\C\notin \RCA_n$. This algebra is constructed in \cite{ANT}, cf. item (1) of Theorem \ref{main}.
\end{proof}

Reproving the main results in \cite{HH} and \cite{maddux} 
in a completely diferent way using Monk like algebras rather than rainbow ones, we get:
\begin{corollary}\label{HH} Let $2<n<\omega$
\begin{enumerate} 
\item The  set of equations using only one variable that holds in each of the varieties ${\sf RCA}_n$ and $\sf RRA$,  
together with any finite first order definable expansion of each, 
cannot be derived from any finite set of equations valid in 
the variety \cite{Biro, maddux}. Furthermore,  ${\sf LCA}_n$ 
is not finitely axiomatizable.
\item The classes ${\sf CRCA}_n$ and $\sf CRRA$ are not elementary.
\end{enumerate}
\end{corollary}
\begin{proof} 
1. Let ${\cal R}_l={{\mathfrak{Bb}}}(\R_l, J_l, E_l)\in \sf RRA$ where ${\cal R}_l$ is the relation algebra having atom structure denoted  $At$ in \cite[p. 73]{ANT} 
when the blown up and blurred algebra denoted $\R_l$ happens to be the finite Maddux algebra  $\mathfrak{E}_{f(l)}(2, 3)$
 and let  $\A_l=\mathfrak{Nr}_n{{\mathfrak{Bb}}}_l(\R_l, J_l, E_l)\in \RCA_n$ as defined in \cite[Top of p.80]{ANT} (with $\R_l=\mathfrak{E}_{f(l)}(2, 3)$).
Then $(\At{\cal R}_l: l\in \omega\sim n)$, and $(\At\A_l: l\in \omega\sim n)$ are sequences of weakly representable atom structures 
that are not strongly representable with a completely representable 
ultraproduct. 

2. The algebra $\B$ constructed in Theorem \ref{bsl} satisfies that $\B\in \Nr_n\CA_{\omega}\subseteq {\sf LCA}_n={\bf El}\CRCA_n$.
To see why,  $\B\in \bold S_c\Nr_n\CA_{\omega}$ is atomic, then by Lemma \ref{n}, \pe\ has a \ws\ in $\bold G^{\omega}(\At\B)$, hence in $G_{\omega}(\At\B)$, {\it a fortiori}, \pe\ has a \ws\ 
in $G_k(\At\B)$ for all $k<\omega$, so (by definition) $\B\in {\sf LCA}_n$.
We have already dealt with the $\CA$ case in Theorem \ref{bsl}, since the algebra $\B$ constructed therein satisfies that $\B\in {\bf El}\CRCA_n\sim \CRCA_n$. 
For relation algebras, we use the algebra $\A$ constructed in the previous Theorem, too.  We have  $\A\in \Ra\CA_{\omega}$ and $\A$ has no complete representation. The rest is like the $\CA$ case, 
using the $\Ra$ analogue of Lemma \ref{n} , when the dilation is $\omega$--dimensional, 
namely, $\A\in \bold S_c\Ra\CA_{\omega}\implies$ \pe\ has a \ws\ in $F^{\omega}$ with the last notation taken from \cite{r}. 
The last  argument proves that $\Ra\CA_{\omega}\subseteq {\bf El}\sf CRRA$.
\end{proof}

\begin{remark} Observe (from the proof of Proposition \ref{bsl})  that any atomic algebra in $\Nr_n\CA_{\omega}$ $(\Ra\CA_{\omega}$) with no complete representation, 
witnesses that $\CRCA_n$ $(\sf CRRA)$  is not elementary.  This cannot be witnessed on algebras having countably many atoms because restricting to such algebras we have 
$\bold S_c\Nr_n\CA_{\omega}\subseteq \sf CRCA_n$ ($\bold S_c\Ra\CA_{\omega}\subseteq \sf CRRA$.)
\end{remark}

Nevertheless, if we impose extra conditions on theories and possibly uncountably many non-principal types to be omitted, we get positive results in terms of $\omega$- square Tarskian 
usual semantics. Weprove a positive $\sf OTT$ for $L_n$ theories by imposing `elimination of quantifiers' on theoies and  maximality conditions on non-principal types we wish to omit,  such as being 
{\it complete}. If $T$ is a first order theory in a language $L$, then a set of formulas each using at most $m$ variables) $\Gamma$ say, is complete, if for any $L$-formula $\phi$
having at most $m$ variables, either $T\cup \Gamma\vdash \phi$ or $T\cup \Gamma \neg \phi$.
\begin{corollary} \label{eliminate} Let $n$ be any finite ordinal.  Let $T$ be a countable and consistent $L_n$ theory and $\lambda$ be a cardinal $< \mathfrak{p}$. 
Let $\bold F=(\Gamma_i: i<\lambda)$ be a family of non-principal types of $T$. Suppose that $T$ admits elimination of quantifiers. Then the following hold:
\begin{enumerate}
\item If $\phi$ is a formula consistent with $T$, then there 
is a model $\Mo$ of $T$ that omits $\bold F$, and $\phi$ is satisfiable in $\Mo$. 
If $T$ is complete, then we can replace $\mathfrak{p}$ by $\sf cov K$, 

\item If the non-principal types constituting $\bold F$ are maximal, then we can replace $\mathfrak{p}$ by $2^{\omega}$.
\end{enumerate}
\end{corollary}
\begin{proof} Let $T$ be as given in a signature $L$ having $n$ variables.  Let $\A=\Fm_T$, 
and $\bold G_i=\{\phi_T: \phi\in \Gamma_i\}$. Then $\bold G_i$ is a a non-principal ultrafilter; maximality follows from the completeness of 
types considered. By completeness of $T$, $\A$ is simple. Since $T$ admits elimination of quantifiers, then $\Fm_T\in \Nr_n\CA_{\omega}$. Indeed, let 
$T_{\omega}$ be the theory in the same signature $L$ but using $\omega$ many variables. Let $\C=\Fm_{T_{\omega}}$ be the Tarski-Lindenbaum quotient cylindric algebra algebra. 
Then $\C\in \CA_{\omega}$ (because we have $\omega$ many variables); in fact $\C\in \bold I\sf Cs_{\omega}$, and the map $\Phi$ defined from $\A$ to $\Nrr_n\C$ via $\phi/\equiv_T\mapsto \phi/\equiv_{T_{\omega}}$ is injective 
and bijective, that is to say, $\Phi$ having domain $\A$ and codomain $\Nrr_n\C$ is in fact {\it onto} $\Nrr_n\C$ due to quantifier 
elimination. An application of Theorem \ref{main} finishes the proof.
\end{proof}
Let $n<\omega$. By observing that if $T$ is a topological theory using $n$ variable admitting quantifier elimination, then $\Fm_T\in \Nr_n\TCA_{\omega}$, cf \cite[Theorem 3.2.10]{Sayed}, then we get:
\begin{corollary} Let $T$ be a topological countable predicate theory in $n$ variables, that admits elimination of quantifies. 
Then any family of $<2^{\omega}$ non-principal complete types can be omitted in a countable 
topological model.
\end{corollary}
\section{Non elementary classes of atom structues (Kripke frames)}
We start with an easy lemma to be used in the last item of the next theorem. If $\B$ is a Boolean algebra and $b\in \B$, 
then $\Rl_b\B$ denotes the Boolean algebra with domain $\{x\in B: x\leq b\}$, top element $b$, and other Boolean operations those of $\B$ relativized to $b$.
\begin{lemma}\label{join} 
In the following $\A$ and $\D$ are Boolean algebras.
\begin{enumerate}
  
\item  If $\A$  is atomic  and $0\neq a\in \A$, then $\Rl_a\A$ is also atomic. 
If $\A\subseteq_d \D$, and $a\in A$, then $\Rl_a\A\subseteq_d \Rl_a\D$,

\item  If $\A\subseteq_d \D$ then $\A\subseteq_c \D$. In particular, for any class $\sf K$ of $\sf BAOs$, ${\sf K}\subseteq \bold S_d{\sf K}\subseteq \bold  S_c{\sf K}$. 
If furthermore $\A$ and $\D$ are atomic,  then $\At\D\subseteq \At\A$. 
\end{enumerate}
\end{lemma}
\begin{proof}
(1): Entirely straightforward.
%Let $b\in \Rl_a\D$ be non--zero. Then $b\leq a$ and $b$ is non-zero in $\D$. By atomicity of $\D$ there is an atom $c$ of $\D$ such that $c\leq b$. 
%So $c\le b\leq a$, thus $c\in \Rl_a\D$. Also $c$ is an atom in $\Rl_a\D$ because if not, then it will not be an atom in $\D$.  
%The second part is similar.

(2): Assume that $\sum^{\A}S=1$ and for contradiction that there exists $b'\in \D$, $b'<1$ such that 
$s\leq b'$ for all $s\in S$. Let $b=1-b'$ then $b\neq 0$, hence by assumption (density) there exists a non-zero  
$a\in \A$ such that $a\leq b$, i.e $a\leq (1- b')$. If $a\cdot s\neq 0$ for some $s\in S$, then $a$ is not less than $b'$ which is impossible.
So $a\cdot s=0$  for every $s\in S$, implying that $a=0$, contradiction.
Now we prove the second part. Assume that $\A\subseteq_d \D$ and $\D$ is atomic. Let $b\in \D$ be an atom. We show that $b\in \At\A$.
By density there is a 
non--zero $a'\in \A$, such that $a'\leq b$ in $\D$. Since $\A$ is atomic, there is an atom $a\in \A$ such that $a\leq a'\leq b$. 
But $b$ is an atom of  $\D$,  and $a$ is non--zero in $\D$, too, so it must be the case that $a=b\in \At\A$.
Thus $\At\B\subseteq \At\A$ and we are done.
\end{proof}
 
Fix $2<n<\omega$. Call an atomic $\A\in \CA_n$ {\it weakly (strongly) representable} $\iff \At\A$ is weakly (strongly) representable.
Let  ${\sf WRCA}_n$ (${\sf SRCA}_n$) denote the class of all such $\CA_n$s, respectively. 
Then the class ${\sf SRCA}_n$  is not elementary and 
${\sf LCA}_n\subsetneq {\sf SRCA}_n\subsetneq {\sf WRCA}_n$ \cite{HHbook2}; the strictness of the two inclusions follow from the fact that the classes ${\sf LCA}_n$ and ${\sf WRCA}_n$ 
are elementary. 
For an atom structure $\bf At$, let $\F(\bf At)$ be the subalgebra of $\Cm\bf At$
consisting of all sets of atoms in $\bf At$ of the form $\{a\in {\bf At}: {\bf At}\models \phi(a, \bar{b})\}(\in \Cm {\bf At})$,
for some first order formula $\phi(x, \bar{y})$ of the signature
of $\bf At$  and some tuple $\bar{b}$ of atoms, cf.\cite[item (3), p. 456]{HHbook} for the analogous definition for 
relation algebras. 
Let ${\sf FCA}_n$  be the class of all such $\CA_n$s. Then it can be proved, similarly 
to the $\RA$ case that  
${\sf SRCA}_n\subseteq {\sf FCA}_n$ 
and that ${\sf FCA}_n$ is elementary, cf. \cite[Theorem 14.17]{HHbook}, hence the inclusion is strict.

In the following $\bf Up$, $\bf Ur$, $\bold P$ and  $\bold H$ denote the operations of 
forming ultraproducts, ultraroots, products 
and homomorphic images, respectively. 
\begin{theorem}\label{iii}
For $2<n<\omega$ the following hold:
\begin{enumerate}

\item For  $n<m\leq \omega$, $\Nr_n\CA_m$ is a psuedo elementay class that is not elementay;  
it is closed under $\bold P$, $\bold H$ but not under $\bold S_d$, {\t a fortiori $\bold S_c$ nor $\bold S$},  nor $\bf Ur$. 
The elementary theory of $\Nr_n\K_{\omega}$ is recursively enumerable.
\item For any class $\sf K$ of frames such that $\sf At\CRCA_n\subseteq \bold K\subseteq \At{\sf SRCA}_n$, $\bold K$ generates $\RCA_n$. 
If $\bold K$ is elementary then $\RCA_n$ is canonical, but $\bold S\Cm\At\A\nsubseteq \RCA_n$, since the last class is ot atom-canonical.
\item Although ${\sf Str}\RCA_n=\{\F\in \At\RCA_n: \Cm\F\in \RCA_n\}$ 
is not elementay and $\bold S\Cm\At\A\nsubseteq \RCA_n$, there is an elementary class of frames that generates $\RCA_n$, 
so the last class is canonical.
\item $\Nr_n\CA_{\omega}\subsetneq \bold S_d\Nr_n\CA_{\omega}\subseteq \bold S_c\Nr_n\CA_{\omega} \subsetneq {\bf El}\bold S_c\Nr_n\CA_{\omega} \subsetneq \RCA_n.$
Furthermore, the strictness of inclusions are witnessed by atomic algebras,

%\item ${\bf El}\bold L$ for any $\bold L$  of the classes  in the last item is an elementary subclass of $\RCA_n$ that is 
%not finitely axiomatizable.
\end{enumerate}
\end{theorem}
\begin{proof}
The case when $m$ is finite is eays we need only a tw sorted defining theory. To show that $\Nr_n\CA_{\omega}$  is pseudo-elementary, 
we use a three sorted defining theory, with one sort for a cylindric algebra of dimension $n$
$(c)$, the second sort for the Boolean reduct of a cylindric algebra $(b)$
and the third sort for a set of dimensions $(\delta)$; the argument is analogous to that of Hirsch used for relation algebra reducts \cite[Theorem 21]{r}.
We use superscripts $n,b,\delta$ for variables
and functions to indicate that the variable, or the returned value of the function,
is of the sort of the cylindric algebra of dimension $n$, the Boolean part of the cylindric algebra or the dimension set, respectively.
%We do it for $\CA$s.  The other cases can be dealt with in exactly the same way.
The signature includes dimension sort constants $i^{\delta}$ for each $i<\omega$ to represent the dimensions.
The defining theory for $\Nr_n{\sf CA}_{\omega}$ includes sentences stipulating
that the constants $i^{\delta}$ for $i<\omega$
are distinct and that the last two sorts define
a cylindric algebra of dimension $\omega$. For example the sentence
$$\forall x^{\delta}, y^{\delta}, z^{\delta}(d^b(x^{\delta}, y^{\delta})=c^b(z^{\delta}, d^b(x^{\delta}, z^{\delta}). d^{b}(z^{\delta}, y^{\delta})))$$
represents the cylindric algebra axiom ${\sf d}_{ij}={\sf c}_k({\sf d}_{ik}.{\sf d}_{kj})$ for all $i,j,k<\omega$.
We have have a function $I^b$ from sort $c$ to sort $b$ and sentences requiring that $I^b$ be injective and to respect the $n$ dimensional
cylindric operations as follows: for all $x^r$
$$I^b({\sf d}_{ij})=d^b(i^{\delta}, j^{\delta})$$
$$I^b({\sf c}_i x^r)= {\sf c}_i^b(I^b(x)).$$
Finally we require that $I^b$ maps onto the set of $n$ dimensional elements
$$\forall y^b((\forall z^{\delta}(z^{\delta}\neq 0^{\delta},\ldots (n-1)^{\delta}\rightarrow c^b(z^{\delta}, y^b)=y^b))\leftrightarrow \exists x^r(y^b=I^b(x^r))).$$
In all cases, it is clear that any algebra of the right type is the first sort of a model of this theory.
Conversely, a model for this theory will consist of an $n$ dimensional cylindric algebra type (sort c),
and a cylindric algebra whose dimension is the cardinality of
the $\delta$-sorted elements, which is at least $|m|$.
Thus the three sorted theory defines the class of neat reduct, furthermore, it is clearly recursive.
Finally, if $\K$ be a pseudo elementary class, that is
$\K=\{M^a|L: M\models U\}$ of $L$ structures, and $L, L^s, U$ are recursive.
Then there a set of first order recursive theory  $T$ in $L$,
so that for any $\A$ an $L$ structure, we have
$\A\models T$ iff there is a $\B\in \K$ with $\A\equiv \B$. In other words,
$T$ axiomatizes the closure of $\K$ under elementary equivalence, see
\cite[Theorem 9.37]{HHbook} for unexplained notation and proof.
Closure under $\bold P$ follows from that  
and that $\bold P\Nr_n\CA_{\omega}=\Nr_n\CA_{\omega}$. Let $\langle A_i: i\in I\rangle$ be  system of $\CA_m$s indexed by th non empty  set $I$. 
Then, in fact, $\bold P_{i\in I}\Nr_n\A_i=\Nr_n\bold P_{i\in i}\A_i$. 
Closure under $\bf Up$ 
follows from that $\Nr_n\CA_m$ for any $n<m\leq \omega$ 
is pseudo-elementary, cf. \cite[Theorem 21]{r} and \cite[\S 9.3]{HHbook} for similar cases. 
Hence it is clde under ultraproducts. Using the  system of agebras as above with $F$ a non principal ultrafilter on $I$ 
(more explicity),  we have $\bold P_{i\in I}\Nr_n\A_i/F\cong \Nr_n\bold P_{i\in I} A_i/F$
Closure of $\Nr_n\CA_m$ under $\bold H$ is proved in \cite{SL}. 
Th last is not closed under $\bold S_d$ by example \ref{SL}, 
and not under $\bf Ur$, because it is pseudo-elementary, but not elmentary.
Being closed under ultrapoducts, the Keisler-Shelah 
Ultrapower Theorem finishes the proof, showing that $\Nr_n\CA_m$ is not closed under ultraroots, that is to say, that here is an algebar $\A\notin \N_n\CA_m$ but $\A^I/F\in \Nr_n\CA_m$  

(2): Given an algebra $\A$ having $\CA_n$ signature, then $\A\in \RCA_n$ $\iff$ the canonical extension $\A^+$ of $\A$, based on the ultrafilter frame in symbols $\sf Uf\A$ 
or Stone spac of $\A$, whose underyling set consists of 
all   Boolean ultrafilters of $\A$,
namely, the complex algebra of $\sf Uf$, in symbols, $\Cm\sf Uf\A$,  is completely representable.
Therefore $\RCA_n$ is atomically generated by $\CRCA_n$ in the (strong sense sense), that is to say,  $\bold S\CRCA_n=\RCA_n$ (without the help of  $\bf HP$.)
For $t(n)=n(n+1)/2+1$, $\bold S\Nr_n\CA_m$ is not Sahlqvist axiomatizable by \cite{Venema}, because it is not atom-canonical {\it a fortiotri} 
not closed under \de\ completions. 

(3) The variety $\RCA_n$ is not atom-canonical $\iff$ $\bold S\Cm\At\RCA_n\nsubseteq \RCA_n$. 
Since by Theorem \ref{can}, there exist $n<m<\omega$, namely, $m=n(n+1)/2+1$ such 
that $\bold S\Nr_n\CA_m\supset \RCA_n$ and $\bold S\Nr_n\CA_m$ is not atom-canonical with respect
to $\RCA_n$, it follows that $\bold S\Cm\At\A\nsubseteq \RCA_m$, though $\bold S\Cm{\sf StrCA}_n\subseteq \RCA_n$.
 Althuogh ${\sf Str}\RCA_n$ is not elementay, then for any elementary class $\bold K$ such that ${\sf LCA}_n\subseteq \bold K\subseteq {\sf RCA}_n$, 
where recall that  ${\sf LCA}_n={\bf El}\CRCA_n$ is elementary 
by definition, consisting only of atomic algebras, we would have $\At\bold K$ elementary genetrating 
$\RCA_n$ in the srong sense, meaning that $\bold S\Cm\At\bold K=\RCA_n$ 
hence $\RCA_n$ is canonical.

(4):  The algebra $\B$ used 
in example \ref{SL} witnesses that $\Nr_n\CA_{\omega}\subsetneq \bold S_d\Nr_n\CA_{\omega}$, because, as proved in \cite{SL}, 
$\B\notin {\bf El}\Nr_n\CA_{\omega}(\supsetneq \Nr_n\CA_{\omega})$ and $\E\subseteq_ d \A$ where
$\A\in \Nr_n\CA_{\omega}$ is the full ${\sf Cs}_n$ with top element $^n\Q$ (and universe $\wp(^n\Q))$. Let $\A\in \CA_n$ be the algebra constructed in Theorem \ref{can}
We know that $\A\in {\sf RCA}_n\cap \bf At$,  but  
$\A\notin {\sf LCA}_n$, because $\At\A$ does not satisfy the Lyndon conditions, lest $\Cm\At\A\in {\sf LCA}_n(\subseteq \RCA_n)$.
We conclude that $\A\notin {\bf ElS}_c\Nr_n\CA_{\omega}$ proving the strictness of the last inclusion. 
Since $\E, \C$ and $\A$ are all atomic, we are done.
To show  that $\bold S_d\Nr_n\CA_{\omega}\subsetneq \bold S_c\Nr_n\CA_{\omega}$, 
we slighty modify the construction in \cite[Lemma 5.1.3, Theorem 5.1.4]{Sayedneat} as done below. 
The algebra denoted by $\B$ in {\it op.cit} 
witnesses the strictness of the inclusion.

\end{proof}

\begin{theorem}\label{finalresult} Let $2<n<\omega$. Then the following hold: 
\begin{enumerate}
\item $\bold S_c\Nr_n\CA_{\omega}\cap {\sf Count}={\sf CRCA}_n\cap {\sf Count}$,
\item ${\sf CRCA}_{n} \subseteq \bold S_c{\Nr}_{n}(\CA_{\omega}\cap {\bf At})\cap {\bf At}\subseteq \bold S_c{\Nr}_{n}\CA_{\omega}\cap \bf At.$
At least two of the above three classes are distinct but they coincide on  
algebras having countably many atoms. Non of all these classes 
is elementary. 
\item ${\bf El}(\bold S_c\Nr_n\CA_{\omega}\cap {\bf At})={\bf El}\bold S_c\Nr_n\CA_{\omega}\cap {\bf At}={\sf LCA}_n$, 
\item $\bold S\Nr_n\CA_{\omega}\cap {\bf At}= {\sf WRCA}_n$,
\item ${\bf P El S}_c\Nr_n\CA_{\omega}\cap {\bf At}\subseteq {\sf SRCA}_n$,
and ${\bf El P El S}_c\Nr_n\CA_{\omega}\cap {\bf At}\subseteq {\sf FCA}_n$.
\end{enumerate}
\end{theorem}
\begin{proof}The first required is aleady dealt with. For the second required, we know show that ${\sf CRCA}_{n}\subseteq \bold S_c{\Nr}_{n}(\CA_{\omega}\cap \bf At)\cap \bf At$. 
Let $\A\in {\sf CRCA}_n$. Assume that $\M$ is the base of a complete representation of $\A$, whose
unit is a generalized cartesian space,
that is, $1^{\M}=\bigcup {}^nU_i$, where $^{n}U_i\cap {}^{n}U_j=\emptyset$ for distinct $i$ and $j$, in some
index set $I$, that is, we have an isomorphism $t:\B\to \C$, where $\C\in {\sf Gs}_{n}$ 
has unit $1^{\M}$, and $t$ preserves arbitrary meets carrying
them to set--theoretic intersections.
For $i\in I$, let $E_i={}^{n}U_i$. Take  $f_i\in {}^{\omega}U_i$ 
and let $W_i=\{f\in  {}^{\omega}U_i^{(f_i)}: |\{k\in \omega: f(k)\neq f_i(k)\}|<\omega\}$.
Let ${\C}_i=\wp(W_i)$. Then $\C_i\in {\sf Ws}_{\omega}$ and is atomic; indeed the atoms are the singletons. 
 Let $x\in \mathfrak{Nr}_{n}\C_i$, that is ${\sf c}_ix=x$ for all $n\leq i<\omega$.
Now if  $f\in x$ and $g\in W_i$ satisfy $g(k)=f(k) $ for all $k<n$, then $g\in x$.
Hence $\mathfrak{Nr}_{n}\C_i$
is atomic;  its atoms are $\{g\in W_i:  \{g(i):i<n\}\subseteq U_i\}.$
Define $h_i: \A\to \mathfrak{Nr}_{n}\C_i$ by
$h_i(a)=\{f\in W_i: \exists a'\in \At\A, a'\leq a;  (f(i): i<n)\in t(a')\}.$
Let $\D=\bold P _i \C_i$. Let $\pi_i:\D\to \C_i$ be the $i$th projection map.
Now clearly  $\D$ is atomic, because it is a product of atomic algebras,
and its atoms are $(\pi_i(\beta): \beta\in \At(\C_i))$.  
Now  $\A$ embeds into $\mathfrak{Nr}_{n}\D$ via $J:a\mapsto (\pi_i(a) :i\in I)$. If $x\in \mathfrak{Nr}_{n}\D$,
then for each $i$, we have $\pi_i(x)\in \mathfrak{Nr}_{n}\C_i$, and if $x$
is non--zero, then $\pi_i(x)\neq 0$. By atomicity of $\C_i$, there is an $n$--ary tuple $y$, such that
$\{g\in W_i: g(k)=y_k\}\subseteq \pi_i(x)$. It follows that there is an atom
of $b\in \A$, such that  $x\cdot  J(b)\neq 0$, and so the embedding is atomic, hence complete.
We have shown that $\A\in \bold S_c{\sf Nr}_{n}\CA_{\omega}\cap \bf At$, and since $\A$ is atomic because $\A\in {\sf CRCA}_n$
we are done with the first inclusion.  The construction of an atomic $\A\in \Nr_n\CA_{\omega}$
(having uncountably many atoms) that lacks a complete representation in Theorem \ref{bsl} 
shows that the first and last classes are distinct. 
We show that ${\sf LCA}_n={\bf El}{\sf CRCA}_n={\bf El}(\bold S_c{\sf Nr}_n\CA_{\omega}\cap {\bf At}$). 
Assume that $\A\in {\sf LCA}_n$.
Then, by definition,  for all $k<\omega$, \pe\ has a \ws\ in $G_k(\At\A)$. Using ultrapowers followed by an elementary chain argument like in  \cite[Theorem 3.3.5]{HHbook2},   \pe\ has a \ws\ in
$G_{\omega}(\At\B)$ for some countable $\B\equiv \A$, and so by \cite[Theorem 3.3.3]{HHbook2} $\B$ is completely representable.
Thus $\A\in {\bf El}{\sf CRCA}_n$. 
If $\A\in \bold S_c\Nr_n\CA_{\omega}$ is atomic, then by Lemma \ref{n}, \pe\ has a \ws\ in $\bold G^{\omega}(\At\A)$, hence in $G_{\omega}(\At\A)$, {\it a fortiori}, \pe\ has a \ws\ 
in $G_k(\At\A)$ for all $k<\omega$, so (by definition) $\A\in {\sf LCA}_n$ so (since ${\sf LCA}_n$ is elementary) ${\bf El}(\bold S_c{\sf Nr}_n\CA_{\omega}\cap {\bf At})\subseteq {\sf LCA}_n$.
So ${\sf LCA}_n={\bf El}{\sf CRCA}_n\subseteq {\bf El}(\bold S_c{\sf Nr}_n\CA_{\omega}\cap {\bf At})\subseteq {\sf LCA}_{n}$, 
and we are done.  Item (3) follows by definition taking into account that  $\RCA_n=\bold S\Nr_n\CA_{\omega}$.

Assume that $\D\in \CRCA_n$ and   $\A\subseteq_c \D$. 
%We show that $\A\in \CRCA_n$, as well. 
Identifying set algebras with their domain let $f:\D\to \wp(V)$ be a complete representation of
$\D$, where $V$ is a ${\sf Gs}_n$ unit. 
We claim that $g=f\upharpoonright \A$ is a complete representation of $\A$. 
Let $X\subseteq \A$ be such that $\sum^{\A}X=1$. 
Then by $\A\subseteq_c \D$, we have  $\sum ^{\D}X=1$. Furthermore, for all $x\in X(\subseteq \A)$ we have $f(x)=g(x)$, so that 
$\bigcup_{x\in X}g(x)=\bigcup_{x\in X} f(x)=V$, since $f$ is a complete representation, 
and we are done. 
Let $\sf C$ be any of the two remaining classes. Closure under $\bold S_c$ follows from that $\bold S_c\bold S_c\sf C=\bold S_c\sf C$ 
Non-closure under $\bold S$ is trivial for a subalgebra of an atomic algebra may not be atomic.  We prove non--closure under $\bold H$ for all three classes in one blow. 
Take a family $(U_i: i\in \N)$ of pairwise disjoint non--empty sets. Let $V_i={}^nU_i (i\in \N)$. Take the full ${\sf Gs}_n$, $\A$ with universe $\wp(V)$, where $V=\bigcup_{i\in \N}{}V_i$. 
Then $\A\in \CRCA_n\subseteq \sf C$.  
Let 
$I$ be the ideal consisting of elements of $\A$ that intersect only finitely many of the $V_i$'s. 
Then $\A/I$ is not atomic, so $\A/I$ is outside all three classes.
Now we approach closure under ${\bf Ur}$. 
Let $\C\in {\sf CA}_n\sim {\sf CRCA}_n$ be atomic 
having countably many atoms 
and elementary equivalent to a $\B\in {\sf CRCA}_n$. Such algebras exist, cf. \cite{HH}, \cite[Theorem 5.12]{mlq}. 
Then $\C\equiv \B$, $\C$ will be outside all three classes (since they coincide on atomic algebras having countably many atoms),  
while $\B$ will be inside them all proving that non of the three is elementary, so being closed under $\bf Up$, since they are psuedo-elementary classes
(cf. \cite[Theorem 21]{r} and \cite[\S 9.3]{HHbook} for analogous cases), 
by the Keisler-Shelah ultrapower Theorem 
they are not closed under $\bf Ur$.  

Item (4) follows from that ${\sf LCA}_n\subseteq {\sf SRCA}_n$, that (it is easy to check that) 
${\sf SRCA}_n$ is closed under $\bold P$, that ${\sf SRCA}_n\subseteq {\sf FCA}_n$ and finally that the last class is elementary.
\end{proof}

For a variety $\sf V$ of $\sf BAO$s, $\sf Str(V)=\{\F: \Cm\F\in \V\}$. Fix finite $k>2$.  Then $\V_k={\sf Str}(\bold S\Nr_n\CA_{n+k})$ is not elementary 
$\implies   \V_k$ is not-atom canonical because (argueing contrapositively) in the case of atom--canonicity, 
we get that ${\sf Str}(\bold S\Nr_n\CA_{n+k})={\sf At}(\bold S\Nr_n\CA_{n+k})$, and the last class is elementary \cite[Theorem 2.84]{HHbook}. However, the converse implication may fail. 
In particular, we do not know whether ${\sf Str}(\bold S\Nr_n\CA_{n+k})$, for a particular finite $k\geq 3$, is elementary or not.
But we know that {\it there has to be a finite $k<\omega$ such that $\V_j$ is not elementary for all $j\geq k$}:
There is a finite $k$ such that the class of strongly representable algbras up to $n+k$
is not elementary. 
In the next few examples we check the boundaries of theorems \ref{rainbow1}, \ref{rainbow2} motivated by the fact that there could be 
an atomic algebra $\A$, such $\At\A\in \At\Nr_n\CA_{\omega}$ but $\A\notin \Nr_n\CA_{m}$ for all $m\leq \omega$. 
This can happen also for $\bold S\Nr_n\CA_{t(n)}$ (in place of $\Nr_n\CA_m$ with $m$ as above) 
where $t(n)=n(n+1)/2+1$ as shown in Theorem \ref{can}, but cannot happen for $\CRCA_n$. 
\begin{example}\label{SL}
Assume that $1<n<\omega$. Let $V={}^n\mathbb{Q}$ and let ${\A}\in {\sf Cs}_n$ have universe $\wp(V)$.
Then $\A\in {\sf Nr}_{n}\CA_{\omega}$.  
Let 
$y=\{s\in V: s_0+1=\sum_{i>0} s_i\}$ and ${\E}=\Sg^{\A}(\{y\}\cup X)$, where $X=\{\{s\}: s\in V\}$. 
Now $\E$ and $\A$ having same top element $V$, share the same atom structure, namely, the singletons, so 
$\Cm\At\E=\A$. Thus $\At\E\in \At\Nr_n\CA_{\omega}$ and $\A=\Cm\At\E\in \Nr_n\CA_{\omega}$. 
Since $\E\subseteq_d \A$, so $\E\in \bold S_d\Nr_n\CA_{\omega}\subseteq \bold S_c\Nr_n\CA_{\omega}$, 
but as proved in \cite{SL} $\E\notin {\bf El}\Nr_m\CA_{n+1}\subseteq \Nr_m\CA_{n+1}\supseteq \Nr_n\CA_{\omega}$. 
%Here $\bold S_d$ denotes the operation of forming dense subalgebras.
This can be generalized as follows:  
Let $\alpha$ be an ordinal $>1$; could be infinite. Let $\F$ is field of characteristic $0$.
$$V=\{s\in {}^{\alpha}\F: |\{i\in \alpha: s_i\neq 0\}|<\omega\},$$
$${\C}=(\wp(V),
\cup,\cap,\sim, \emptyset , V, {\sf c}_{i},{\sf d}_{i,j})_{i,j\in \alpha}.$$
Then clearly $\wp(V)\in \Nr_{\alpha}\sf CA_{\alpha+\omega}$.
Indeed let $W={}^{\alpha+\omega}\F^{(0)}$. Then
$\psi: \wp(V)\to \Nr_{\alpha}\wp(W)$ defined via
$$X\mapsto \{s\in W: s\upharpoonright \alpha\in X\}$$
is an isomorphism from $\wp(V)$ to $\Nr_{\alpha}\wp(W)$.
We shall construct an algebra $\B$, $\B\notin \Nr_{\alpha}{\sf CA}_{\alpha+1}$.
Let $y$ denote the following $\alpha$-ary relation:
$$y=\{s\in V: s_0+1=\sum_{i>0} s_i\}.$$
Let $y_s$ be the singleton containing $s$, i.e. $y_s=\{s\}.$
Define as before
${\B}\in {\CA }_{\alpha}$
as follows:
$${\B}=\Sg^{\C}\{y,y_s:s\in y\}.$$
The first order sentence that codes the idea of the proof says
that $\B$ is neither an elementary nor complete subalgebra of $\wp(V)$.
Let $\At(x)$ be the first order formula asserting that $x$ is an atom.
Let $$\tau(x,y) ={\sf c}_1({\sf c}_0x\cdot {\sf s}_1^0{\sf c}_1y)\cdot {\sf c}_1x\cdot {\sf c}_0y.$$
Let $${\sf Rc}(x):=c_0x\cap c_1x=x,$$
$$\phi:=\forall x(x\neq 0\to \exists y(\At(y)\land y\leq x))\land
\forall x(\At(x) \to {\sf Rc}(x)),$$
$$\alpha(x,y):=\At(x)\land x\leq y,$$
and  $\psi (y_0,y_1)$ be the following first order formula
$$\forall z(\forall x(\alpha(x,y_0)\to x\leq z)\to y_0\leq z)\land
\forall x(\At(x)\to \At(\sf c_0x\cap y_0)\land \At(\sf c_1x\cap y_0))$$
$$\to [\forall x_1\forall x_2(\alpha(x_1,y_0)\land \alpha(x_2,y_0)\to \tau(x_1,x_2)\leq y_1)$$
$$\land \forall z(\forall x_1 \forall x_2(\alpha(x_1,y_0)\land \alpha(x_2,y_0)\to
\tau(x_1,x_2)\leq z)\to y_1\leq z)].$$
Then
$$\Nr_{\alpha}{\sf CA}_{\beta}\models \phi\to \forall y_0 \exists y_1 \psi(y_0,y_1).$$
But this formula does not hold in $\B$.
We have $\B\models \phi\text {  and not }
\B\models \forall y_0\exists y_1\psi (y_0,y_1).$
In words: we have a set $X=\{y_s: s\in V\}$ of atoms such that $\sum^{\A}X=y,$ and $\B$
models $\phi$ in the sense that below any non zero element there is a
{\it rectangular} atom, namely a singleton.
Let $Y=\{\tau(y_r,y_s), r,s\in V\}$, then
$Y\subseteq \B$, but it has {\it no supremum} in $\B$, but {\it it does have one} in any full neat reduct $\A$ containing $\B$,
and this is $\tau_{\alpha}^{\A}(y,y)$, where
$\tau_{\alpha}(x,y) = {\sf c}_{\alpha}({\sf s}_{\alpha}^1{\sf c}_{\alpha}x\cdot {\sf s}_{\alpha}^0{\sf c}_{\alpha}y).$
In $\wp(V)$ this last is $w=\{s\in {}^{\alpha}\F^{(\bold 0)}: s_0+2=s_1+2\sum_{i>1}s_i\},$
and $w\notin \B$. .
For $y_0=y$, there is no $y_1\in \B$ satisfying $\psi(y_0,y_1)$.
\end{example}
If $\bold K$ is  class of $\sf BAO$s, and $\A\in \bold K$ is atomic, then plainly $\At\A\in \bold \At\bold K$. But the converse might not be true.
But, conversely, if $\A\in {\sf CRCA}_n$ 
has atom structure $\bf At$, then if $\A\in \CA_n$ and $\At\A=\bf At$, then $\A\in {\sf CRCA}_n$. This motivates:
\begin{definition}\label{grip} 
\begin{enumarab}
\item The  class $\bold K$ is {\it gripped by its atom structures,} or simply {\it gripped,}
if for $\A\in \CA_n$, whenever $\At\A\in \At \bold K$, then $\A\in \bold K$.
\item An $\omega$--rounded game $\bold H$ {\it grips} $\bold K$, if whenever  $\A\in\CA_n$ is atomic with countably many atoms and \pe\ has a \ws\ in $\bold H(\A)$,
then $\A\in \bold K$. The game $\bold H$ {\it weakly grips} $\bold K$, if whenever  $\A\in\CA_n$ is atomic with countably many atoms and
\pe\ has a \ws\ in $\bold H(\At\A)$, then $\At\A\in \At\bold K$. The game $\bold H$ {\it densely grips} $\bold K$,  if whenever  $\A\in\CA_n$ is atomic with countably many atoms and
\pe\ has a \ws\ in $\bold H(\At\A)$, then $\At\A\in \At\bold K$ and $\Cm\At\A\in \bold K$. 
\end{enumarab}
\end{definition}

\begin{example} 
Let $2<n<m\leq \omega$. 
\begin{itemize}
\item The classes $\RCA_n$ and $\Nr_n\CA_m$ are not gripped, by \cite{mlq, Hodkinson} and example \ref{SL} 
In fact, by \cite{mlq}, if $m\geq n+3$,  then $\bold S\Nr_n\CA_m$ is not gripped.
For any $n<\omega$, the class ${\sf CRCA}_n$, and its elementary closure, namely,
${\sf LCA}_n$ (the class of algebras satisfying the Lyndon conditions) 
are gripped. 

\item The class $\bold S_c\Nr_n\CA_m$ is gripped.

\item The usual atomic game $G$ weakly grips, densely grips and 
grips ${\sf CRCA}_n$. 
\end{itemize}
\end{example}
Now we define the game $\bold H$, mentioned above outline of Theorem \ref{rainbow}, that  
densely grips, hence weakly grips, 
$\Nr_n\CA_{\omega}$, but we it can be shown that $\bold H$ 
does not grip $\Nr_n\CA_{\omega}$. Accordingly $\bold H$ will be used to prove a slightly weaker result 
as stated in the idea of proof.  

%The game $\bold H$ is similar to the game devised 
The following definition, inspired by the result in Theorem \ref{can}, stresses the fact that some algebras are 'more representable' than others.

\begin{definition}

Let $m\leq\omega$ and $\A\in \RCA_n$ be atomic:. 

1. Then  $\A$ is  {\it strongly representable} up to $m$ if  if $\Cm\At\A\in \bold S\Nr_n\CA_{n+m}$.

2. We say that $\A$ is {\it extremely} representable up to $m$ if   if $\Cm\At\A\in \bold S_c\Nr_n\CA_{n+m}$.

\end{definition}

\begin{example} 
The atomic algebra $\C=\C_{\Z, \N}$ used in \cite{mlq}, to be recalled below is $\Cm({\sf At}\C)\in \RCA_n$, hence $\C$ strongly representable up to $m$ for any $m\leq \omega$, 
but $\C$  is not extremely represetable up to $\omega$ since 
it lacks a complete representation to be proved in a while
 upon observing that it has only countably many atoms. 
Worthy of note is that in fact $\C\in {\sf LCA}_n(\subseteq {\sf SRCA}_n$).    
However,  if $\A\in \RCA_n$ is finite and $m\leq \omega$, 
then $\A$ is strongly representable up to $m$ $\iff$ $\A$ is  extremely representable up to $m$.
\end{example}

\begin{theorem} An atomic algebra $\A\in \RCA_n$  is strongly representable up to $\omega\iff \A\in {\sf SRCA}_n$. If $\A$ is atomic with countably many atoms, then 
$\A$ extremely representable up to $\omega$ $\iff$ $\A\in \CRCA_n$. 
\end{theorem}
\begin{proof} The first part follows by observing that $\RCA_n=\bold S\Nr_n\CA_{\omega}$. 
For the second part. Assume that $\A$ has countably many atoms and $\A$ is extremely representabe up to $\omega$. 
Then $\Cm\At\A\in \bold S_c\Nr_n\CA_{\omega}$. Since $\A\subseteq_d \Cm\At\A$, 
then $\A\subseteq_c \Cm\At\A$, and since the operator $\bold S_c$ is obviously idempotent ($\bold S_c\bold S_c=\bold S_c$), we get that
$\A\in \bold S_c\Nr_n\CA_{\omega}$; having countably many atoms we get $\A\in \CRCA_n$. Conversely, if $\A\in \CRCA_n$ is atomic 
having countably many atoms, then $\Cm\At\A\in \CRCA_n$, 
$\A\in \bold S_c\Nr_n\CA_{\omega}$. 
\end{proof}
\begin{lemma}\label{gripneat2} Let $\alpha$ be a countable atom structure. If \pe\ has a \ws\ in $\bold H_{\omega}(\alpha)$, 
then there exists a complete $\D\in \RCA_{\omega}$ such that 
$\alpha\cong \At\Nrr_n\D$. 
\end{lemma}
\begin{proof} We show that $\alpha\cong \At\mathfrak{Nr}_n\D$ with $\D$ as in Theorem \ref{gripneat}
%The argument used is like the argument used in \cite[Theorem 39]{r} adapted to $\CA$s.
Let $x\in \D$. Then $x=(x_a:a\in\alpha)$, where $x_a\in\D_a$.  For $b\in\alpha$ let
$\pi_b:\D\to \D_b$ be the projection map defined by
$\pi_b(x_a:a\in\alpha) = x_b$.  Conversely, let $\iota_a:\D_a\to \D$
be the embedding defined by $\iota_a(y)=(x_b:b\in\alpha)$, where
$x_a=y$ and $x_b=0$ for $b\neq a$.  
Suppose $x\in\Nrr_n\D\setminus\set0$.  Since $x\neq 0$,
then it has a non-zero component  $\pi_a(x)\in\D_a$, for some $a\in \alpha$.
Assume that $\emptyset\neq\phi(x_{i_0}, \ldots, x_{i_{k-1}})^{\D_a}= \pi_a(x)$, for some $L$-formula $\phi(x_{i_0},\ldots, x_{i_{k-1}})$.  We
have $\phi(x_{i_0},\ldots, x_{i_{k-1}})^{\D_a}\in\Nrr_{n}\D_a$.
Pick
$f\in \phi(x_{i_0},\ldots, x_{i_{k-1}})^{\D_a}$  
and assume that ${\cal M}_a, f\models b(x_0,\ldots x_{n-1})$ for some $b\in \alpha$.
We show that
$b(x_0, x_1, \ldots, x_{n-1})^{\D_a}\subseteq
 \phi(x_{i_0},\ldots, x_{i_{k-1}})^{\D_a}$.  
Take any $g\in
b(x_0, x_1\ldots, x_{n-1})^{\D_a}$,
so that ${\cal M}_a, g\models b(x_0, \ldots x_{n-1})$.  
The map $\{(f(i), g(i)): i<n\}$
is a partial isomorphism of ${\cal M}_a.$ Here that short hyperedges are constantly labelled by $\lambda$ 
is used.
This map extends to a finite partial isomorphism
$\theta$ of $M_a$ whose domain includes $f(i_0), \ldots, f(i_{k-1})$.
Let $g'\in {\cal M}_a$ be defined by
\[ g'(i) =\left\{\begin{array}{ll}\theta(i)&\mbox{if }i\in\dom(\theta)\\
g(i)&\mbox{otherwise}\end{array}\right.\] 
We have ${\cal M}_a,
g'\models\phi(x_{i_0}, \ldots, x_{i_{k-1}})$. But 
$g'(0)=\theta(0)=g(0)$ and similarly $g'(n-1)=g(n-1)$, so $g$ is identical
to $g'$ over $n$ and it differs from $g'$ on only a finite
set.  Since $\phi(x_{i_0}, \ldots, x_{i_{k-1}})^{\D_a}\in\Nrr_{n}\D_a$, we get that
${\cal M}_a, g \models \phi(x_{i_0}, \ldots,
x_{i_k})$, so $g\in\phi(x_{i_0}, \ldots, x_{i_{k-1}})^{\D_a}$ (this can be proved by induction on quantifier depth of formulas).  
This
proves that 
$$b(x_0, x_1\ldots x_{n-1})^{\D_a}\subseteq\phi(x_{i_0},\ldots,
x_{i_k})^{\D_a}=\pi_a(x),$$ and so
$$\iota_a(b(x_0, x_1,\ldots x_{n-1})^{\D_a})\leq
\iota_a(\phi(x_{i_0},\ldots, x_{i_{k-1}})^{\D_a})\leq x\in\D_a\setminus\set0.$$
Now every non--zero element 
$x$ of $\Nrr_{n}\D_a$ is above a non--zero element of the following form 
$\iota_a(b(x_0, x_1,\ldots, x_{n-1})^{\D_a})$
(some $a, b\in \alpha$) and these are the atoms of $\Nrr_{n}\D_a$.  
The map defined  via $b \mapsto (b(x_0, x_1,\dots, x_{n-1})^{\D_a}:a\in \alpha)$ 
is an isomorphism of atom structures, 
so that $\alpha\in \At{\sf Nr}_n\CA_{\omega}$.  
\end{proof}
Having Lemma \ref{gripneat2} at our hand,  if we go down to the level of atom structures (of atomic algebras), 
then we can remove the $\bold S_d$ appearing in Theorem \ref{rainbow1}. Going from atomic algebras to atom strctures this is feasable, 
but the process is not reversile as shown in 
example \ref{SL}. More succinctly the implicaion Corolary \ref{rainbow2}$\implies$ Theorem \ref{rainbow1} is not valid, 
because precisely $\Nr_n\CA_m$ for any $2<n<m\leq m$ is not 
gripped.
\begin{corollary}\label{rainbow2} Any class of frames $\bold K$, such $\At(\Nr_n\CA_{\omega}\cap \CRCA_n)\subseteq \bold K\subseteq \At\bold S_c\Nr_n\CA_{n+3}$  is not elementary
\end{corollary} 
\begin{proof} 
Using the technique in Theorem \ref{rainbow1} together with Lemmata \ref{gripneat} and \ref{gripneat2}.
 \end{proof}
\begin{theorem} Any class $\bold K$ such that $\Nr_n\CA_{\omega}\subseteq \bold K\subseteq \bold S_d\Nr_n\CA_{n+1}$, $\bold K$ is not elementary.
\end{theorem}
\begin{proof} We slighty modify the construction in \cite[Lemma 5.1.3, Theorem 5.1.4]{Sayedneat}. Using the same notation, the algebras $\A$ and $\B$ constructed in {\it op.cit} satisfy 
$\A\in {\sf Nr}_n\CA_{\omega}$, $\B\notin {\sf Nr}_n\CA_{n+1}$ and $\A\equiv \B$.
As they stand, $\A$ and $\B$ are not atomic, but it 
can be  fixed that they are atomic, giving the same result with the rest of the proof unaltered. This is done by interpreting the uncountably many tenary relations in the signature of 
$\Mo$ defined in \cite[Lemma 5.1.3]{Sayedneat}, which is the base of $\A$ and $\B$ 
to be {\it disjoint} in $\Mo$, not just distinct.  The construction is presented this way in \cite{IGPL}, where (the equivalent of) 
$\Mo$ is built in a 
more basic step-by-step fashon.
We work with $2<n<\omega$ instead of only $n=3$. The proof presented in {\it op.cit} lift verbatim to any such $n$.
Let $u\in {}^nn$. Write $\bold 1_u$ for $\chi_u^{\Mo}$ (denoted by $1_u$ (for $n=3$) in \cite[Theorem 5.1.4]{Sayedneat}.) 
We denote by $\A_u$ the Boolean algebra $\Rl_{\bold 1_u}\A=\{x\in \A: x\leq \bold 1_u\}$ 
and similarly  for $\B$, writing $\B_u$ short hand  for the Boolean algebra $\Rl_{\bold 1_u}\B=\{x\in \B: x\leq \bold 1_u\}.$
Then exactly like in \cite{Sayedneat}, it can be proved that $\A\equiv \B$.
Using that $\Mo$ has quantifier elimination we get, using the same argument in {\it op.cit} 
that $\A\in \Nr_n\CA_{\omega}$.  The property that $\B\notin \Nr_n\CA_{n+1}$ is also still maintained.
To see why, consider the substitution operator $_{n}{\sf s}(0, 1)$ (using one spare dimension) as defined in the proof of \cite[Theorem 5.1.4]{Sayedneat}.
Assume for contradiction that 
$\B=\Nr_{n}\C$, with $\C\in \CA_{n+1}.$ Let $u=(1, 0, 2,\ldots n-1)$. Then $\A_u=\B_u$
and so $|\B_u|>\omega$. The term  $_{n}{\sf s}(0, 1)$ acts like a substitution operator corresponding
to the transposition $[0, 1]$; it `swaps' the first two co--ordinates.
Now one can show that $_{n}{\sf s(0,1)}^{\C}\B_u\subseteq \B_{[0,1]\circ u}=\B_{Id},$ 
so $|_{n}{\sf s}(0,1)^{\C}\B_u|$ is countable because $\B_{Id}$ was forced by construction to be 
countable. But $_{n}{\sf s}(0,1)$ is a Boolean automorpism with inverse
$_{n}{\sf s}(1,0)$, 
so that $|\B_u|=|_{n}{\sf s(0,1)}^{\C}\B_u|>\omega$, contradiction.
Since $\A\equiv \B$ and $\A\in \Nr_n\CA_{\omega}\cap {\sf CRCA}_n$, it suffices to show  
(since $\B$ is atomic) that 
$\B$  is in fact outside $\bold S_d\Nr_n\CA_{n+1}\cap \bf At$.  
Take $\kappa$ the signature of $\Mo$; more specifically,  the number of $n$-ary relation symbols to be $2^{2^{\omega}}$, and assume for contradiction that  
$\B\in \bold S_d\Nr_n\CA_{n+1}\cap \bf At$. 
Then $\B\subseteq_d \mathfrak{Nr}_n\D$, for some $\D\in \CA_{n+1}$ and $\mathfrak{Nr}_n\D$ is atomic. For brevity, 
let $\C=\mathfrak{Nr}_n\D$. Then by item (1) of Lemma \ref{join} $\Rl_{Id}\B\subseteq_d \Rl_{Id}\C$.
Since $\C$ is atomic,  then by item (1) of the same Lemma $\Rl_{Id}\C$ is also atomic.  Using the same reasoning as above, we get that $|\Rl_{Id}\C|>2^{\omega}$ (since $\C\in \Nr_n\CA_{n+1}$.) 
By the choice of $\kappa$, we get that $|\At\Rl_{Id}\C|>\omega$. 
By density, we get from item (2) of Lemma \ref{join}, that $\At\Rl_{Id}\C\subseteq \At\Rl_{Id}\B$. 
Hence $|\At\Rl_{Id}\B|\geq |\At\Rl_{Id}\C|>\omega$.   
But by the construction of $\B$, $|\Rl_{Id}\B|=|\At\Rl_{Id}\B|=\omega$,   which is a  contradiction and we are done.
But $\B$ is completely reprtesentable, thus $\B\in \bold S_c\Nr_n\CA_{\omega}$.
Thus $\B\in {\bf El}(\Nr_n\CA_{\omega}\cap {{\sf CRCA}_n})\sim \bold S_d\Nr_n\CA_{\omega}$.
\end{proof}
It can be shown that neither of the two classes $\bold S_d\Nr_n\CA_{n+1}$ and $\bold S_c\Nr_n\CA_{n+3}$ 
are included in the other. The algebra $\B$ in example \ref{SL} is in the last, because it is completely representable, but not the first, 
while in \cite{t},  it is shown that for any $k\geq 1$, there is a finite $\D_k\in \CA_n$ such that $\D_k\in \bold S\Nr_n\CA_{n+k+1}\sim \Nr_n\CA_{n+k}$.
Since for any such algebra $\D_k=\D_k^+$, and for any $\E\in \CA_n$ and $m>n$, $\E\in \bold S\Nr_n\CA_{m}\iff \E\in \bold S_c\Nr_n\CA_m$, 
then $\D_k\in \bold S_c\Nr_n\CA_{n+3}\sim \bold S_d\Nr_n\CA_{n+1}$.
From this we immediately get that  $\D_2\in\bold S_c\Nr_n\CA_{n+3}\sim \bold S_d\Nr_n\CA_{n+1}$. 
Excluding elementary classes between $\Nr_n\CA_{\omega}$ and $\bold S_d\Nr_n\CA_{\omega}$ and excluding ones 
between $\bold S_d\Nr_n\CA_{\omega}$ and 
$\bold S_c\Nr_n\CA_{n+3}$, this prompts the following: 
Is there an elementary class $\sf K$ between $\Nr_n\CA_{\omega}$ and $\bold S_c\Nr_n\CA_{n+3}$ ?
Observe that if ${\sf K}\nsubseteq \bold S_d\Nr_n\CA_{\omega}=\emptyset$, then ${\sf K}\cap \sim \bold S_d\Nr_n\CA_{\omega}\neq \emptyset$. 

\begin{athm}{Conjecture}\label{rainbow4} 
For $2<n<\omega$, any class $\bold K$ such that  $\Nr_n\CA_{\omega}\cap \CRCA_n\subseteq \bold K\subseteq \bold S_c\Nr_n\CA_{n+3}$, 
$\bold K$   is not elementary.  
%Furthermore any class $\sf L$ such tht $\At\Nr_n\CA_{\omega}\subseteq \sf K\subseteq \At\bold S_c\Nr_n\CA_{n+3}$ i not elementary
\end{athm}

{\bf Summary of results on (non-) first order definability of classes of algebras 
involving the operators $\Ra$ and $\Nr_n$ ($2<n<\omega)$:}
In the next table we summarize the results obtained on non first order definability proved in theorems \ref{rainbow}, \ref{ra}, \ref{raa}. 
The last column in the second row remains unsettled for $\RA$s.
      
\vskip3mm
\begin{tabular}{|l|c|c|c|c|c|c|}    \hline
					Cylindric algebras                      &Elementary                                      \\

                                                               \hline
                                                                          $\Nr_n\CA_{\omega}\subseteq \bold K\subseteq \bold S_c\Nr_n\CA_{n+3}$&?\\

                                                                           \hline
                                                                          $\At\bold \Nr_n\CA_{\omega}\subseteq \bold K\subseteq \bold \At \bold S_c\Nr_n\CA_{n+3}$  &no\\
                                                            
                                                                           \hline
                                                                          $\bold S_d\Nr_n\CA_{\omega}\subseteq \bold K\subseteq \bold S_c\Nr_n\CA_{n+3}$  &no\\

                                                                           \hline
                                                                          $\bold S_c\Nr_n\CA_{\omega}\subseteq \bold K\subseteq \bold S_c\Nr_n\CA_{n+3}$&no\\                               
\hline $\At\bold S_d\Nr_n\CA_{\omega}\subseteq \bold K\subseteq \At\bold S_d\Nr_n\CA_{n+3}$  
no\\

                                                                           \hline
                                                                          $\Nr_n\CA_{\omega}\subseteq \bold K\subseteq \bold \Nr_n\CA_{n+1}$   &no\\

\hline

\end{tabular}

\section{Representability Theory}

\subsection{Notions of representability and neat embeddings}

%Fix $2<n<\omega$. Call an atomic $\A\in \CA_n$ {\it weakly (strongly) representable} $\iff \At\A$ is weakly (strongly) representable.
%Let  ${\sf WRCA}_n$ (${\sf SRCA}_n$) denote the class of all such $\CA_n$s, respectively. 
%Then the class ${\sf SRCA}_n$  is not elementary and 
%${\sf LCA}_n\subsetneq {\sf SRCA}_n\subsetneq {\sf WRCA}_n$ \cite{HHbook2}; the strictness of the two inclusions follow from the fact that the classes ${\sf LCA}_n$ and ${\sf WRCA}_n$ 
%are elementary. 
%For an atom structure $\bf At$, let $\F(\bf At)$ be the subalgebra of $\Cm\bf At$
%consisting of all sets of atoms in $\bf At$ of the form $\{a\in {\bf At}: {\bf At}\models \phi(a, \bar{b})\}(\in \Cm {\bf At})$,
%for some first order formula $\phi(x, \bar{y})$ of the signature
%of $\bf At$  and some tuple $\bar{b}$ of atoms, cf.\cite[item (3), p. 456]{HHbook} for the analogous definition for 
%relation algebras. 
%Let ${\sf FCA}_n$  be the class of all such $\CA_n$s. Then it can be proved, similarly 
%to the $\RA$ case that  
%${\sf SRCA}_n\subseteq {\sf FCA}_n$ 
%and that ${\sf FCA}_n$ is elementary, cf. \cite[Theorem 14.17]{HHbook}, hence the inclusion is strict.

\begin{corollary}\label{iiii} Let $2<n<\omega$. Then the following hold:
\begin{enumerate}
\item $\Nr_n\CA_{\omega}\cap {\bf At}\subsetneq \bold S_d\Nr_n\CA_{\omega}\cap {\bf At}\subsetneq \bold S_c\Nr_n\CA_{\omega}\cap {\bf At} \subsetneq {\bf El}\bold S_c\Nr_n\CA_{\omega}\cap {\bf At}=
{\bf El}(\bold S_c\Nr_n\CA_{\omega}\cap 
{\bf At})={\sf LCA}_n\subsetneq {\sf SRCA}_n\subsetneq {\bf El}{\sf SRCA}_n\subseteq {\sf FCA}_n\subsetneq {\sf WRCA}_n,$ 
 \item  ${\sf CRCA}_n\subsetneq \bold S_c\Nr_n\CA_{\omega}\cap {\bf At} \subsetneq {\bf El}{\sf CRCA}_n={\bf El}\bold S_c\Nr_n\CA_{\omega}\cap 
{\bf At}={\sf LCA}_n$,
\item For any class ${\sf K}(\subseteq \RCA_n\cap \bf At)$ occurring in the previous two items, ${\bf El}\sf K$ is not finitely axiomatizable,
and $\bold S{\sf K}=\RCA_n$.
\end{enumerate}
\end{corollary}
\begin{proof} The algebra $\B$ used in the last item of Theorem \ref{iiii} is in 
$\bold S_c\Nr_n\CA_{\omega}\cap {\bf At}\sim \bold S_d\Nr_n\CA_{\omega}$. 
For the strictness of the last 
inclusion in the first item, we refer the reader to \cite[Theorem 14.17]{HHbook} for the relation algebra analogue. Item (2) is already dealt with.
Now we approach item (3): From \cite[Construction 3.2.76, p.94]{HMT2}, it can be easily distilled that 
the elementary closure of any class $\bold K$, such that $\Nr_n\CA_{\omega}\cap {\bf At}\subseteq \bold K\subseteq {\sf RCA}_n$, $\bold K$ is not finitely axiomatizable. 
In the aforementioned construction, non--representable finite (Monk) algebras outside ${\sf RCA}_n$
are constructed, such that any (atomic) non--trivial ultraproduct of such algebras
is in  ${\sf Nr}_n\CA_{\omega}\cap {\bf At}$.  
For proving non--finite axiomatizability one uses \cite[Construction 3.2.76, pp.94]{HMT2}.
In {\it op.cit} non--representable finite Monk algebras outside ${\sf RCA}_n\supseteq {\bf El}{\sf SRCA}_{n}\supseteq {\sf LCA}_n$
are constructed, such that any (atomic) non--trivial ultraproduct of such algebras
is in   ${\sf Nr}_n\CA_{\omega}\cap {\bf At}\subseteq {\bf El}{\sf Nr}_n\CA_{\omega}\cap {\bf At}\subseteq {\bf El}\bold S_c{\sf Nr}_n\CA_{\omega}\cap {\bf At}=
{\sf LCA}_n\subseteq {\bf El} {\sf  SRCA}_{n}$ 
We give the details. Fix $2<n<\omega$, and let $\bold K$ be any elementary class between $\Nr_n\CA_{\omega}$ and ${\sf RCA}_n$. For $3\leq n,i<\omega$, with $n-1\leq i, {\C}_{n,i}$ denotes
the finite ${\sf CA}_n$ associated with the cylindric atom structure as defined on \cite[p. 95]{HMT2}.
Then by \cite[Theorem 3.2.79]{HMT2}
for $3\leq n$, and $j<\omega$,
$\Rd_3{\C}_{n,n+j}$ can be neatly embedded in a
${\sf CA}_{3+j+1}$ \ \  (1).
By \cite[Theorem 3.2.84]{HMT2}), we have for every $j\in \omega$,
there is an $3\leq n$ such that $\Rd_{df}\Rd_{3}{\C}_{n,n+j}$
is a non--representable ${\sf Df}_3$ \ \  (2).
Now suppose that $m\in \omega$. By (2),
there is a $j\in \omega\sim 3$ so that $\Rd_{df}\Rd_3{\C}_{j,j+m+n-4}$
is  not a representable ${\sf Df}_3$. 

By (1) we have
$\Rd_3{\C}_{j,j+m+n-4}\subseteq \Nrr_3{\B_m}$, for some
${\B_m}\in {\sf CA}_{n+m}$. We can assume that $\mathfrak{Rd}_3{\C}_{j,j+m+n-4}$ generates $\B_m$, so that $\B_m$ is finite.
Put ${\A}_m=\Nrr_n\B_m$,
then $\A_m$ is finite, too, and $\mathfrak{Rd}_{df}{\A}_m$ is not representable,
{\it a fortiori} $\A_m\notin \RCA_{n}$. Therefore $\A_m\notin {\bf El}\Nr_n{\sf CA}_{\omega}$.
Let $\C_m$ be an algebra similar to ${\sf CA}_{\omega}$'s such that $\B_m=\mathfrak{Rd}_{n+m}\C_m$.
Then $\A_m=\Nrr_n\C_m$.  (Note that $\C_m$ cannot belong to $\CA_{\omega}$ for else $\A_m$ will be representable).
If $F$ is a non--trivial ultrafilter on $\omega,$ we have 
$\Pi_{m\in \omega}\A_m/F=\Pi_{m\in \omega}(\Nrr_n\C_m)/F=\Nrr_n(\Pi_{m\in \omega}\C_m/F).$ 
But  $\Pi_{m\in \omega}\C_m/F\in {\sf CA}_{\omega},$
we conclude that  ${\sf CA}_n\sim {\bf El}\bold K$ is not closed under ultraproducts, 
because $\A_m\notin {\sf RCA}_n\supseteq {\bf El}\Nr_n\CA_{\omega}$ and 
$\Pi_{m\in \omega}\A_m/F\in \Nr_n\CA_{\omega}\subseteq \bf El K$.

%We leave the proof of item (3) to the reader (Hint: witness the proof of item (1) of Theorem \ref{square}.)  
%We prove item (3). The proof highly resembles that of item (1) of Theorem \ref{square}.
We prove the last part of item (3): 
Let $\A\in \RCA_n$. Then $\A\cong \B$, $\B\in {\sf Gs}_n$ with top element $V$ say. Let $\C$ be the full ${\sf Gs}_n$ with top element $V$
(and universe $\wp(V)$). Then $\C\in \Nr_n\CA_{\omega}\cap \bf At$
and $\B\subseteq \C$. Thus $\A\in \bold S(\Nr_n\CA_{\omega}\cap \bf At)$.
For classes in the last item one has another option; one can take canonical extensions instead of generalized full set algebras
upon  observing that $\A\in \RCA_n\iff \A^+\in {\sf CRCA}_n$.

%All the rest follow from (previous) proofs used in Theorem \ref{iii}.
\end{proof}

Using the notation and proof of \cite[Theorem 3.7.4]{HHbook2} dealing with 
inclusions and first order definability of atom structures, 
we get:
\begin{theorem} 
Let $2<n<\omega$.  Then the following hold with elementary classes of atom stuctures or algebras  underlined:
\begin{enumerate}

\item  $\At\Nr_n\CA_{\omega}\nsubseteq {\sf CRAS}_n$ and 
$\sf At\Nr_n\CA_{\omega}\subseteq {\sf At}\bold S_d\Nr_n\CA_{\omega}\subseteq {\sf At}\bold S_c\Nr_n\CA_{\omega}\subsetneq 
\underline{{\sf At}{\bf El}\Sc\Nr_n\CA_{\omega}}=\underline{{\bf El}{\sf At}\bold S_c\Nr_n\CA_{\omega}}=\underline {\sf LCA}_n$.
 
\item  ${\sf CRAS_n}\subsetneq {\sf At}\bold S_c \Nr_n\CA_{\omega}\subsetneq \underline{{\sf At}{\bf El}\Sc\Nr_n\CA_{\omega}}=\underline{{\sf LCAS}}_n\subsetneq 
{\sf SRAS}_n\subsetneq \underline{{\sf WRAS}}_n$.

\item ${\sf At}{\sf Nr}_n\CA_{\omega}\nsubseteq {\sf CRAS}_n$. 

\end{enumerate}
\end{theorem}

We start by a neat embedding theorem, a $NET$ for short, 
formulated 
for $\TCA$s and $\TeCA$s, lifting Henkin's famous neat embedding theorem to the topological and temporal context, 
respectively.%. Recall that neat reducts for $\TeCA$s are defined exactly like for $\CA$s, by noting that 
%$G$ and $H$ do not perturb 
%dimension sets, that is if $\A\in \TeCA_{\alpha}$ and $a\in \A$, then $\Delta(a)=\Delta(G(a))=\Delta(H(a))$.

\begin{theorem}\label{neatt} 
An algebra $\A\in \TCA_{\alpha}$ is representable if and only if 
$\A\in \bold S\Nr_{\alpha}\TCA_{\alpha+\omega}$. 
\end{theorem}
\begin{proof}
First for any pair of ordinals $\alpha<\beta$, 
$S\Nr_{\alpha}\sf TCA_{\beta}$ is a variety is exactly like the $\CA$ case. 
To show that ${\sf RTCA}_{\alpha}\subseteq S\Nr_{\alpha}\TCA_{\alpha+\omega}$,
it suffices to consider algebras in ${\sf TWs}_{\alpha}$ the weak set algebras as defined in \cite{HMT2} whose base is endoe wit an Alexandrov topology.., 
since $S\Nr_{\alpha}\TCA_{\alpha+\omega}$ is closed
under ${\bf SP}$.  Let  $\A\in \sf TWs_{\alpha}$ and assume that $\A$ has top element 
$^{\alpha}U^{(p)}$. Let $\beta=\alpha+\omega$ and let $p^*\in {}^{\beta}U$ be a fixed sequence
such that $p^*\upharpoonright \alpha=p$. Let $\C$ be the $\TCA_{\beta}$ with top element
$^{\beta}U^{(p^*)}$; cylindrifiers and diagonal elements  are defined the usual way and the interior operators induced
by the topology on $U$. Define $\psi: \A\to \C$
via $X\mapsto \{s\in {}^{\beta}U^{(p^*)}: s\upharpoonright \alpha\in X\}.$
Then $\psi$ is a homomorphism, further it is injective, and as easily
checked, $\psi$ is a neat embedding that is  
$\psi(\A)\subseteq \Nr_{\alpha}\C$.
Maybe the hardest part is to show that if $\A\in S\Nr_{\alpha}\sf TCA_{\alpha+\omega}$
then it is representable. But this follows from the fact that we can assume that $\A\subseteq \Nr_{\alpha}\B$, where
$\B\in \sf TDc_{\alpha+\omega}$. By the representability of dimension complemented algebras,
baring in mind that a neat reduct of a representable algebra is 
representable, we get the required result. 
\end{proof}
Unless otherwise indicated $\alpha$ is an arbitrary ordinal 
and we sometimes writ $n$ instead of $\alpha$ if the last is finite.

\subsection{Non-finite axiomatizability and other complexity issues for the variety $\sf RTCA_{\alpha}$}

Now we prove quite intricate and sharp non-finite axiomatizability results using their cylindric version. 
We address both finite and infinite dimensions.  But we start with a very simple 
fact that allows us to recursively associate with every $\CA$ of any dimension both a $\TCA$ and a $\TeCA$ 
of the same dimension, such that the last two algebras are representable if and only if the original 
$\CA$ is. This mechanical procedure will be the  main technique we use 
to obtain negative results for both $\TCA$s and $\TeCA$s 
by bouncing them back 
to their cylindric  counterpart.
\begin{definition}\label{expand} Let $\A\in \CA_{\alpha}$. 
Then the algebra $\A^{\sf top}\in \TCA_{\alpha}$ 
is a  {\it topologizing of  $\A$} if 
$\Rd_{ca}\A^{\sf top}=\A$. The {\it discrete topologizing} of $\A$ is the $\sf TCA_{\alpha}$ obtained 
from $\A$ by expanding $\A$ with $\alpha$ many identity 
operators.
\end{definition}
Oberseve that if $\A$ is a generalized set algebra, then this can be done by giving all of  its subbasis the descrete Alexandrov topology.
\begin{theorem} The discrete topologizing of $\A\in \CA_{\alpha}$ is unique up to isomorphism. 
Furthermore, if $\A^{\sf top}$ is the discrete topologizing of $\A$, 
then $\A$ is representable if and only if $\A^{\sf top}$ 
is representable.
\end{theorem}
\begin{proof} The first part is trivial. 
The second part is also very easy. 
If $\A^{\sf top}$ is representable then obviously $\A=\Rd_{ca}\A^{\sf top}$ is representable. 
For the last part if $\A$ is representable with base
$U$ then $\A^{\sf top}$ have the same universe of $\A$, 
hence it is representable by endowing $U$ with 
the discrete topology, 
which induces the identity interior operators.
\end{proof}
For a $\sf BAO$,  $\A$, say,  and $a\in \A$, we write $\Rl_a\A$ for the algebra with universe $\{x\in \A: x\leq a\}$, top element $a$  and operations relativized to $\A$. 
If $\A\in \CA_n$; it is not always the case that $\Rl_a\A$ is a $\CA_n$, too.
We show that for any ordinal $\alpha>2$, for any  $r\in \omega$,
and for any $k\geq 1$, there exists
$\B^r\in \bold S\Nr_{\alpha}\TeCA_{\alpha+k}\sim S\Nr_{\alpha}\TeCA_{\alpha+k+1}$ such
that $\Pi_{r/U}\B^r\in \sf TeRCA_{\alpha}$, for any non-principal ultrafilter on $\omega$.
We will use quite sophisticated constructions of Hirsch and Hodkinson for relation and cylindric algebras reported 
in \cite{HHbook}.
Assume that $3\leq m\leq n<\omega$. For $r\in \omega$,  let $\C_r=\Ca(H_m^{n+1}(\A(n,r),  \omega))$
as defined in \cite[definition 15.3]{HHbook}.
We denote $\C_r$ by $\C(m,n,r)$.
Then the following hold:
\begin{lemma}\label{2.12}
\begin{enumarab}
\item For any $r\in \omega$ and $3\leq m\leq n<\omega$, we
have $\C(m,n,r)\in \Nr_m{\sf CA}_n,$ $\C(m,n,r)\notin S\Nr_m{\sf CA_{n+1}}$
and $\Pi_{r/U}\C(m,n,r)\in {\sf RCA}_m.$ Furthermore, for any $k\in \omega$, 
$\C(m, m+k, r)\cong \Nrr_m\C(m+k, m+k, r).$

\item  If $3\leq m<n$, $k\geq 1$ is finite,  and $r\in \omega$, there exists $x_n\in \C(n,n+k,r)$
such that $\C(m,m+k,r)\cong \Rl_{x_n}\C(n, n+k, r)$ and ${\sf c}_ix_n\cdot {\sf c}_jx_n=x_n$
for all $i,j<m$.

\end{enumarab}
\end{lemma}
\begin{proof}
1. Assume that $3\leq m\leq n<\omega$, and let
$$\mathfrak{C}(m,n,r)=\Ca(H_m^{n+1}(\A(n,r),  \omega)),$$
be as defined in \cite[Definition 15.4]{HHbook}.
Here $\A(n,r)$ is a finite Monk-like relation algebra \cite[Definition 15.2]{HHbook}
which has an $n+1$-wide $m$-dimensional hyperbasis $H_m^{n+1}(\A(n,r), \omega)$
consisting of all $n+1$-wide $m$-dimensional  wide $\omega$ hypernetworks \cite[Definition 12.21]{HHbook}.
For any $r$ and $3\leq m\leq n<\omega$, we have $\mathfrak{C}(m,n,r)\in \Nr_m{\sf CA}_n$.
Indeed, let $H=H_n^{n+1}(\A(n,r), \omega)$. Then $H$  is an $n+1$-wide $n$ dimensional $\omega$ hyperbasis,
so $\Ca H\in {\sf CA}_n.$ But, using the notation in \cite[Definition 12.21 (5)]{HHbook},
we have  $H_m^{n+1}(\A(n,r),\omega)=H|_m^{n+1}$.
Thus
$$\mathfrak{C}(m,n,r)=\Ca(H_m^{n+1}(\A(n,r), \omega))=\Ca(H|_m^{n+1})\cong \Nrr_m\Ca H.$$
The second part is proved in \cite[Corollary 15.10]{HHbook},
and the third in \cite[exercise 2, p. 484]{HHbook}.

2. Let $3\leq m<n$. Take $$x_n=\{f:{}^{\leq n+k+1}n\to \At\A(n+k, r)\cup \omega:  m\leq j<n\to \exists i<m, f(i,j)=\Id\}.$$
Then $x_n\in C(n,n+k,r)$ and ${\sf c}_ix_n\cdot {\sf c}_jx_n=x_n$ for distinct $i, j<m$.
Furthermore
\[{I_n:\C}(m,m+k,r)\cong \Rl_{x_n}\Rd_m {\C}(n,n+k, r),\]
via the map, defined for $S\subseteq H_m^{m+k+1}(\A(m+k,r), \omega)),$ by
$$I_n(S)=\{f:{}^{\leq n+k+1}n\to \At\A(n+k, r)\cup \omega: f\upharpoonright {}^{\leq m+k+1}m\in S,$$
$$\forall j(m\leq j<n\to  \exists i<m,  f(i,j)=\Id)\}.$$

\end{proof}

The essential argument used in the next 
proof is basically a lifting argument 
initiated by Monk \cite[Theorem 3.2.67]{HMT2}.

\begin{theorem}\label{negaxiom1} Let $\alpha>2$ be an  ordinal. Then for any $r\in \omega$, for any
finite $k\geq 1$, for any $l\geq k+1$ (possibly infinite),
there exist $\B^{r}\in \bold S\Nr_{\alpha}\sf TCA_{\alpha+k}$, $\Rd_{ca}\B\notin \sim S\Nr_{\alpha}\TCA_{\alpha+k+1}$ such
$\Pi_{r\in \omega}\B^r\in S\Nr_{\alpha}\sf TCA_{\alpha+l}$.
Also $\sf RTCA_{\alpha}$ cannot be axiomatized with a set of universal formulas 
having only finitely many variables. Same holds for $\sf TCA$s.
\end{theorem}
\begin{proof} 
We use the algebras $\C(m,n,r)$ in Theorem \ref{2.12} in the signature of $\sf TCA_m$, by 
{\it static temporalization} for both algebras, the $\C(m,m+k, r)$ and its dilation  
by defining $G=H=Id$ 
as the identity function, $T=\{t\}$ and $<$ be the empty set.
Using the same notation for the expanded algebras, we still obviously  have $\C(m, m+k, r)\cong \Nrr_m\C(m+k, m+k, r)$ 
for any $k\in \omega$.
Fix $r\in \omega$.
Let $I=\{\Gamma: \Gamma\subseteq \alpha,  |\Gamma|<\omega\}$.
For each $\Gamma\in I$, let $M_{\Gamma}=\{\Delta\in I: \Gamma\subseteq \Delta\}$,
and let $F$ be an ultrafilter on $I$ such that $\forall\Gamma\in I,\; M_{\Gamma}\in F$.
For each $\Gamma\in I$, let $\rho_{\Gamma}$
be a one to one function from $|\Gamma|$ onto $\Gamma.$
Let ${\C}_{\Gamma}^r$ be an algebra similar to $\TeCA_{\alpha}$ such that
\[\Rd^{\rho_\Gamma}{\C}_{\Gamma}^r={\C}(|\Gamma|, |\Gamma|+k,r).\]
Let
\[\B^r=\Pi_{\Gamma/F\in I}\C_{\Gamma}^r.\]
Then it can be proved  
\begin{enumerate}
\item\label{en:1} $\B^r\in \bold S\Nr_\alpha \sf TCA_{\alpha+k},$ 
\item\label{en:2} $\Rd_{ca}\B^r\not\in \bold S\Nr_\alpha\CA_{\alpha+k+1},$
\item\label{en:3} $\Pi_{r/U}\B^r\in \sf RTCA_{\alpha}.$
\end{enumerate}
For the first part, for each $\Gamma\in I$ we know that $\C(|\Gamma|+k, |\Gamma|+k, r) \in\K_{|\Gamma|+k}$ and
$\Nrr_{|\Gamma|}\C(|\Gamma|+k, |\Gamma|+k, r)\cong\C(|\Gamma|, |\Gamma|+k, r)$.
Let $\sigma_{\Gamma}$ be a one to one function
 $(|\Gamma|+k)\rightarrow(\alpha+k)$ such that $\rho_{\Gamma}\subseteq \sigma_{\Gamma}$
and $\sigma_{\Gamma}(|\Gamma|+i)=\alpha+i$ for every $i<k$. Let $\A_{\Gamma}$ be an algebra similar to a
$\CA_{\alpha+k}$ such that
$\Rd^{\sigma_\Gamma}\A_{\Gamma}=\C(|\Gamma|+k, |\Gamma|+k, r)$.
We claim that
$\Pi_{\Gamma/F}\A_{\Gamma}\in \TeCA_{\alpha+k}$.
For this it suffices to prove that each of the defining axioms for $\TeCA_{\alpha+k}$
hold for $\Pi_{\Gamma/F}\A_\Gamma$.
Let $\sigma=\tau$ be one of the defining equations for $\TeCA_{\alpha+k}$,
and we assume to simplify notation that
the number of dimension variables is one. Let $i\in \alpha+k$, we must prove
that $\Pi_{\Gamma/F}\A_\Gamma\models \sigma(i)=\tau(i)$.  If $i\in\rng(\rho_\Gamma)$,
say $i=\rho_\Gamma(i_0)$,
then $\Rd^{\rho_\Gamma}\A_\Gamma\models \sigma(i_0)=\tau(i_0)$,
since $\Rd^{\rho_\Gamma}\A_\Gamma\in\CA_{|\Gamma|+k}$,
so $\A_\Gamma\models\sigma(i)=\tau(i)$.
Hence $\set{\Gamma\in I:\A_\Gamma\models\sigma(i)=\tau(i)}\supseteq\set{\Gamma\in I: i\in\rng(\rho_\Gamma)}\in F$,
hence $\Pi_{\Gamma/F}\A_\Gamma\models\sigma(i)=\tau(i)$.
Thus, as claimed, we have $\Pi_{\Gamma/F}\A_\Gamma\in\CA_{\alpha+k}$.
We prove that $\B^r\subseteq \Nr_\alpha\Pi_{\Gamma/F}\A_\Gamma$.  Recall that $\B^r=\Pi_{\Gamma/F}\C^r_\Gamma$ and note
that $\C^r_{\Gamma}\subseteq A_{\Gamma}$
(the universe of $\C^r_\Gamma$ is $C(|\Gamma|, |\Gamma|+k, r)$, the universe of $\A_\Gamma$ is $C(|\Gamma|+k, |\Gamma|+k, r)$).
So, for each $\Gamma\in I$,
\begin{align*}
\Rd^{\rho_{\Gamma}}\C_{\Gamma}^r&=\C((|\Gamma|, |\Gamma|+k, r)\\
&\cong\Nrr_{|\Gamma|}\C(|\Gamma|+k, |\Gamma|+k, r)\\
&=\Nrr_{|\Gamma|}\Rd^{\sigma_{\Gamma}}\A_{\Gamma}\\
&=\Rd^{\sigma_\Gamma}\Nrr_\Gamma\A_\Gamma\\
&=\Rd^{\rho_\Gamma}\Nrr_\Gamma\A_\Gamma
\end{align*}
%$\Rd^{\rho_\Gamma}\A_\Gamma \in \K_{|\Gamma|}$, for each $\Gamma\in I$  then $\Pi_{\Gamma/F}\A_\Gamma\in \K_\alpha$.
Thus (using a standard Los argument) we have:
$\Pi_{\Gamma/F}\C^r_\Gamma\cong\Pi_{\Gamma/F}\Nrr_\Gamma\A_\Gamma=\Nrr_\alpha\Pi_{\Gamma/F}\A_\Gamma$,
proving \eqref{en:1}.
The above isomorphism $\cong$ follows from the following reasoning.
Let $\B_{\Gamma}= \Nrr_{\Gamma}\A_{\Gamma}$. Then universe of the $\Pi_{\Gamma/F}\C^r_\Gamma$ is
identical to  that of $\Pi_{\Gamma/F}\Rd^{\rho_\Gamma}\C^r_\Gamma$
which is identical to the universe of $\Pi_{\Gamma/F}\B_\Gamma$.
Each operator $o$ of $\CA_{\alpha}$ is the same
for both ultraproducts because $\set{\Gamma\in I:\dim(o)\subseteq\rng(\rho_\Gamma)} \in F$.

Now we prove \eqref{en:2}.
For this assume, seeking a contradiction, that $\Rd_{ca}\B^r\in S\Nr_{\alpha}\CA_{\alpha+k+1}$,
$\B^r\subseteq \Nrr_{\alpha}\C$, where  $\C\in \CA_{\alpha+k+1}$.
Let $3\leq m<\omega$ and  $\lambda:m+k+1\rightarrow \alpha +k+1$ be the function defined by $\lambda(i)=i$ for $i<m$
and $\lambda(m+i)=\alpha+i$ for $i<k+1$.
Then $\Rd^\lambda(\C)\in \CA_{m+k+1}$ and $\Rd_m\B^r\subseteq \Nrr_m\Rd^\lambda(\C)$.
For each $\Gamma\in I$,\/  let $I_{|\Gamma|}$ be an isomorphism
\[{\C}(m,m+k,r)\cong \Rl_{x_{|\Gamma|}}\Rd_m {\C}(|\Gamma|, |\Gamma+k|,r).\]
Exists by item (2) of lemma \ref{2.12}.
Let $x=(x_{|\Gamma|}:\Gamma)/F$ and let $\iota( b)=(I_{|\Gamma|}b: \Gamma)/F$ for  $b\in \C(m,m+k,r)$.
Then $\iota$ is an isomorphism from $\C(m, m+k,r)$ into $\Rl_x\Rd_m\B^r$.
Then by \cite[theorem~2.6.38]{HMT1} we have $\Rl_x\Rd_{m}\B^r\in S\Nr_m\CA_{m+k+1}$.
It follows that  $\C (m,m+k,r)\in S\Nr_{m}\CA_{m+k+1}$ which is a contradiction and we are done.

Now we prove the third part of the theorem, putting the superscript $r$ to use.
The first two items are as before. 
Now we prove \eqref{en:3} putting the superscript $r$ to use.
Recall that $\B^r=\Pi_{\Gamma/F}\C^r_\Gamma$, where $\C^r_\Gamma$ has the type of $\TCA_{\alpha}$
and $\Rd^{\rho_\Gamma}\C^r_\Gamma=\C(|\Gamma|, |\Gamma|+k, r)$.
We know 
from item (1) of lemma \ref{2.12} that $\Pi_{r/U}\Rd^{\rho_\Gamma}\C^r_\Gamma=
\Pi_{r/U}\C(|\Gamma|, |\Gamma|+k, r) \subseteq \Nrr_{|\Gamma|}\A_\Gamma$,
for some $\A_\Gamma\in\TeCA_{|\Gamma|+\omega}$.
Let $\lambda_\Gamma:|\Gamma|+k+1\rightarrow\alpha+k+1$
extend $\rho_\Gamma:|\Gamma|\rightarrow \Gamma \; (\subseteq\alpha)$ and satisfy
\[\lambda_\Gamma(|\Gamma|+i)=\alpha+i\]
for $i<k+1$.  Let $k+1\leq l\leq \omega$.
Let $\F_\Gamma$ be a $\TCA_{\alpha+l}$ type algebra such that $\Rd^{\lambda_\Gamma}\F_\Gamma=\Rd_l\A_\Gamma$.
As before, $\Pi_{\Gamma/F}\F_\Gamma\in\TCA_{\alpha+l}$.  And
\begin{align*}
\Pi_{r/U}\B^r&=\Pi_{r/U}\Pi_{\Gamma/F}\C^r_\Gamma\\
&\cong \Pi_{\Gamma/F}\Pi_{r/U}\C^r_\Gamma\\
&\subseteq \Pi_{\Gamma/F}\Nrr_{|\Gamma|}\A_\Gamma\\
&=\Pi_{\Gamma/F}\Nrr_{|\Gamma|}\Rd^{\lambda_\Gamma}\F_\Gamma\\
&\subseteq\Nrr_\alpha\Pi_{\Gamma/F}\F_\Gamma,
\end{align*}
Hence, we get by the neat embedding theorem that  $\Pi_{r/U}\B^r\in S\Nr_{\alpha}\TCA_{\alpha+l}$
and we are done.

Let $k$ be as before; $k$ is finite and $>0$ and let $l$ be as in the hypothesis of the theorem,
that is, $l\geq k+1$, and we can assume without loss that $l\leq \omega$.
Recall that $\B^r=\Pi_{\Gamma/F}\C^r_\Gamma$, where $\C^r_\Gamma$ has the type of $\CA_{\alpha}$
and $\Rd^{\rho_\Gamma}\C^r_\Gamma=\C(|\Gamma|, |\Gamma|+k, r)$.
We know (this is the main novelty here)
from item (2) that $\Pi_{r/U}\Rd^{\rho_\Gamma}\C^r_\Gamma=\Pi_{r/U}\C(|\Gamma|, |\Gamma|+k, r) \subseteq \Nrr_{|\Gamma|}\A_\Gamma$,
for some $\A_\Gamma\in\CA_{|\Gamma|+\omega}$.
If $\rho:\omega\to \alpha$ is an injection, then $\rho$ extends recursively
to a function $\rho^+$ from $\CA_{\omega}$ terms to $\CA_{\alpha}$ terms.
On variables $\rho^+(v_k)=v_k$, and for compound terms
like ${\sf c}_k\tau$, where $\tau$ is a $\CA_{\omega}$ term, and $k<\omega$, $\rho^+({\sf c}_k\tau)={\sf c}_{\rho(k)}\rho^+(\tau)$.
For an equation $e$ of the form $\sigma=\tau$ in the language of $\CA_{\omega}$, $\rho^+(e)$ is
the equation $\rho^+(\tau)=\rho^+(\sigma)$ in the language
of $\CA_{\alpha}$. This last equation, namely, $\rho^+(e)$ is called an $\alpha$ instance of $e$
obtained by applying the injection $\rho$.
Let $k\geq 1$ and $l\geq k+1$. Assume for contradiction that $\bold S\Nr_{\alpha}\TCA_{\alpha+l}$ is axiomatizable
by a finite schema over  $\bold S\Nr_{\alpha}\TCA_{\alpha+k}$.
We can assume that there is only one equation, such that all its $\alpha$ instances,  axiomatize  $S\Nr_{\alpha}\TCA_{\alpha+l}$ over
$S\Nr_{\alpha}\TCA_{\alpha+k}$.
So let $\sigma$ be such an  equation in the signature of $\TCA_{\omega}$ and let $E$ be its $\alpha$ instances; so that
 for any $\A\in \bold S\Nr_{\alpha}\TCA_{\alpha+k}$ we have $\A\in \bold S\Nr_{\alpha}\TCA_{\alpha+l}$ $\iff$
$\A\models E$.  Then for all $r\in \omega$, there is an instance of $\sigma$,  $\sigma_r$ say,
such that $\B^r$ does not model $\sigma_r$.
$\sigma_r$ is obtained from $\sigma$ by some injective map $\mu_r:\omega\to \alpha$.
For $r\in \omega,$ let $v_r\in {}^{\alpha}\alpha$,
be an injection such that $\mu_r(i)=v_r(i)$ for $i\in {\sf ind}(\sigma_r)$, and
let $\A_{r}= \Rd^{v_r}\B^r$.
Now $\Pi_{r/U} \A_{r}\models \sigma$. But then
$$\{r\in \omega: \A_{r}\models \sigma\}=\{r\in \omega: \B^r\models \sigma_r\}\in U,$$
contradicting that
$\B^r$ does not model $\sigma_r$ for all $r\in \omega$.

\end{proof}

We note that all of Andr\'eka's complexity results and other results showing the resilient robust 
undecidblity of 'three variable first order logic (and more finitely many variables $n$ say)' expressed in classes between $\CA_3$ and ${\sf RCA}_3$ (between $\RCA_n$ and $\CA_n$)
proved in \cite{Andreka} and \cite {k, k2} respectively, 
with possibly no exception lift to the topological 
contexts by forming  discrete topologizing. 
Let $2<n<\omega$. We approach the modal version of $L_n$ without equality, namely, ${\bf S5}^n$. The corresponding class of modal algebras
is the variety ${\sf RDf}_n$ of {\it diagonal free $\RCA_n$s} \cite{HMT2, k2}. Let $\Rd_{df}$ denote `diagonal free reduct'. 
\begin{lemma}\label{dfb} Let $2<n<\omega$. If $\A\in \CA_n$ is such that $\Rd_{df}\A\in {\sf RDf}_n$,
and $\A$ is the smallest subalgebra of itself containing $J=\{x\in \A: \Delta x\neq n\}$ and closed under complementation, infinite intersection, cylindrifications and diagonal elements, 
then
$\A\in \RCA_n$. 
\end{lemma}
\begin{proof}  Easily follows from \cite[Lemma 5.1.50, Theorem 5.1.51]{HMT2}.  Assume that $\A\in \CA_n$, $\Rd_{df}\A$ is a set algebra (of dimension $n$) 
with base $U$,  and $R\subseteq U\times U$ are as in the hypothesis of \cite[Theorem 5.1.49]{HMT2}. 
Let $E=\{x\in A:  (\forall x, y\in {}^nU)(\forall i <n)(x_iR y_i\implies (x\in X\iff y\in X))\}$.
Then $\{x\in \A: \Delta x\neq n\}\subseteq E$ and $E\in \CA_n$ 
is closed under infinite intersections. The required follows.
\end{proof}
\begin{theorem}\label{negaxiom2} Let $2<n<\omega$.  
\begin{enumerate} 
\item  Any universal 
axiomatization of ${\sf RTCA}_{\alpha}$ 
must contain infinitely many variables and infinitely many diagonal 
constants.
\item It is undecidable to whether a 
finite  $\TCA_n$ is representable or not. 
In particular, for any such $n$,
the variety of representable algebras in the two cases cannot be axiomatized in $m$th 
order logic for any 
finite $m$. 
\item The class of topological Kripke frames of the form $\{\F\in \At\TCA_n: \Cm\F\in {\sf RTCA}_n\}$ is not elementary. 

\item Any axiomatization of any model logic beten ${\sf S5}^n$ and modal topological 
logics wih $n$ variables, viewed  as an $n$-dimensional multi-modal logic  has to contain genuinely 
second order formulas on Kripke frames, and furthermore cannot be axiomatized by first order formulas 
{\it a fortiori Sahlqvist},  
nor canonical ones. 
\item Though canonical, any equational aximatization of $\sf TRCA_n$ must contain infnitely 
many non canonical 
equations.
\end{enumerate}
\end{theorem}
\begin{proof} 
1. Let $2<n< \omega$, and $k\in \omega$. Then in \cite{Andreka} a non-
representable $\CA_{n}$ is constructed (by splitting an atom in a $\sf Cs_n$-a cylindric set algebra of dimension $n$- 
into $k+1$ parts)
such that its $k$ generated subalgebras, that is, subalgebras generated by $\leq k$ many elements, 
are representable.
One expands $\A$ with the interior identity operators, it remains
non-representable of course, but its $k$ generated subalgebras are representable in the expanded signature, 
the representation of the newly added identity $\sf S4$ modalities 
induced by the discrete topology on the base. 

2. 
In \cite{k2}, it is proved that it is undecidable to tell whether a finite frame
is a frame for $\L$,  and this gives
the non--finite axiomatizability result required as indicated in {\it op.cit}, 
and obviously implies undecidability.
The rest follows by transferring the required results holding for ${\bf S5}^n$ \cite{b, k2}
to $\L$ since ${\bf S5}^n$ is finitely axiomatizable 
over $\L$, and any axiomatization of ${\sf RDf}_n$ must contain infinitely many 
non-canonical equations. 
%We also have the following quite deep result obtained from the $\CA$ counterpart proved by Hodkinson by our 
%mechanical procedures:
We prove the second item for $\TCA_n$. 
The proof uses the main result in \cite{k2} 
using the recursive procedure of static 
topologizing.  That is if the problem is decidable for $\TCA_n$, then this implies its decidability for $\CA_n$ for any finite $n>2$.
Indeed take any finite $\A\in \CA_n$.
By discrete topologizing we obtain $\A^{\sf top}$, apply the available algorithm for finite $\TCA_n$s, 
the answer is the correct one for $\A$, 
since $\A$ is representable 
$\iff$ $\A^{\sf top}$ is. The last part  concerning non finite axiomatizability 
is proved for both relation and cylindric algebras 
(having an analogous undecidability result) in \cite{HHbook, k, k2}.
The idea is the existence of any such finite axiomatization in 
$m$th order logic for any positive $m$  
gives a decision procedure for telling whether a finite algebra is 
representable
or not.

3. Let $\bold L$ be the class of square frames for ${\bf S5}^n$.
Then $\L(\bold L)={\bf S5}^n$ \cite[p.192]{k}. But the class of frames $\F$ valid in $\L(\bold L)$ coincides 
with the class of  {\it strongly representable ${\sf Df}_n$ atom structures} which  
is {\it not elementary} as proved in \cite{b}. This gives the  required result for ${\bf S5}^n$. With Lemma \ref{dfb} at our disposal, 
a slightly different proof can be easily distilled from the construction addressing $\CA$s in \cite{HHbook2} or \cite{k2}. 
We adopt the construction in the 
former reference, using the Monk--like $\CA_n$s  
${\mathfrak M}(\Gamma)$, $\Gamma$ a graph, as defined in
\cite[Top of p.78]{HHbook2}. 
For a graph $\G$, let $\chi(\G)$ denote it chromatic number. 
Then it is proved in {\it op.cit} that 
for any graph $\Gamma$, ${\mathfrak M}(\Gamma)\in \RCA_n$ 
$\iff$ $\chi(\Gamma)=\infty$.
By Lemma \ref{dfb} and discretely topologizing ${\mathfrak M}(\Gamma)$, getting the expansion $\mathfrak{M}(\Gamma)^{\sf top}$, say, we have 
${\mathfrak M}(\Gamma)^{\sf top}\in {\sf RTCA}_n \iff \Rd_{df}{\mathfrak M}(\Gamma)\in {\sf RDf}_n\iff \chi(\Gamma)=\infty$,
because $\mathfrak{M}(\Gamma)$ 
is generated by the set $\{x\in {\mathfrak M}(\Gamma): \Delta x\neq n\}$ using infinite unions and 
$\A\in \RCA_n\iff \A^{\sf top}\in {\sf RTCA}_n$.
Now we adopt the argument in \cite{HHbook2}. Using Erdos' probabalistic graphs \cite{Erdos}, 
for each finite
$\kappa$, there is a finite graph $G_{\kappa}$ with
$\chi(G_{\kappa})>\kappa$ and with no cycles of length $<\kappa$. 
Let $\Gamma_{\kappa}$ be the disjoint union of the $G_{l}$ for
$l>\kappa$.  Then $\chi(\Gamma_{\kappa})=\infty$, and so
${\mathfrak{M}(\Gamma)_{\kappa}}^{\sf top}\in {\sf RTCA}_n$.
Now let $\Gamma$ be a non-principal ultraproduct
$\Pi_{D}\Gamma_{\kappa}$ for the $\Gamma_{\kappa}$s. For $\kappa<\omega$, let $\sigma_{\kappa}$ be a
first-order sentence of the signature of the graphs stating that
there are no cycles of length less than $\kappa$. Then
$\Gamma_{l}\models\sigma_{\kappa}$ for all $l\geq\kappa$. By
Lo\'{s}'s Theorem, $\Gamma\models\sigma_{\kappa}$ for all
$\kappa$. So $\Gamma$ has no cycles, and hence by $\chi(\Gamma)\leq 2$.
Thus $\mathfrak{Rd}_{df}\mathfrak{M}(\Gamma)$
is not representable.  
(Observe that the 
the term algebra $\Tm\At(\mathfrak{M}(\Gamma))$
is representable (as a $\CA_n$), 
because the class of weakly representable atom structures is elementary \cite[Theorem 2.84]{HHbook}.)

4.  The first part follows from that the class of strongly representable $\Df_n$ atom structures is not
elementary proved in the prevuos item and in \cite{b}. 
Sincee Sahlqvist formulas have first order correspondents, then ${\bf S5}^n$ is not Sahlqvist.  

5. From the construction in \cite{b} using discrete topologizing.
\end{proof}

\end{document}